\documentclass{amsart}

\usepackage{latexsym}
\usepackage{amsmath}
\usepackage{amssymb}
\usepackage{amsfonts}
\usepackage{amsxtra}
\usepackage[all, 2cell, knot]{xy} 
\UseAllTwocells 
\SilentMatrices
\usepackage{xspace}
\usepackage{amsmath}
\usepackage{amstext}
\usepackage{amsfonts}
\usepackage[mathscr]{euscript}
\usepackage{amscd}
\usepackage{hyperref}
\hypersetup{
    colorlinks=false,
    pdfborder={0 0 0}
}

\newcommand{\bbC}{\mathbb{C}}
\newcommand{\bbD}{\mathbb{D}}
\newcommand{\bbE}{\mathbb{E}}
\newcommand{\bbX}{\mathbb{X}}
\newcommand{\bbY}{\mathbb{Y}}
\newcommand{\bbComp}{\mathbb{C}\mbox{\rm omp}}

\newcommand{\Comp}{\mbox{\rm Comp}}

\newcommand{\calB}{\mathcal{B}}
\newcommand{\calC}{\mathcal{C}}
\newcommand{\calD}{\mathcal{D}}

\newcommand{\calJ}{\mathcal{J}}
\newcommand{\calU}{\mathcal{U}}
\newcommand{\calV}{\mathcal{V}}

\newcommand{\bfC}{\mathbf{C}}
\newcommand{\bfD}{\mathbf{D}}

\newcommand{\Hom}{\mbox{\rm Hom}}

\newcommand{\Bic}{\mbox{\bf Bic}}
\newcommand{\smBic}{\mbox{\bf\scriptsize Bic}}

\newcommand{\Bicat}{\mbox{\sf Bicat}}
\newcommand{\NHom}{\Bicat_{\mbox{\scriptsize\rm icon}}}
\newcommand{\Dbl}{\mbox{\bf Dbl}}
\newcommand{\W}{\{W\}}
\newcommand{\CW}{\bfC\W}

\newcommand{\DblCat}{\mbox{\sf DblCat}}

\newcommand{\ps}{\mbox{\scriptsize\rm ps}}

\newcommand{\st}{\mbox{\scriptsize\rm st}}

\newcommand{\WGDbl}{\mbox{\sf WGDbl}}
\newcommand{\smWGDbl}{\mbox{\sf\scriptsize WGDbl}}

\newcommand{\Cat}{\mbox{\sf Cat}}

\newcommand{\tiund}[1]{{\times}_{#1}}

\newcommand{\dop}{\Delta^{\mbox{\scriptsize\rm op}}}

\newcommand{\pro}[3]{#1\tiund{#2}\overset{#3}{\cdots}\tiund{#2}#1}

\newcommand{\JC}{\calJ_W}

\newcommand{\rw}{\rightarrow}

\newtheorem{thm}{Theorem}[section]
\newtheorem{cor}[thm]{Corollary}

\newtheorem{prop}[thm]{Proposition}

\newtheorem{lma}[thm]{Lemma}

\theoremstyle{definition} 
\newtheorem{dfn}[thm]{Definition}

\newtheorem{eg}[thm]{Example}

\newtheorem{rmks}[thm]{Remarks}
\newtheorem{rmk}[thm]{Remark}

\begin{document}
\title {The Weakly Globular Double Category of Fractions of a Category}

\author{Simona Paoli}
\address{Department of Mathematics, 
University of Leicester,
LE17RH, UK}
\email{sp424@le.ac.uk}

\author{Dorette Pronk}
\address{Department of Mathematics and Statistics, Dalhousie University, Halifax, NS, B3H 4R2, Canada}
\email{pronk@mathstat.dal.ca}

\keywords{double categories, strict 2-categories, bicategories, pseudo-functors, companions, conjoints, 
localizations, bicategories of fractions, double categories of fractions}

\begin{abstract}
This paper introduces the construction of a weakly globular double category of fractions for a category
and studies its universal properties. It shows that this double category is locally small and considers 
a couple of concrete examples.
\end{abstract}

\maketitle

\section{Introduction}
Localization of categories or bicategories is the process of freely adding inverses or pseudo inverses
to a class of arrows to turn them into isomorphisms or internal equivalences respectively.
This has turned out to be important in homotopy theory where one only considers maps up to homotopy.
In doing this we need to add inverses to those maps that are equivalences up to homotopy, but don't have a 
continuous inverse. And in geometry, it can be useful to view some well-known categories as localizations of 
simpler categories. For instance, when objects are defined using an atlas of local charts, 
we can define maps between these objects as (smooth) maps defined on charts, also called `atlas maps', but 
we need to add inverses to those atlas maps that correspond to the identity 
on the underlying objects and may not be invertible as atlas maps.
As a further example, the homotopy theory of orbifolds and stacks  
is best studied in terms of groupoid representations \cite{MP, Adem, Noohi},
but in order to obtain the correct notion of maps between orbifolds we need 
to invert the Morita equivalences between groupoids 
(which in general do not have inverse maps as smooth/continuous groupoid homomorphisms).

Localizing a category  can be done in various ways. 
The simplest way to construct a localization is to simply add the required inverse arrows to the category and
take the free category of paths modulo the old composition structure.
However, this has the draw-back that the equivalence relation is rather unwieldy 
and one does not know a priori how long the paths need to be in order
to have a representative for all arrows in the localized category.
The first way to deal with this is to consider classes of arrows that satisfy the conditions
of a left or right category of fractions, as defined by Gabriel and Zisman \cite{GZ}, 
or its bicategorical generalization, the bicategory of fractions \cite{Pr-comp}.
In this situation we only need to consider spans (or cospans) of arrows where the left-hand (right-hand) side of the 
span (cospan) is an arrow that needed to be inverted. The equivalence relation becomes a lot easier to describe as well in this case.
However, there is still an issue left: we do not have a guarantee that the resulting 
category of fractions will have small hom-sets, unless the original category had only a set of objects.

In practice, we see that besides the canonical category of fractions 
there are often other concrete models of the localized category that clearly do have small hom-sets, 
and this guarantees then that the category of fractions has small hom-sets as well. 
For instance, we can view the bicategory of orbifolds as the bicategory of fractions of the category of orbifold groupoids
with respect to essential equivalences.
However, orbifolds  can also be viewed as  \'etendues (a special kind of toposes) with a proper diagonal, 
and orbifold morphisms correspond then to geometric morphisms between \'etendues. For more details on  
this, see \cite{Pr-comp}.  
Another example is that of 2-groups or crossed modules, where the arrows in the homotopy category 
with respect to weak equivalences can be described in terms of so-called butterflies \cite{Noohi}.

As a way to systematically deal with the size-problem Quillen model structures have been 
introduced \cite{Quillen}. However, one does not always have the extra structure of fibrations and cofibrations
to be able to use this theory. Also, generalizing the theory of Quillen model structures to higher categories
is highly non-trivial. An attempt that provides the fibrant half of a Quillen model structure for a bicategory was
given in \cite{PW}. To obtain a simplicially enriched localization, the hammock localization \cite{hammock} was developed by Dwyer and Kan.

In this paper we want to take a different approach, and use an alternate model for a weak 2-categorical 
structure to study localizations in terms of categories of fractions.
In \cite{PP1} we introduced the notion of weakly globular double category
to obtain a new model for weak 2-categories which is a subcategory of the category of double categories.
The main idea behind the definition of weakly globular double category is that we want to weaken 
the globularity condition on 2-categories rather than the unit or associativity conditions.
The globularity condition says that in each dimension the collection of cells is {\em discrete}, i.e., a set or class.
One way to describe a double category is as an internal category in the category of categories, i.e. a diagram
\begin{equation}\label{intcat}
\xymatrix{\bbX_1\times_{\bbX_0}\bbX_1 \ar[r]^-m & \bbX_1\ar@<1ex>[r]^s\ar@<-1ex>[r]_t & \ar[l]|u\bbX_0.}
\end{equation}
So a double category $\bbX$ has a category (rather than a set)  $\bbX_i$ in each dimension.
For weakly globular double categories, we require that $\bbX_0$ be a posetal groupoid (or, equivalence relation)
rather than a set. This is what we call our weak globularity condition. This comes then with additional requirements on $\bbX_1$,
which are spelled out in Definition \ref{wklglb} below. 
In \cite{PP1} we obtained a biequivalence of 2-categories 
$$
\Dbl\colon \NHom\leftrightarrows\WGDbl_{\ps,v}:\Bic.
$$
Here, $\NHom$ is the 2-category of bicategories, homomorphisms and icons, whereas $\WGDbl_{\ps,v}$
is the 2-category of weakly globular double categories, pseudo-functors and vertical transformations.

In this paper we will construct $\bfC\{W\}$, the weakly globular double category of fractions for a category $\bfC$ with respect to a
class $W$ of arrows. We will consider the category $\bfC$ here as a horizontal double category, which means that 
in the diagram (\ref{intcat}), $\bbX_0$ is the discrete category on the objects of $\bfC$ and $\bbX_1$ is the discrete category on the arrows of
$\bfC$.  We denote this double category by $H\bfC$.
The weakly globular double category $\CW$ has sets of horizontal arrows, vertical arrows and cells, so there are no
size-issues here. This can be explained by the fact that this construction works by adding additional objects to the category that represent the
arrows in $W$ and the fact that in some sense the new inverses are being added in the vertical direction.
For instance, for the category of atlases and atlas maps, the new objects correspond to atlas refinements 
$\calU'\stackrel{\nu}{\rightarrow}\calU$.
The horizontal arrows are given by atlas maps between the domains of these refinements and there is a vertical map between 
two refinements of the same atlas if they have a common refinement.  There is a double cell precisely when two atlas maps 
become the same when restricted to some common refinements for the domain and codomain.

Aside from this size property, what makes the weakly globular double category of fractions interesting is its universal property.
For the category of fractions $\bfC[W^{-1}]$, composition with the inclusion functor $I_W\colon \bfC\rightarrow \bfC[W^{-1}]$ 
gives an equivalence of categories,
$$
\Cat(\bfC[W^{-1}],\bfD)\simeq\Cat_W(\bfC,\bfD),
$$
where $\Cat_W(\bfC,\bfD)$ is the category of functors and natural transformations that send arrows in $W$ 
to isomorphisms. This determines the category of fractions up to an equivalence of categories.

For the bicategory of fractions $\bfC(W^{-1})$, we want an equivalence of categories or 2-categories
between appropriate hom categories or hom 2-categories.
Note that the hom-category $\mbox{Bicat}\,(\calB,\calC)$ 
of two bicategories $\calB$ and $\calC$ can be viewed as a 2-category in various ways, depending on the type of morphisms and 
transformations one is interested in.
In particular, one may take homomorphisms of bicategories (which preserve units and composition both up to
coherent isomorphisms) as objects and (op)lax transformations as morphisms. When one considers these transformations 
as homomorphisms into the cylinder bicategory on $\calC$, this leads the way to the next level of cells: the modifications 
which can be viewed as transformations between the first-level transformations when considered as homomorphisms. 
One can also restrict oneself to pseudo natural transformations, or to icons. And one can restrict the objects to only consider 
strict functors or relax them to consider lax or oplax functors.

For a bicategory $\calC$ with a class of arrows $W$ satisfying the conditions to form a bicategory of fractions,
the universal property of the bicategory of fractions $\bfC(W^{-1})$ is stated as follows.

\begin{thm}
Composition with the homomorphism $J_W\colon\mathbf{C}\rightarrow\mathbf{C}(W^{-1})$
induces an equivalence of bicategories
$$
\mbox{Bicat}(\mathbf{C}(W^{-1}))\simeq\mbox{Bicat}_W(\mathbf{C},\mathcal{D}),
$$
for any bicategory $\mathcal{D}$.
\end{thm}

In \cite{Pr-comp} this universal property was given in terms of homomorphisms and oplax transformations, 
but it can be restricted to homomorphisms with pseudo transformations or to homomorphisms with icons. 
There is also an analogous result  for homomorphisms with lax transformations.

For weakly globular double categories $\bbC$ and $\bbD$, the hom object $\WGDbl(\bbC,\bbD)$ 
can be given the structure of a category in two ways; namely, by using horizontal or vertical transformations.
We show that $\CW$ has a universal property with respect to both structures.
The inclusion of weakly globular double categories, $\calJ_W\colon H\bfC\rightarrow\CW$, induces two equivalences of categories,
$$
\WGDbl_{\ps,v}(\CW,\bbD)\simeq\WGDbl_{\ps,v,W}(H\bfC,\bbD)
$$
and
$$
\quad \WGDbl_{\st,h}(\CW,\bbD)\simeq\WGDbl_{\st,h,W}(H\bfC,\bbD).
$$
Note that the first equivalence, also called the {\em vertical universal property}, 
is with respect to pseudo-functors, whereas the second universal property,
also called the {\em horizontal universal property}, is with respect to strict functors. 
The vertical universal property determines $\CW$ up to 
vertical equivalence and the horizontal universal property determines $\CW$ up to horizontal equivalence. 
The two notions of equivalence are generally 
unrelated, so it is important to have both. 

One of the main purposes of this paper is to establish what the objects and arrows 
of the categories $\WGDbl_{\ps,v,W}(H\bfC,\bbD)$ 
and $\WGDbl_{\st,h,W}(H\bfC,\bbD)$ are.
One way to do this is to consider the image of $J_W\colon\mathbf{C}\rightarrow\mathbf{C}(W^{-1})$ under the 2-functor $\Dbl$.
However, this functor is only defined on icons, not on any other transformations
(since it needs to have a 2-category of bicategories as its domain). So we need apply it to the universal property
expressed in terms of icons.
This translates then into a universal property for weakly globular double categories in 
terms of pseudo-functors and vertical transformations. Composition with $\Dbl(J_W)$ gives rise to
an equivalence of categories,
$$\mbox{WGDbl}_{\ps,v}({\bf Dbl}({\bf C}(W^{-1})) , \Dbl\,\calD ) \simeq \WGDbl_{\ps,v,W}(\Dbl(\bfC),\Dbl\,\calD ).$$
Note that in this equation, $\WGDbl_{\ps,v,W}(\Dbl(\bfC),\Dbl\calD)$ is really a short-hand for 
$$\Dbl(\Bicat_{\mbox{\scriptsize\rm icon},W}(\bfC,\calD)).$$

In order to find a description of  $$\mbox{WGDbl}_{\ps,v,W}(\Dbl\bfC,\Dbl\mathcal{D})\subseteq\WGDbl_{\ps,v}(\Dbl\bfC,\Dbl\,\calD)$$
we need to describe the special properties of horizontal arrows in weakly globular
double categories that correspond to internal equivalences in bicategories (since this notion is 
used to describe  the universal property of a bicategory of fractions).
To find this we first study for a bicategory $\mathcal B$  what type of horizontal arrows in the weakly globular 
double category $\Dbl({\mathcal B})$ correspond to quasi units in $\mathcal B$ (where we call an arrow 
$f\colon A\rightarrow A$ in a bicategory a quasi unit when $f\cong 1_A$). We show that these are precisely the 
horizontal arrows in $\Dbl({\mathcal B})$ that have a vertical companion as defined in \cite{GP-companions}.
Furthermore, we also show that for a weakly globular double category $\mathbb X$, horizontal arrows which have
vertical companions in $\mathbb X$ correspond to quasi units in $\Bic({\mathbb X})$.

Corresponding to the bicategorical notion of internal equivalence, we define the notion of 
{\em precompanion} for horizontal arrows in a weakly globular double category.
And we prove that the objects of $\mbox{WGDbl}_{\ps,v,W}(\Dbl\bfC,\Dbl\mathcal{D})$
are indeed the pseudo-functors that send the arrows in the image of  $W$ 
to precompanions. The arrows in this category are the so-called $W$-transformations that 
respect the precompanion structure in an appropriate way (see Definition \ref{W-trafo}). 
Since every weakly globular double category is vertically equivalent to one of the form $\Dbl(\calB)$ 
for a bicategory $\calB$, the notion of $W$-pseudo-functor and $W$-transformation can be extended 
to pseudo-functors with an arbitrary weakly globular double category as domain,
and the weakly globular double category $\Dbl(\bfC(W^{-1}))$ 
has the following (vertical) universal property.

\begin{thm}
Composition with the functor $\Dbl(J)\colon\Dbl(\bfC)\rightarrow\Dbl(\bfC(W^{-1}))$ induces 
an equivalence of categories
$$
\WGDbl_{\ps,v}(\Dbl(\bfC(W^{-1})),\bbD)\simeq\WGDbl_{\ps,v,W}(\Dbl(\bfC),\bbD),
$$
where the objects of $\WGDbl_{\ps,v,W}(\Dbl(\bfC),\bbD)$ are pseudo-functors that send arrows related to those in $W$ to 
precompanions.
\end{thm}

So this seems to be the right vertical universal property for a weakly globular double category of fractions.
However, this is as far as $\Dbl(J_W)\colon\Dbl(\bfC)\rightarrow \Dbl(\bfC(W^{-1}))$
can help us.
 This functor is strict, 
but it has no obvious universal property in terms of horizontal transformations. 
This is not surprising, since $\Dbl$ is part of an equivalence that involves only the 
pseudo-functors and vertical transformations.
So the functor $\Dbl(J_W)\colon\Dbl(\bfC)\rightarrow \Dbl(\bfC(W^{-1}))$  cannot 
be our prototypical example of a universal arrow into a weakly globular double category of fractions.
Now we could try to just adjust its codomain - as long as the new codomain is vertically equivalent
to  $\Dbl(\bfC(W^{-1}))$ it will have the correct vertical universal property.
However, as an example of a weakly globular double category with a class of horizontal arrows
$W$, $\Dbl(\bfC)$ does not look general enough (the arrows corresponding to those in $W$ occur in 
many different guises, namely as composites of paths that are not necessarily in $W$, 
and we do need to take all of them into account when constructing the category of fractions). 

Now note that there is  another embedding
of categories into weakly globular double categories (which does not extend to all bicategories), namely 
the horizontal embedding $H\colon \mbox{\bf Cat}\rightarrow\mbox{\bf WGDbl}_{\mathrm{st},h}$,
which sends a category $\mathbf{C}$ to the weakly globular double category 
$H(\mathbf{C})$ with the objects of $\mathbf{C}$ as objects, the arrows of $\mathbf{C}$ as 
horizontal arrows, only identity vertical arrows, and only vertical identity cells on the arrows of 
$\mathbf{C}$ as double cells. There is a vertical weak equivalence between the weakly globular 
double categories $H(\mathbf{C})$ and $\Dbl(\mathbf{C})$,
with a strict double functor $\Dbl(\mathbf{C})\rightarrow H(\mathbf{C})$, but its vertical 
pseudo-inverse is only a pseudo-functor.
As a consequence, the composed pseudo-functor $H(\mathbf{C})\rightarrow \Dbl(\mathbf{C}(W^{-1}))$
has still the correct vertical universal property for the weakly globular double category of fractions.

So we are looking for a strict double functor $\mathcal{J}_{W}\colon H(\mathbf{C})\rightarrow \mathbf{C}\{W\}$
such that $\mathbf{C}\{W\}$ is vertically weakly equivalent to $\Dbl(\mathbf{C}(W^{-1}))$.
It turns out that this can be done by turning the arrows in $W$ into precompanions in a very specific way.
The resulting weakly globular double category has both a horizontal universal property and a vertical universal property.

\begin{thm}
The double functor $\mathcal{J}_{W}\colon H(\mathbf{C})\rightarrow \mathbf{C}\{W\}$
induces equivalences of categories
$$
\WGDbl_{\ps,v}(\CW,\bbD)\simeq\WGDbl_{\ps,v,W}(H\bfC,\bbD)
$$
and
$$
\WGDbl_{\st,h}(\CW,\bbD)\simeq\WGDbl_{\st,h,W}(H\bfC,\bbD),
$$
where $\WGDbl_{\ps,v,W}(H\bfC,\bbD)$ consists of pseudo-functors that send arrows in $W$ to precompanions and vertical transformations
that respect this structure, and $\WGDbl_{\st,h,W}(H\bfC,\bbD)$ consists of strict functors and horizontal transformations with a $W$-friendly structure. 
\end{thm}

The 2-functor $H$ is defined for all 2-categories, so one might wonder whether we can extend this 
to 2-categories with a class of arrows $W$. This is indeed possible, but the technicalities 
related to the 2-cells obscure the ideas of this construction. So we will leave this case for a sequel to this paper \cite{sequel}. 
We will then show that the additional ideas needed to handle arbitrary 2-categories are exactly what we need to define 
the weakly globular double category of fractions for an arbitrary weakly globular double category and a class of horizontal 
arrows with the appropiate properties.

The current paper is organized as follows. In Section \ref{dblbic} we give the background material 
on double categories and the equivalence between bicategories and weakly globular double categories.
We also explicitly describe pseudo-functors, and vertical and horizontal transformations between weakly globular double categories.
In Section \ref{comppre} we describe companions and precompanions and show how they are related to quasi units and 
internal equivalences respectively. 
In Section \ref{univ1} we describe the vertical universal property of $\Dbl(\bfC(W^{-1}))$.
In Section \ref{CW-con} we describe the weakly globular double category of fractions $\CW$ and in Section \ref{fracns} we give its universal properties.
Section \ref{conc} consists of concluding remarks and questions for further research.

\medskip

\textbf{Acknowledgements.} The first author is supported by a Marie Curie
International Reintegration Grant no. 256341. She would also like to thank the
Department of Mathematics and Statistics of Dalhousie University for their
hospitality during a visit in June 2012. The second author is supported by an NSERC Discovery grant.
She would also like to thank the Departments of Mathematics of the University of Leicester and Calvin College for their
hospitality during a visit in August 2012 and February 2014 respectively. Both authors thank the Department of Computing of Macquarie
University and CoAct for their hospitality during July 2013.

\bigskip

\section{Double categories and bicategories}\label{dblbic}
In this section we review the definition of weakly globular double category and 
the concepts and notation related to double categories and bicategories used in this paper.
We begin by introducing our notation for double categories.

\subsection{Double categories}
A double category $\bbX$ is an internal category in $\Cat$, i.e.,
a diagram of the form
\begin{equation}\label{dblcat}
\bbX=\left(\xymatrix{\bbX_1\times_{\bbX_0}\bbX_1\ar[r]|-m
    &\bbX_1\ar@<.7ex>[r]^{d_0}\ar@<-.7ex>[r]_{d_1}&\bbX_0\ar[l]|s}\right).
\end{equation}
The elements of $\bbX_{00}$, i.e., the objects of the category $\bbX_0$, are the {\em objects} of the double category.
The elements of $\bbX_{01}$, i.e., the arrows of the category $\bbX_0$, are the {\em vertical arrows} of the double category.
Their (vertical) domains, codomains, identities, and composition in the double category $\bbX$ are as in the category $\bbX_0$. 
For objects $A,B\in \bbX_{00}$,
we write $$\bbX_v(A,B)=\bbX_0(A,B)$$ for the set of vertical arrows from $A$ to $B$.
We denote a vertical identity arrow by $1_A$ and write $\cdot$ for vertical composition.
The elements of $\bbX_{10}$, i.e., the objects of the category $\bbX_1$, are the {\em horizontal arrows} of the double category, and their domain, codomain, identities
and composition are determined by  $d_0$, $d_1$, $s$, and $m$ in (\ref{dblcat}).
For objects $A,B\in\bbX_{00}$, we write
$$\bbX_h(A,B)=\{f\in\bbX_{10}| \,d_0(f)=A\mbox{ and }d_1(f)=B\}.$$
We use $\circ$ for the composition of horizontal arrows, $g\circ f=m(g,f)$.
In order to make a notational distinction between horizontal and vertical arrows, we denote the vertical arrows
by $\xymatrix@1{\ar[r]|\bullet&}$ and the horizontal arrows by $\xymatrix@1{\ar[r]&}$.
The elements of $\bbX_{11}$, i.e., the arrows of the category $\bbX_1$, are the {\em double cells} of the double category.
An element $\alpha\in \bbX_{11}$ has a vertical domain and codomain in $\bbX_{10}$ (since $\bbX_1$
is a category), which are horizontal arrows, say $h$ and $k$ respectively.
The cell $\alpha$ also has a horizontal domain, $d_0(\alpha)$, and a horizontal codomain, $d_1(\alpha)$, which are elements of $\bbX_{01}$, i.e., vertical arrows.
Furthermore, the horizontal and vertical domains and codomains of these arrows
match up in such a way that all this data fits together in a diagram
$$
\xymatrix@R=2.5em@C=2.5em{
\ar[r]^{h} \ar[d]|\bullet_{d_0(\alpha)} \ar@{}[dr]|\alpha
    & \ar[d]|\bullet^{d_1(\alpha)}
\\
\ar[r]_{k}& \rlap{\quad.}
}
$$
These double cells can be composed vertically by composition in $\bbX_1$  (again written as $\cdot$) and horizontally
by using $m$, and written as $\alpha_1\circ\alpha_2=m(\alpha_1,\alpha_2)$.
The identities in $\bbX_1$ give us vertical identity cells, denoted by
$$
\xymatrix{
A\ar[d]|\bullet_{1_A}\ar@{}[dr]|{1_f}\ar[r]^f &B\ar[d]|\bullet^{1_B}\ar@{}[drr]|{\mbox{or}}
  && A\ar@{=}[d] \ar@{}[dr]|{1_f}\ar[r]^f &B\ar@{=}[d]
\\
A\ar[r]_f &B && A\ar[r]_f &B\rlap{\,.}
}
$$
The image of $s$ gives us horizontal identity cells, denoted by
$$
\xymatrix{
A\ar[d]|\bullet_v\ar@{}[dr]|{\mbox{\scriptsize id}_{v}} \ar[r]^{\mbox{\scriptsize Id}_A}
  &A\ar[d]|\bullet^v \ar@{}[drr]|{\mbox{or}}
&& A\ar[d]|\bullet_v\ar@{}[dr]|{\mbox{\scriptsize id}_{v}} \ar@{=}[r]
  &A\ar[d]|\bullet^v
\\
B\ar[r]_{\mbox{\scriptsize Id}_B} & B && B\ar@{=}[r] &B\rlap{\,,}
}$$
where $\mbox{Id}_A=s(A)$ and $\mbox{id}_v=s(v)$.
Composition of squares satisfies horizontal and vertical associativity laws and functoriality of $m$ is equivalent to the middle four interchange law.
Further, $\mbox{id}_{1_A}=1_{\mbox{\scriptsize Id}_A}$ and we will denote this cell by $\iota_A$,
$$
\xymatrix{
A\ar[d]|\bullet_{1_A}\ar[r]^{\mbox{\scriptsize Id}_A} \ar@{}[dr]|{\iota_A}& A\ar[d]|\bullet^{1_A}
\\
A\ar[r]_{\mbox{\scriptsize Id}_A} & A\rlap{\,.}
}$$

For any double category $\bbX$, the {\em horizontal  nerve} $N_h\bbX$ is defined to be the functor
$N_h\bbX\colon \dop\rightarrow\Cat$ such that $(N_h\bbX)_0=\bbX_0$, $(N_h\bbX)_1=\bbX_1$
and $(N_h\bbX)_k=\bbX_1\times_{\bbX_0}\stackrel{k}{\cdots}\times_{\bbX_0}\bbX_1$ for $k\ge 2$.
So $N_h\bbX$ is given by the diagram
$$
\xymatrix{\ar@{}[r]|-\cdots&\bbX_1\times_{\bbX_0}\bbX_1\ar[r]|-m\ar@<.7ex>[r]^-{\pi_1}\ar@<-.7ex>[r]_-{\pi_2}
    &\bbX_1\ar@<.7ex>[r]^{d_0}\ar@<-.7ex>[r]_{d_1}&\bbX_0\ar[l]|s}
$$
The objects at level $k$ are composable paths of horizontal arrows of length $k$
and the arrows at level $k$ are horizontally composable paths of double cells,
$$
\xymatrix{
A_0\ar[d]|\bullet_{v_0}\ar[r]^{h_1}\ar@{}[dr]|{\alpha_1} & A_1 \ar[d]|\bullet_{v_1} \ar[r]^{h_2}\ar@{}[dr]|{\alpha_2} 
		& A_2 \ar[d]|\bullet_{v_1} \ar@{..}[r] & A_{k-1}\ar@{}[dr]|{\alpha_k} \ar[r]^{h_k} & A_k\ar[d]|\bullet^{v_k}
\\
A_0'	\ar[r]_{h_1'} & A_1' \ar[r]_{h_2'} & A_2' \ar@{..}[r] & A_{k-1}'\ar[r]_{h_k'} &A_k'\rlap{\,\,.}
}
$$

\subsection{Pseudo-functors and strict functors}
As maps between (weakly globular) double categories we consider those functors that correspond to 
natural transformations between their horizontal nerves.

\begin{dfn}\rm
\begin{enumerate}
 \item  A {\em strict functor} $F\colon \bbX\rightarrow \bbY$ between double categories 
is given by a natural transformation 
$$F\colon N_h\bbX\Rightarrow N_h\bbY\colon\Delta^{\mbox{\scriptsize op}}\rightarrow\Cat.$$
\item 
A {\em pseudo-functor} $F\colon \bbX\rightarrow \bbY$ between double categories 
is given by a pseudo-natural transformation 
$F\colon N_h\bbX\Rightarrow N_h\bbY\colon\Delta^{\mbox{\scriptsize op}}\rightarrow\Cat.$ 
\end{enumerate}
\end{dfn}

So strict functors send objects to objects, horizontal arrows to horizontal arrows, vertical arrows to vertical arrows, 
and double cells to double cells,
and preserve domains, codomains, identities and horizontal and vertical composition strictly.
Pseudo-functors preserve all of these up to coherent isomorphisms. 
Note that this notion of pseudo-functor between double categories is different 
from what is described in \cite{DPP-spans2}, for instance. The pseudo-functors of \cite{DPP-spans2} 
preserve domains and codomains strictly.

A pseudo-functor $(F,(\varphi_i),\sigma,\mu, (\theta_i))\colon\bbX\rightarrow\bbY$ consists of functors 
$F_0\colon \bbX_0\rightarrow\bbY_0$, $F_1\colon \bbX_1\rightarrow \bbY_1$
and 
$F_k\colon \bbX_1\times_{\bbX_0}\stackrel{k}{\cdots}\times_{\bbX_0}\bbX_1
\rightarrow \bbY_1\times_{\bbY_0}\stackrel{k}{\cdots}\times_{\bbY_0}\bbY_1$ for $k\ge 2$,
together with invertible natural transformations,
\begin{equation}\label{D:varphi}
\xymatrix{
\bbX_1\ar[d]_{d_i}\ar[r]^{F_1}\ar@{}[dr]|{\varphi_i\,\Downarrow}
  &\bbY_1\ar[d]^{d_i} & \bbX_0\ar[d]_s\ar[r]^{F_0}\ar@{}[dr]|{\sigma\,\Downarrow}
  &\bbY_0\ar[d]^s 
\\
\bbX_0\ar[r]_{F_0}&\bbY_0 & \bbX_1\ar[r]_{F_1}&\bbY_1 
\\
\bbX_1\times_{\bbX_0}\bbX_1\ar[d]_m \ar[r]^{F_2}\ar@{}[dr]|{\mu\,\Downarrow} 
  &\bbY_1\times_{\bbY_0}\bbY_1\ar[d]^m 
  &  \bbX_1\times_{\bbX_0}\bbX_1\ar[d]_{\pi_i}\ar[r]^{F_2} \ar@{}[dr]|{\theta_i\,\Downarrow}
  &\bbY_1\times_{\bbY_0}\bbY_1\ar[d]^{\pi_i}
\\
\bbX_1\ar[r]_{F_1} & \bbY_1 
  &\bbX_1\ar[r]_{F_1} & \bbY_1\rlap{\,,}
}
\end{equation}
(where $i=1,2$), and analogously for $F_k$ with $k\ge 2$.
These satisfy the usual naturality and coherence conditions, that can be derived from their simplicial description. 
We will list the ones that we will use in the remainder of this paper.
For instance,
\begin{equation}\label{identityeqn}
\xymatrix{
\bbX_0\ar[d]_s\ar[r]^{F_0}\ar@{}[dr]|{\sigma\,\Downarrow}
  &\bbY_0\ar[d]^s  & \bbX_0\ar@{=}[dd]\ar[r]^{F_0}\ar@{}[ddr]|{1_{F_0}} & \bbY_0\ar@{=}[dd]
\\
\bbX_1\ar[d]_{d_i}\ar[r]^{F_1}\ar@{}[dr]|{\varphi_i\,\Downarrow}
  &\bbY_1\ar[d]^{d_i} \ar@{}[r]|{\textstyle =} & &&\mbox{for }i=1,2.
\\
\bbX_0\ar[r]_{F_0}&\bbY_0& \bbX_0\ar[r]_{F_0}&\bbY_0
}\end{equation}
This means that vertical composition and domains and codomains are preserved strictly,
but horizontal composition and domains and codomains are only preserved up to a vertical isomorphism.
For the pseudo-morphisms we consider in this paper the typical image of a horizontal arrow 
$\xymatrix@1{A\ar[r]^f&B}$ under $F$ corresponds to the following diagram:
$$
\xymatrix@R=2.5em@C=3em{d_0F_1f\ar[r]^{F_1f}\ar[d]|\bullet_{(\varphi_0)_f} ^\sim& d_1F_1f\ar[d]|\bullet_\sim^{(\varphi_1)_f}
\\
F_0A&F_0B
}$$
where $(\varphi_0)_f$ and $(\varphi_1)_f$ are components of $\varphi_0$ and $\varphi_1$ given in (\ref{D:varphi}).
The equation (\ref{identityeqn}) means that 
the components of $\sigma$ have the following shape:
$$
\xymatrix@R=2.5em@C=3em{
F_0A\ar[r]^{\mbox{\scriptsize Id}_{F_0A}}\ar[d]|\bullet_{(\varphi_0)_{\mathrm{Id}_A}^{-1}}\ar@{}[dr]|{\sigma_A} 
  & F_0A\ar[d]|\bullet^{(\varphi_1)_{\mathrm{Id}_A}^{-1}}
\\
d_0F_1(\mathrm{Id}_A) \ar[r]_{F_1(\mbox{\scriptsize Id}_A)} & d_1F_1(\mathrm{Id}_A)
}$$

Naturality of $\mu$ means that for any pair of horizontally composable double cells in $\bbX$,
$$
\xymatrix{
A_1\ar[d]|\bullet_u\ar[r]^{f_1}\ar@{}[dr]|\alpha &B_1\ar[d]|\bullet_{v_1}\ar[r]^{g_1}\ar@{}[dr]|\beta 
	& C_1\ar[d]|\bullet^w
\\
A_2\ar[r]_{f_2} & B_2\ar[r]_{g_2} & C_2\rlap{\,,}
}$$
the following equation holds:
\begin{equation}\label{comp-nat}
 \xymatrix@C=5em@R=3em{
\ar[rrr]^{F_1(g_1\circ f_1)}\ar[d]|\bullet_\sim \ar@{}[drrr]|{\mu_{g_1,f_1}} &&& \ar[d]|\bullet^\sim
\\
\ar[r]^{\pi_2F_2(g_1,f_1)} \ar[d]|\bullet_\sim \ar@{}[dr]|{(\theta_2)_{g_1,f_1}} 
	&\ar[d]|\bullet^\sim \ar@{=}[r] \ar@{}[dddr]|=
	&\ar[r]^{\pi_1F_2(g_1,f_1)} \ar[d]|\bullet_\sim \ar@{}[dr]|{(\theta_1)_{g_1,f_1}} 
	&\ar[d]|\bullet^\sim
\\
\ar[r]|{F_1f_1}\ar[d]|\bullet_\sim \ar@{}[dr]|{F_1\alpha} & \ar[d]|\bullet^\sim 
	& \ar[d]|\bullet_\sim\ar@{}[dr]|{F_1\beta}\ar[r]|{F_1g_1} & \ar[d]|\bullet^\sim \ar@{}[drr]|{\textstyle =}
	&& \ar[r]^{F_1(g_1\circ f_1)}\ar@{}[dr]|{F_1(\beta\circ\alpha)} \ar[d]|\bullet_\sim &\ar[d]|\bullet^\sim
\\
\ar[d]|\bullet_\sim \ar[r]_{F_1f_2} \ar@{}[dr]|{(\theta_2)_{g_2,f_2}^{-1}} &\ar[d]|\bullet^\sim 
	& \ar[d]|\bullet_\sim \ar[r]|{F_1g_2} \ar@{}[dr]|{(\theta_1)_{g_2,f_2}^{-1}} &\ar[d]|\bullet^\sim 
	&& \ar[r]_{F_1(g_2\circ f_2)} &\rlap{\qquad.}
\\
\ar[r]_{\pi_2F_2(g_2,f_2)} \ar[d]|\bullet_\sim \ar@{}[drrr]|{\mu_{g_2,f_2}^{-1}}
	& \ar@{=}[r] & \ar[r]_{\pi_1F_2(g_2,f_2)} &\ar[d]|\bullet^\sim
\\
\ar[rrr]_{F_1(g_2\circ f_2)} &&&
}
\end{equation}
The identity coherence axioms for $m$ and $s$ state that
for a horizontal arrow  $\xymatrix@1{A\ar[r]^f&B}$,
\begin{equation}\label{ass-id1}
\xymatrix@C=5em@R=3em{
 \ar[rrr]^{F_1(\mbox{\scriptsize Id}_B\circ f)} \ar@{}[drrr]|{\mu_{\mathrm{Id}_B, f}} \ar[d]|\bullet 
	&&& \ar[d]|\bullet \ar@{}[4,1]|{\textstyle =}
	& \ar[rr]^{F_1(f)}\ar[4,0]|\bullet\ar@{}[4,2]|{\mathrm{id}_{f}}
	&& \ar[4,0]|\bullet
\\
\ar[ddd]|\bullet\ar[r]^{\pi_2(F_2(\mathrm{Id}_B,f))} \ar@{}[dddr]|{(\theta_2)_{\mathrm{Id}_B,f}}
	&\ar[ddd]|\bullet\ar@{=}[r] 
	& \ar[r]^{\pi_1(F_2(\mathrm{Id}_B,f))} \ar@{}[dr]|{(\theta_1)_{\mathrm{Id}_B,f}}\ar[d]|\bullet 
	&\ar[d]|\bullet & 
\\
&  \ar@{}[dr]|{=}
      &\ar[r]|{F(\mbox{\scriptsize Id}_B)}\ar[d]|\bullet \ar@{}[dr]|{\sigma_B^{-1}} 
      & \ar[d]|\bullet 
\\
&&\ar[d]|\bullet \ar[r]|{\mbox{\scriptsize Id}_{F_0B}} \ar@{}[dr]|{\mbox{\scriptsize id}} &\ar[d]|\bullet 
\\
\ar[r]_{F_1(f)} & \ar@{=}[r] & \ar[r]_{\mbox{\scriptsize Id}_{d_1F_1(f)}} && \ar[rr]_{F_1(f)} &&}
\end{equation}
and
\begin{equation}\label{ass-id2}
\xymatrix@C=5em@R=3em{
 \ar[rrr]^{F_1(f\circ \mbox{\scriptsize Id}_A)} \ar@{}[drrr]|{\mu_{f,\mathrm{Id}_A}} \ar[d]|\bullet 
	&&& \ar[d]|\bullet \ar@{}[4,1]|{\textstyle =}
	& \ar[rr]^{F_1(f)}\ar[4,0]|\bullet\ar@{}[4,2]|{\mathrm{id}_{f}} 
	&& \ar[4,0]|\bullet
\\
\ar[d]|\bullet\ar[r]^{\pi_2(F_2(f,\mathrm{Id}))}\ar@{}[dr]|{(\theta_2)_{f,\mathrm{Id}_A}}
	&\ar[d]|\bullet\ar@{=}[r] \ar@{}[dddr]|{=}
	& \ar[r]^{\pi_1(F_2(f,\mathrm{Id}_A))} \ar@{}[dddr]|{(\theta_1)_{f,\mathrm{Id}_A}}\ar[ddd]|\bullet 
	&\ar[ddd]|\bullet 
\\
\ar[r]|{F_1(\mbox{\scriptsize Id}_A)} \ar[d]|\bullet \ar@{}[dr]|{\sigma_A^{-1}}&\ar[d]|\bullet  
\\
\ar[d]|\bullet \ar[r]|{\mbox{\scriptsize Id}_{F_0A}} \ar@{}[dr]|{\mbox{\scriptsize id}} &\ar[d]|\bullet &&
\\
\ar[r]_{\mbox{\scriptsize Id}_{d_0F_1(f)}}  & \ar@{=}[r] &\ar[r]_{F_1(f)}  & & \ar[rr]_{F_1(f)} &&\rlap{\,.}}
\end{equation}

\subsection{Transformations}
Since double categories have two types of arrows, there are two possible choices
for types of transformations between maps of double functors: vertical and horizontal transformations.
Vertical transformations correspond to modifications between natural transformations of 
functors from $\dop$ into $\Cat$.

\begin{dfn}\rm
A {\em vertical transformation} $\gamma\colon F\Rightarrow G\colon \bbX\rightrightarrows \bbY$ between 
strict double functors has components
vertical arrows $\xymatrix@1{\gamma_A\colon FA\ar[r]|-\bullet &GA}$
indexed by the objects of $\bbX$ and for each horizontal arrow $\xymatrix@1{A\ar[r]^h&B}$
in $\bbX$, a double cell
$$
\xymatrix{
FA\ar[r]^{Fh}\ar[d]|\bullet_{\gamma_A} \ar@{}[dr]|{\gamma_h}& FB\ar[d]|\bullet^{\gamma_B}
\\
GA\ar[r]_{Gh} & GB,
}$$
such that $\gamma$ is strictly functorial in the horizontal direction,
i.e., $\gamma_{h_2\circ h_1}=\gamma_{h_2}\circ\gamma_{h_1}$,
and natural in the vertical direction, i.e., 
\begin{equation}\label{vertnat}
\xymatrix{
FA\ar[r]^{Fh}\ar[d]|\bullet_{Fv}\ar@{}[dr]|{F\zeta} & FB\ar[d]|\bullet^{Fw} 
		&& FA\ar[d]|\bullet_{\gamma_A}\ar[r]^{Fh}\ar@{}[dr]|{\gamma_h} & FB\ar[d]|\bullet^{\gamma_B}
\\
FC\ar[d]|\bullet_{\gamma_C}\ar@{}[dr]|{\gamma_k} \ar[r]_{Fk} & FD\ar[d]|\bullet^{\gamma_D}
		& \equiv & GA\ar[d]|\bullet_{Gv}\ar@{}[dr]|{G\zeta}\ar[r]_{Gh} &GB\ar[d]|\bullet^{Gw}
\\
GC\ar[r]_{Gk} & GD && GC\ar[r]_{Gk} & GD\rlap{\,,}
}
\end{equation}
for any double cell $\zeta$ in $\bbX$.
\end{dfn}

To give a vertical transformation between pseudo-functors of double categories,
we need to require that the data above fits together with the structure cells of the
pseudo-functors.

\begin{dfn}\rm
A {\em vertical transformation} $\gamma\colon F\Rightarrow G\colon \bbX\rightrightarrows \bbY$ between 
pseudo double functors corresponds to a modification between their representations as simplicial natural transformations
between the horizontal nerves of the double categories. It 
 has components
vertical arrows $\xymatrix@1{\gamma_A\colon FA\ar[r]|-\bullet &GA}$
indexed by the objects of $\bbX$ and for each horizontal arrow $\xymatrix@1{A\ar[r]^h&B}$
in $\bbX$, a double cell
$$
\xymatrix{
d_0Fh\ar[r]^{Fh}\ar[d]|\bullet_{d_0\gamma_h} \ar@{}[dr]|{\gamma_h}& d_1Fh\ar[d]|\bullet^{d_1\gamma_h}
\\
d_0Gh\ar[r]_{Gh} & d_1Gh,
}$$
such that the following squares of vertical arrows commute:
$$
\xymatrix{
F_0A\ar[r]|\bullet\ar[d]|\bullet_{\gamma_A} & d_0Fh\ar[d]|\bullet^{d_0\gamma_h}
		&F_0B\ar[r]|\bullet\ar[d]|\bullet_{\gamma_B} & d_1F_1h\ar[d]|\bullet^{d_1\gamma_h}
\\
G_0A\ar[r]|\bullet&d_0G_1h & G_0B\ar[r]|\bullet &d_1G_1h
}
$$
where the unlabeled arrows are the structure isomorphisms corresponding to $F$ and $G$.

We require that $\gamma$ is natural in the vertical direction, in the sense that 
the following square of vertical arrows commutes for a 
vertical arrow $\xymatrix@1{A\ar[r]|\bullet ^v&B}$
in $\bbX$,
$$
\xymatrix{
F_0A\ar[r]|\bullet^{F_1v} \ar[d]|\bullet_{\gamma_A}& F_0B \ar[d]|\bullet^{\gamma_B}
\\
G_0A \ar[r]|\bullet^{G_1v}& G_0B
}$$
and furthermore, for any double cell $\zeta$ in $\bbX$,
$$
\xymatrix{
d_0F_1h\ar[r]^{Fh}\ar[d]|\bullet_{d_0F_1\zeta}\ar@{}[dr]|{F_1\zeta} & d_1F_1h\ar[d]|\bullet^{d_1F_1\zeta} 
		&& d_0F_1h\ar[d]|\bullet_{d_0\gamma_h}\ar[r]^{Fh}\ar@{}[dr]|{\gamma_h} 
		& d_1F_1h\ar[d]|\bullet^{d_1\gamma_h}
\\
d_0F_1k\ar[d]|\bullet_{d_0\gamma_k}\ar@{}[dr]|{\gamma_k} \ar[r]_{Fk} 
		& d_1F_1k\ar[d]|\bullet^{d_1\gamma_k}
		& \equiv & d_0G_1h\ar[d]|\bullet_{d_0G_1\zeta}\ar@{}[dr]|{G_1\zeta}\ar[r]_{Gh} 
		&d_1G_1h\ar[d]|\bullet^{d_1G_1\zeta}
\\
d_0G_1k\ar[r]_{Gk} &d_1G_1k && d_0G_1k\ar[r]_{Gk} & d_1G_1k\rlap{\,.}
}
$$

In the horizontal direction we require pseudo-functoriality, which means that
$$
\xymatrix@C=5em@R=3em{
\ar[rrr]^{F_1(gf)} \ar[d]|\bullet\ar@{}[drrr]|{\mu^F_{g,f}} &&& \ar[d]|\bullet 
		& \ar[rr]^{F_1(gf)} \ar[5,0]|\bullet \ar@{}[5,2]|{\gamma_{gf}} && \ar[5,0]|\bullet
\\
\ar[d]|\bullet \ar[r]^{\pi_2F_2(g,f)} \ar@{}[dr]|{(\theta_2)_{g,f}} & \ar[d]|\bullet \ar@{=}[r] 
	&\ar[d]|\bullet \ar[r]^{\pi_1F_2(g,f)} \ar@{}[dr]|{(\theta_1)_{g,f}} &\ar[d]|\bullet 
\\
\ar[r]|{F_1f} \ar[d]|\bullet_{d_0\gamma_f} \ar@{}[dr]|{\gamma_f} & \ar[d]|\bullet 
	&\ar[d]|\bullet \ar[r]|{F_1g}\ar@{}[dr]|{\gamma_g} 
	& \ar[d]|\bullet^{d_1\gamma_g} \ar@{}[dr]|{\textstyle =} 
	&&   
\\
\ar[d]|\bullet\ar@{}[dr]|{(\theta_2)_{g,f}^{-1}} \ar[r]|{G_1f} &\ar[d]|\bullet 
	&\ar[d]|\bullet \ar[r]|{G_1g}\ar@{}[dr]|{(\theta_1)_{g,f}^{-1}} 
	&\ar[d]|\bullet &&
\\
\ar[d]|\bullet\ar[r]_{\pi_2G_2(g,f)}\ar@{}[drrr]|{(\mu_{g,f}^G)^{-1}} &\ar@{=}[r] & \ar[r]_{\pi_2G_2(g,f)} 
	&\ar[d]|\bullet 
\\
\ar[rrr]_{G_1(gf)}&&& & \ar[rr]_{G_1(gf)}&& \rlap{\quad.}
}$$
\end{dfn}

In order to state the universal property  in Section \ref{fracns} below, 
we also need to consider horizontal 
transformations between strict functors of double categories.
The definition of a horizontal transformation is dual to that of a vertical transformation 
in that all mentions of vertical and horizontal
have been exchanged.

\begin{dfn} \rm A {\em horizontal transformation} $a\colon G\Rightarrow K\colon \bbD\rightrightarrows\bbE$
between strict functors of (weakly globular) double categories has components horizontal arrows 
$\xymatrix@1{GX\ar[r]^{a_X}&KX}$ indexed by the objects of $\bbD$
and for each vertical arrow $\xymatrix@1{X\ar[r]|\bullet^v&Y}$ in $\bbD$, a double cell
$$
\xymatrix@R=1.9em{
GX\ar[r]^{a_X}\ar[d]|\bullet_{Gv} \ar@{}[dr]|{a_v}& KX\ar[d]|\bullet^{Kv}
\\
GY\ar[r]_{a_Y} & KY,
}$$
such that $a$ is strictly functorial in the vertical direction, i.e., $a_{v_1\cdot v_2}=a_{v_1}\cdot a_{v_2}$
and natural in the horizontal direction,
i.e., the composition of
$$
\xymatrix{
GX\ar[d]|\bullet_{Gv}\ar@{}[dr]|{G\zeta}\ar[r]^{Gf} 
    & GX'\ar[d]|\bullet^{Gv'}\ar[r]^{a_{X'}} \ar@{}[dr]|{a_{v'}} 
    & KX'\ar[d]|\bullet^{Kv'}
\\
GY\ar[r]_{Gg} & GY'\ar[r]_{a_{Y'}} &KY'
}$$
is equal to the composition of
$$
\xymatrix{
GX\ar[d]|\bullet_{Gv}\ar@{}[dr]|{a_{v}}\ar[r]^{a_X} 
    & KX\ar[d]|\bullet^{Kv}\ar[r]^{Kf} \ar@{}[dr]|{K\zeta} 
    & KX'\ar[d]|\bullet^{Kv'}
\\
GY\ar[r]_{a_Y} & KY\ar[r]_{Kg} &KY'
}$$
for any double cell $\zeta$ in $\bbD$.
\end{dfn}

We write $\DblCat_{\st,v}$, respectively $\DblCat_{\st,h}$, for the 2-categories of double categories, strict functors,
and vertical transformations, respectively horizontal transformations.
We will write $\DblCat_{\ps,v}$ for the 2-category of double categories, pseudo-functors, and vertical transformations
of pseudo-functors.

There are 2-functors $H\colon \mathbf{2}\mbox{\bf -Cat}\to\DblCat_{\st,h}$ and $V\colon\mathbf{2}\mbox{\bf -Cat}\to\DblCat_{\st,v}$.
$H$ sends a 2-category $\calC$ to a double category $H\calC$ with $H\calC_0=\calC_0$, the discrete category on the objects of $\calC$
and $H\calC_1$ has the arrows of $\calC$ as objects and the 2-cells of $\calC$ as arrows 
(and composition in $H\calC_1$ corresponds to vertical composition in $\calC$). 
The horizontal structure of $H\calC$ corresponds to (horizontal) composition in $\calC$.
$V$ sends the 2-category $\calC$ to a double category $V\calC$ with the objects of $\calC$ as objects, 
the arrows of $\calC$ as vertical arrows, only identity horizontal arrows, and the double cells of $\calC$ are the 2-cells of $H\calC$.
$H$ has a right adjoint $h\colon\DblCat{\st,h}\to \mathbf{2}\mbox{\bf -Cat}$ which sends a double category to its 2-category of horizontal arrows
and special cells with identity arrows as vertical arrows. Analogously, $V$ has a right adjoint $v\colon \DblCat{\st,v}\to \mathbf{2}\mbox{\bf -Cat}$ 
which send a double category to its 2-category of vertical arrows and special cells with identity arrows as horizontal arrows.

\subsection{Weakly globular double categories}
In \cite{PP1} we introduced the notion of weakly globular double category. 
Weakly globular double categories were defined in such a way that there is a biequivalence
of 2-categories,
$$
 \Bic:\WGDbl_{\ps,v}\simeq \NHom\; : \Dbl.
$$
Here $\NHom$ denotes the 2-category of bicategories, homomorphisms and icons.

The adjunction $H\dashv h$ described above restricts to an adjoint equivalence 
between the 2-category of double categories $\bbD$ such that 
$\bbD_0$ is discrete 
and the 2-category of 2-categories.  And a similar result is true when we take pseudo double functors as in \cite{DPP-spans2} for instance and 
homomorphisms of 2-categories with oplax transformations.
The requirement that the vertical category $\bbD_0$ be discrete can be viewed
as a globularity condition: we require that the cells have vertical arrows that are identities, so they are globular in shape.

To obtain a kind of double categories that can model bicategories we choose to weaken this globularity condition
to obtain the first condition in the definition below.
This means that rather than having a set of objects we have a set with an equivalence relation, i.e., a category 
which is equivalent to a discrete one, a posetal groupoid.
In order to make sure that this interacts well with the structure of the horizontal arrows, we need 
the second condition in the definition below. 
To obtain an understanding of this condition, consider some arrangement of horizontal and vertical arrows of the form
\begin{equation}\label{staircase}
\xymatrix@C=3.5em@R=3.5em{
\ar[r]_{f_1} & \ar[d]|\bullet_{v_1} & \ar[d]|\bullet_{v_2}\ar[r]_{f_3} &\ar[d]|\bullet_{v_3} 
		&\ar@{}[r]|\cdots&\ar[d]|\bullet_{v_{n-2}} & \ar[d]|\bullet_{v_{n-1}}\ar[r]_{f_n}&
\\
                   &\ar[r]_{f_2}               &                 &\ar[r]_{f_4}                                 &          &\ar[r]_{f_{n-1}}&
}
\end{equation}
We will refer to  such an arrangement of alternating horizontal and vertical arrows as a {\em staircase path}.
Since we think of the vertical arrows as just indicating that two objects are equivalent,
we require that this diagram can be completed to
$$
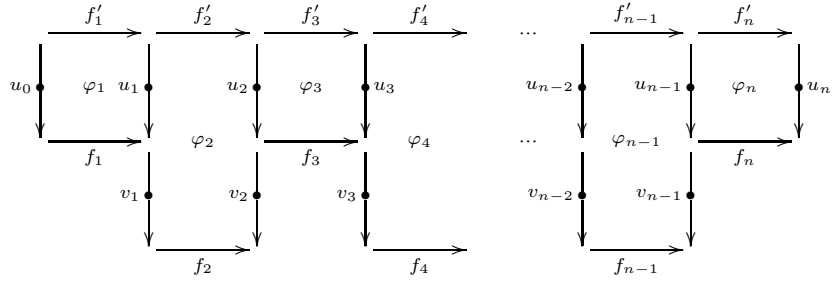
\xymatrix@C=3.5em@R=3.5em{
\ar[r]^{f_1'}\ar[d]|\bullet_{u_0}\ar@{}[dr]|{\varphi_1}	&	\ar[d]|\bullet_{u_1}\ar@{}[ddr]|{\varphi_2}\ar[r]^{f_2'}
		&	\ar[d]|\bullet_{u_2} \ar[r]^{f_3'}\ar@{}[dr]|{\varphi_3}
		&	\ar[d]|\bullet^{u_3} \ar[r]^{f_4'} \ar@{}[ddr]|{\varphi_4} 
		& \ar@{}[r]|\cdots & \ar[d]|\bullet_{u_{n-2}}\ar[r]^{f_{n-1}'}\ar@{}[ddr]|{\varphi_{n-1}} 
		& \ar[d]|\bullet_{u_{n-1}}\ar[r]^{f_n'}\ar@{}[dr]|{\varphi_n} 
		& \ar[d]|\bullet^{u_{n}}
\\
\ar[r]_{f_1} & \ar[d]|\bullet_{v_1} & \ar[d]|\bullet_{v_2}\ar[r]_{f_3} &\ar[d]|\bullet_{v_3}&\ar@{}[r]|\cdots&\ar[d]|\bullet_{v_{n-2}} & \ar[d]|\bullet_{v_{n-1}}\ar[r]_{f_n}&
\\
                   &\ar[r]_{f_2}               &     &\ar[r]_{f_4}      &          &\ar[r]_{f_{n-1}}&
}
$$
where all the double cells are vertically invertible. (And so we can think of the top path as the corresponding path that is horizontally composable.)
Furthermore, we require that this correspondence is part of an equivalence of categories.

\begin{dfn}\label{wklglb}\rm
   A {\em weakly globular double category} $\bbX$ is a  double category
   which satisfies the following two conditions
\begin{itemize}
\item ({\it the weak globularity condition}) there is an equivalence of categories $\gamma:\bbX_0\rw \bbX_0^d$,
    where $\bbX_0^d$ is the discrete category of the path components of $\bbX_{0}$;
 \item ({\it the induced Segal maps condition}) $\gamma$ induces
    an equivalence of categories, for all $n\geq 2$,
    \begin{equation}\label{Segalmaps}
        \pro{\bbX_1}{\bbX_0}{n}\simeq\pro{\bbX_1}{\bbX_0^d}{n}\rlap{\,.}
    \end{equation}
    \end{itemize}
    Analogously to the notation introduced above we write $\WGDbl_{\ps,v}$
for the 2-category of weakly globular double categories, pseudo-functors, and vertical transformations
between them. We also write $\WGDbl_{\st,v}$ and $\WGDbl_{\st,h}$ for the 2-categories with strict functors and
vertical, respectively horizontal, transformations.
\end{dfn}
 
 \begin{rmk}{\rm
Since the vertical arrows in a weakly globular double category are unique and invertible, 
any object $(f_1,f_2,\ldots,f_n)\in \pro{\bbX_1}{\bbX_0^d}{n}$ corresponds uniquely to a diagram 
as in (\ref{staircase}). However, there are many other ways to represent this object as a staircase path.
For instance,
$$
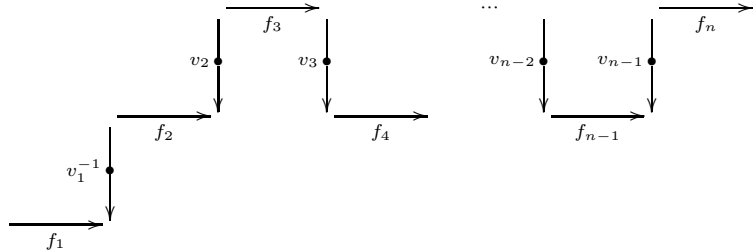
\xymatrix@C=3.5em@R=3.5em{
 & & \ar[d]|\bullet_{v_2}\ar[r]_{f_3} &\ar[d]|\bullet_{v_3} 
		&\ar@{}[r]|\cdots&\ar[d]|\bullet_{v_{n-2}} & \ar[d]|\bullet_{v_{n-1}}\ar[r]_{f_n}&
\\
                   &\ar[d]|\bullet_{v_1^{-1}} \ar[r]_{f_2}               &                 &\ar[r]_{f_4}                                 &          &\ar[r]_{f_{n-1}}&
\\
\ar[r]_{f_1} &                    
}
$$
The induced Segal maps condition would then give us a diagram of vertically invertible double cells that looks slightly different:
$$
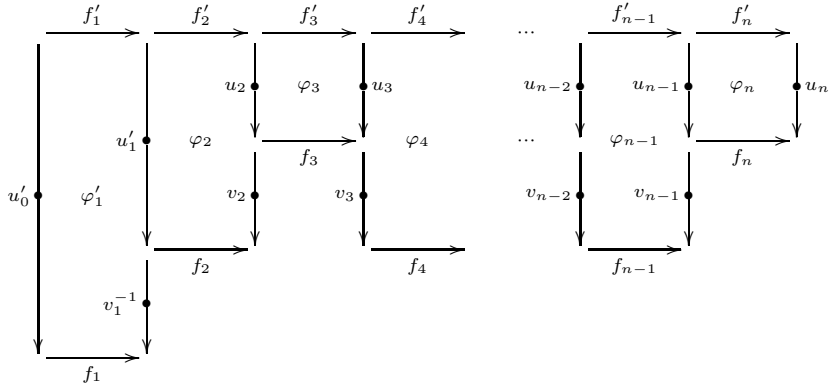
\xymatrix@C=3.5em@R=3.5em{
\ar[r]^{f_1'}\ar[ddd]|\bullet_{u_0'}\ar@{}[dddr]|{\varphi_1'}	&	\ar[dd]|\bullet_{u_1'}\ar@{}[ddr]|{\varphi_2}\ar[r]^{f_2'}
		&	\ar[d]|\bullet_{u_2} \ar[r]^{f_3'}\ar@{}[dr]|{\varphi_3}
		&	\ar[d]|\bullet^{u_3} \ar[r]^{f_4'} \ar@{}[ddr]|{\varphi_4} 
		& \ar@{}[r]|\cdots & \ar[d]|\bullet_{u_{n-2}}\ar[r]^{f_{n-1}'}\ar@{}[ddr]|{\varphi_{n-1}} 
		& \ar[d]|\bullet_{u_{n-1}}\ar[r]^{f_n'}\ar@{}[dr]|{\varphi_n} 
		& \ar[d]|\bullet^{u_{n}}
\\
 & & \ar[d]|\bullet_{v_2}\ar[r]_{f_3} &\ar[d]|\bullet_{v_3} 
		&\ar@{}[r]|\cdots&\ar[d]|\bullet_{v_{n-2}} & \ar[d]|\bullet_{v_{n-1}}\ar[r]_{f_n}&
\\
        &\ar[d]|\bullet_{v_1^{-1}} \ar[r]_{f_2}               &                 &\ar[r]_{f_4}                                 &          &\ar[r]_{f_{n-1}}&
\\
\ar[r]_{f_1} &                    
}
$$
However, note that the composable path of horizontal arrows will still be the same, since this is the path assigned to $(f_1,\ldots,f_n)$ 
by the equivalence. But then we realize that the composites of the vertical arrows also need to be the same, 
so $u_0=u_0'$ and $v_1\cdot u_1=v_1^{-1}\cdot u_1'$ and furthermore, $\varphi_1'=\varphi_1$. 
So we conclude that the choice of cells and horizontal arrows is completely determined by the 
equivalence in the induced Segal maps condition, but the way they are arranged can vary 
depending on which vertical arrows are used in the first staircase path.}
\end{rmk}
 
In the next two subsections,
we briefly review the explicit descriptions of the 2-functors involved in the biequivalence of 2-categories
$$
 \Bic:\WGDbl_{\ps,v}\simeq \NHom\; : \Dbl.
$$
For more details on these functors, see also \cite{PP1}. 

\subsection{The fundamental bicategory}\label{D:Bic}
Let $\bbX$ be a weakly globular double category.
The {\em objects} of $\Bic\bbX$ are obtained as the connected components $\pi_0\bbX_0$
of
the vertical arrow category $\bbX_0$.
When $A$ is an object of $\bbX$, i.e., an element of $\bbX_{00}$,
we write $\bar{A}$ for the corresponding object in $\Bic\bbX$.
Note that  $\bar{A}=\bar{B}$ if and only if there is a (unique) vertical arrow
$v\colon \xymatrix@1{{A}\ar[r]|\bullet&B}$ in $\bbX$
(since the vertical arrow category $\bbX_0$ is posetal and groupoidal).

For any two objects $\bar{A}$ and $\bar{B}$
in $\bbX$, the {\em set of arrows}, $\Bic\bbX(\bar{A},\bar{B})$
is obtained as a disjoint union of horizontal hom-sets in $\bbX$,
$$
\Bic\bbX(\bar{A},\bar{B})
=\coprod_{\scriptstyle \begin{array}{c}\bar{A'}=\bar{A}\\\bar{B'}=\bar{B}\end{array}}\bbX_h(A',B').
$$
Note that we do not put an equivalence relation on the horizontal arrows
of $\bbX$ to obtain the arrows of the fundamental bicategory; we will therefore use the same symbol to denote
a horizontal arrow in $\bbX$ and the corresponding arrow in $\Bic(\bbX)$.

For any two arrows $\xymatrix@1@C=2em{\bar{A}\ar@<.5ex>[r]^f\ar@<-.5ex>[r]_g&\bar{B}}$
in $\Bic\bbX$ represented by horizontal arrows $\xymatrix@1@C=2em{A_1\ar[r]^f&B_1}$
and $\xymatrix@1@C=2em{A_2\ar[r]^g&B_2}$ in $\bbX$, the {\em 2-cells} from $f$ to $g$
correspond to double cells of the form
$$
\xymatrix{
A_1\ar[d]|\bullet_{v}\ar[r]^f\ar@{}[dr]|\alpha &B_1\ar[d]|\bullet^{w}
\\
A_2\ar[r]_g & B_2\rlap{\quad.}
}
$$
Since $v$ and $w$ are unique, we will denote the corresponding 2-cell in $\Bic\bbX$ by
$\alpha\colon f\Rightarrow g$.

Let  $f\colon A_1\rightarrow B_1$ and $g\colon B_2\rightarrow C_2$  be horizontal arrows
such that there is an invertible vertical arrow $\xymatrix@1@C=1.5em{v\colon B_2\ar[r]|-\bullet &B_1}$.
Then the induced Segal maps condition (see Definition \ref{wklglb}) gives rise to a diagram
$$
\xymatrix@R=2em{
A_3\ar[dd]|\bullet_{x}\ar@{}[ddr]|{\varphi_{f_3,f}}\ar[r]^{f_3}
    & B_3\ar[d]|\bullet^y\ar@{}[dr]|{\varphi_{g_3,g}}\ar[r]^{g_3} & C_3\ar[d]|\bullet^z
\\
& B_2\ar[d]|\bullet^{v}\ar[r]_g& C_2
\\
A_1\ar[r]_f & B_1\rlap{\quad.}
}
$$
Then the composition of $f\colon \bar{A}_1\rightarrow\bar{B}_1$ and $g\colon\bar{B}_2\rightarrow\bar{C}_2$
(where $\bar{B}_1=\bar{B}_2$) in $\Bic\bbX$ is defined to be the horizontal composite
$g_3\circ f_3\colon \bar{A}_1=\bar{A}_3\rightarrow\bar{C}_3=\bar{C}_2$.

The horizontal composition of 2-cells is defined as follows:
Let
$$\xymatrix@R=2em{\bar{A}\ar@/^2ex/[r]^{f}\ar@{}[r]|{\Downarrow\alpha}\ar@/_2ex/[r]_g&
    \bar{B} \ar@/^2ex/[r]^{h}\ar@{}[r]|{\Downarrow\beta}\ar@/_2ex/[r]_k& \bar{C}}$$
be a diagram of arrows and cells in the fundamental bicategory, represented by
double cells in $\bbX$,
$$
\xymatrix@R=2em{
A_2 \ar[d]|\bullet_{u_{21}} \ar[r]^f \ar@{}[dr]|\alpha &
    B_2 \ar[d]|\bullet^{v_{21}} \ar@{}[drr]|{\mbox{and}} &&
    B_4 \ar[d]|\bullet_{v_{43}}\ar[r]^h\ar@{}[dr]|\beta & C_4\ar[d]|\bullet^{w_{43}}
\\
A_1\ar[r]_g & B_1 && B_3\ar[r]_k & C_3.}
$$
Let the composite of $g$ and $k$ in $\Bic\bbX$ be the arrow $k_5\circ g_5$
with a corresponding diagram (using the induced Segal maps condition from Definition \ref{wklglb} again),
$$
\xymatrix@R=2em{
A_5\ar[dd]|\bullet_{u_{51}}\ar@{}[ddr]|{\varphi_{g_5,g}}\ar[r]^{g_5}
    & B_5\ar[d]|\bullet_{v_{53}}\ar@{}[dr]|{\varphi_{k_5,k}}\ar[r]^{k_5}
    & C_5\ar[d]|\bullet^{w_{53}}
\\
& B_3\ar[d]|\bullet^{v_{31}}\ar[r]_k & C_3
\\
A_1\ar[r]_g & B_1
}
$$
and let the composite of $f$ and $h$ be the arrow $h_6\circ f_6$ as in the diagram
$$
\xymatrix@R=2em{
A_6\ar[dd]|\bullet_{u_{62}}\ar@{}[ddr]|{\varphi_{f_6,f}}\ar[r]^{f_6}
    & B_6\ar[d]|\bullet_{v_{64}}\ar@{}[dr]|{\varphi_{h_6,h}}\ar[r]^{h_6}
    & C_6\ar[d]|\bullet^{w_{64}}
\\
& B_4\ar[d]|\bullet^{v_{42}}\ar[r]_h & C_4
\\
A_2\ar[r]_f & B_2
}
$$
Then the horizontal composition of $\alpha$ and $\beta$ is represented by the following
pasting of double cells:
$$
\xymatrix@R=2em{
A_6 \ar[ddd]|\bullet_{u_{62}} \ar@{}[dddr]|{\varphi_{f_6,f}} \ar[r]^{f_6}
    & B_6\ar[d]|\bullet_{v_{64}} \ar@{}[dr]|{\varphi_{h_6,h}} \ar[r]^{h_6}
    & C_6\ar[d]|\bullet^{w_{64}}
\\
&B_4\ar[d]|\bullet^{v_{43}}\ar@{}[dr]|{\beta} \ar[r]_h & C_4\ar[d]|\bullet^{w_{43}}
\\
&B_3\ar[r]_k\ar[d]|\bullet^{v_{32}} \ar@{}[dddr]|{\varphi^{-1}_{k_5,k}} & C_3\ar[ddd]|\bullet^{v_{35}}
\\
A_2\ar[d]|\bullet_{u_{21}}\ar@{}[dr]|{\alpha} \ar[r]^f & B_2\ar[d]|\bullet_{v_{21}}
\\
A_1\ar[d]|\bullet_{u_{15}}\ar@{}[dr]|{\varphi_{g_5,g}^{-1}} \ar[r]_g & B_1\ar[d]|\bullet^{v_{15}}
\\
A_5\ar[r]_{g_5}&B_5\ar[r]_{k_5} & C_5.
}
$$
(Here, $u_{ij}=u_{ji}^{-1}$ and $u_{jk}\cdot u_{ij}=u_{ik}$, and analogous for $v$ and $w$,
since the vertical category is groupoidal posetal.
Furthermore, the same holds for the cells, because they are components of a vertical transformation.)

The units for the composition are obtained from the functor $\mu_0\colon\bbX_0^d\rightarrow\bbX_0$
which is part of the equivalence of
categories, $\bbX_0^d\simeq \bbX_0$. For an object $\bar{A}$ in $\Bic\bbX$, $1_{\bar{A}}$
is the horizontal arrow
$\mbox{Id}_{\mu_0(\bar{A})}$.

There are associativity and unit isomorphisms for this composition that satisfy the usual coherence conditions
by the results in \cite{tam} and \cite {lp}.

\subsection{The weakly globular double category of paths}
Let $\calB$ be a bicategory.
Before we begin the construction of $\Dbl(\calB)$, we first choose a composite $\varphi_{f_{1},\ldots,f_{n}}$ for each finite path
$\xymatrix@1@C=1.3em{A_0\ar[r]^{f_{1}}&A_1\ar[r]^{f_{2}}&\cdots\ar[r]^{f_{n}}&A_n}$ of
(composable) arrows in
$\calB$. If the path is empty, we take $\varphi_{A_0}=1_{A_0}$.
For each path of such paths,
$$
\xymatrix@C=3.85em{(\ar[r]^{f_{i_1}}&\cdots\ar[r]^{f_{i_{n_1}}}&)(\ar[r]^{f_{i_{n_1+1}}}
    &\cdots\ar[r]^{f_{i_{n_2}}}&)
    \quad\cdots\quad (\ar[r]^-{f_{i_{n_{m-1}+1}}} &\cdots\ar[r]^{f_{i_{n_m}}}&)}
$$
the associativity and unit cells give rise to unique invertible comparison 2-cells, which we denote by
$$
\xymatrix@C=20em{
\ar@/^5ex/[r]^{\varphi_{\varphi_{f_{i_1},\ldots,f_{i_{n_1}}},\varphi_{f_{i_{n_1+1}},\ldots,f_{i_{n_2}}},\ldots,
    \varphi_{f_{i_{n_{m-1}+1}},\ldots,f_{i_{n_m}}}}} \ar@/_5ex/[r]_{\varphi_{f_{i_1},\ldots,f_{i_{n_m}}}}
    \ar@{}[r]|{\Phi_{f_{i_1}\cdots f_{i_{n_1}},f_{i_{n_1+1}}\cdots f_{i_{n_2}},\ldots,f_{i_{n_{m-1}+1}}\cdots f_{i_{n_m}}}} &\rlap{\quad.}
}
$$
(The uniqueness of these cells follows from the associativity and unit coherence axioms.)

The {\em objects} of $\Dbl(\calB)$ are given as pairs of an arrow $\psi\colon[0]\rightarrow [n]$ in
$\Delta$ with a path, $\xymatrix@1{A_0\ar[r]^{f_1} & A_1 \ar[r]^{f_2} & \cdots\ar[r]^{f_n} & A_n}$,
of length $n$ in $\calB$, for all $n$.
Since the arrow $\psi$ is determined by its image $i_0=\psi(0)\in[n]$, we will denote this object in
$\Dbl(\calB)$ by $$(\xymatrix{A_0\ar[r]^{f_1} & A_1 \ar[r]^{f_2} & \cdots\ar[r]^{f_n}&A_n};i_0)$$
and think of $A_{i_0}$ as a marked object along the path. So we will also use the notation
$$\xymatrix{A_0\ar[r]^{f_1} & A_1\ar[r]^{f_2}&\cdots\ar[r]^{f_{i_0}}
  &[A_{i_0}]\ar[r]^{f_{i_0+1}}& \cdots\ar[r]^{f_n}&A_n\rlap{\,.}}$$

There is a unique {\em vertical arrow} from
$\xymatrix@1@C=1.8em{A_0\ar[r]^{f_1} & A_1\ar[r]^{f_2}&\cdots\ar[r]^{f_{i_0}}
  &[A_{i_0}]\ar[r]^{f_{i_0+1}}& \cdots\ar[r]^{f_n}&A_n}$ to
$\xymatrix@1@C=2em{B_0\ar[r]^{g_1} & B_1\ar[r]^{g_2}&\cdots\ar[r]^{g_{j_0}}
  &[B_{j_0}]\ar[r]^{g_{j_0+1}}& \cdots\ar[r]^{g_m}&B_m}$ if and only if $A_{i_0}=B_{j_0}$.
In diagrams we will include this vertical arrow in the following way
$$\xymatrix{
A_0\ar[r]^{f_1} & A_1\ar[r]^{f_2}&\cdots\ar[r]^{f_{i_0}}
  &[A_{i_0}]\ar@{=}[d]\ar[r]^{f_{i_0+1}}& \cdots\ar[r]^{f_n}&A_n
\\
B_0\ar[r]^{g_1} & B_1\ar[r]^{g_2}&\cdots\ar[r]^{g_{j_0}}
  &[B_{j_0}]\ar[r]^{g_{j_0+1}}& \cdots\ar[r]^{g_m}&B_m
}
$$

{\em Horizontal arrows} in $\Dbl(\calB)$ are given as  pairs of an arrow
$\psi\colon[1]\rightarrow [n]$ in $\Delta$ with a path of length $n$
in $\calB$, for all $n$. Analogous to what we did for objects we denote horizontal arrows by
$$(\xymatrix{A_0\ar[r]^{f_1} & A_1 \ar[r]^{f_2} & \cdots\ar[r]^{f_n}&A_n};i_0,i_1)\quad\mbox{ with }\quad i_0\le i_1,$$
or by
$$\xymatrix{A_0\ar[r]^{f_1} & A_1 \ar[r]^{f_2}&\cdots\ar[r]^{f_{i_0}}&[A_{i_0}]\ar[r]^{f_{i_0+1}}& \cdots
\ar[r]^{f_{i_1}}&[A_{i_1}]\ar[r]^{f_{i_1+1}}& \cdots\ar[r]^{f_n}&A_n\rlap{\,.}}$$
The domain of $(\xymatrix@C=1.2em@1{A_0\ar[r]^{f_1} & A_1 \ar[r]^{f_2}& \cdots\ar[r]^{f_n}&A_n};i_0,i_1)$
is $(\xymatrix@1@C=1.2em{A_0\ar[r]^{f_1} & A_1 \ar[r]^{f_2}& \cdots\ar[r]^{f_n}&A_n};i_0)$
and the codomain is $(\xymatrix@C=1.2em{A_0\ar[r]^{f_1} & A_1 \ar[r]^{f_2}& \cdots\ar[r]^{f_n}&A_n};i_1)$.
For a horizontal identity arrow,
$$
(\xymatrix{A_0\ar[r]^{f_1} & A_1\ar[r]^{f_2}& \cdots\ar[r]^{f_n}&A_n};i_0,i_0)
$$
we will also use the notation
$$
\xymatrix{A_0\ar[r]^{f_1} & A_1\ar[r]^{f_2}& \ar[r]^{f_{i_0}}&[[A_{i_0}]]
\ar[r]^{f_{i_0+1}}& \cdots\ar[r]^{f_n}&A_n}
$$
or
$$
\xymatrix{A_0\ar[r]^{f_1} & A_1\ar[r]^{f_2}& \ar[r]^{f_{i_0}}&[A_{i_0}]\ar@{=}[r] &[A_{i_0}]
\ar[r]^{f_{i_0+1}}& \cdots\ar[r]^{f_n}&A_n}
$$
when this makes it easier to fit such an arrow into a diagram representing a double cell as shown below.

A {\em double cell} consists of two horizontal arrows
$$(\xymatrix@C=1.7em{A_0\ar[r]^{f_1} & A_1\ar[r]^{f_2}& \cdots\ar[r]^{f_n}&A_n};i_0,i_1)\mbox{ and }
(\xymatrix@C=1.7em{B_0\ar[r]^{g_1} & B_1\ar[r]^{g_2}& \cdots\ar[r]^{g_m}&B_m};j_0,j_1)$$
(for the vertical domain and codomain respectively), such that
$A_{i_0}=B_{j_0}$ and $A_{i_1}=B_{j_1}$ (such that there are unique vertical arrows between the domains
of these arrows and between the codomains of these arrows),
together with a 2-cell in $\calB$ between the chosen composites,
$$
\xymatrix@C=5em{\ar@/^4ex/[r]^{\varphi_{f_{i_0+1},\ldots,f_{i_1}}}
      \ar@/_4ex/[r]_{\varphi_{g_{j_0+1},\ldots,g_{j_1}}}
      \ar@{}[r]|{\Downarrow\alpha}&\rlap{\quad.}}
$$
We combine all this information together in the following diagram
$$
\xymatrix@R=4em{
A_0\ar[r]^{f_1}&\cdots\ar[r]^{f_{i_0}}
  &[A_{i_0}]\ar@/_2ex/[rr]_{\varphi_{f_{i_0+1},\ldots,f_{i_1}}}\ar@{}[drr]|{\alpha}\ar@{=}[d]\ar[r]^{f_{i_0+1}}
  &\cdots\ar[r]^{f_{i_1}}
  &[A_{i_1}]\ar@{=}[d]\ar[r]^{f_{i_1+1}}
  &\cdots\ar[r]^{f_n}&A_n
\\
B_0\ar[r]_{g_1}&\cdots\ar[r]_{g_{j_0}}
  &[B_{j_0}] \ar@/^2ex/[rr]^{\varphi_{g_{j_0+1},\ldots,g_{j_1}}}\ar[r]_{g_{j_0+1}}
  &\cdots\ar[r]_{g_{j_1}}
  &[B_{j_1}]\ar[r]_{g_{j_1+1}}
  &\cdots\ar[r]_{g_m}&B_m
}
$$
So this represents a double cell in $\Dbl(\calB)$.

Two horizontal arrows, $$(\xymatrix@1{A_0\ar[r]^{f_1} & A_1\ar[r]^{f_2}& \cdots\ar[r]^{f_n}&A_n};i_0,i_1)
\mbox{ and }(\xymatrix@1{B_0\ar[r]^{g_1} & B_1\ar[r]^{g_2}& \cdots\ar[r]^{g_m}&B_m};j_0,j_1),$$ are composable if
and only if the two paths are the same, i.e., $m=n$, $A_i=B_i$ for $i=0,\ldots,n$, and $f _i=g_i$ for $i=1,\ldots,n$,
and furthermore, $i_1=j_0$. In that case, the {\em horizontal composition} of these arrows is given by
$(\xymatrix@1{A_0\ar[r]^{f_1} & A_1\ar[r]^{f_2}& \cdots\ar[r]^{f_n}&A_n};i_0,j_1)$.

The {\em horizontal composition of double cells}
$$
\xymatrix@R=4em{
A_0\ar[r]^{f_1}&\cdots\ar[r]^{f_{i_0}}
  &[A_{i_0}]\ar@/_2ex/[rr]_{\varphi_{f_{i_0+1},\ldots,f_{i_1}}}\ar@{}[drr]|{\alpha}\ar@{=}[d]\ar[r]^{f_{i_0+1}}
  &\cdots\ar[r]^{f_{i_1}}
  &[A_{i_1}]\ar@{=}[d]\ar[r]^{f_{i_1+1}}
  &\cdots\ar[r]^{f_n}&A_n
\\
B_0\ar[r]_{g_1}&\cdots\ar[r]_{g_{j_0}}
  &[B_{j_0}] \ar@/^2ex/[rr]^{\varphi_{g_{j_0+1},\ldots,g_{j_1}}}\ar[r]_{g_{j_0+1}}
  &\cdots\ar[r]_{g_{j_1}}
  &[B_{j_1}]\ar[r]_{g_{j_1+1}}
  &\cdots\ar[r]_{g_m}&B_m
}
$$
and
$$
\xymatrix@R=4em{
A_0\ar[r]^{f_1}&\cdots\ar[r]^{f_{i_1}}
  &[A_{i_1}]\ar@/_2ex/[rr]_{\varphi_{f_{i_1+1},\ldots,f_{i_2}}}\ar@{}[drr]|{\beta}\ar@{=}[d]\ar[r]^{f_{i_1+1}}
  &\cdots\ar[r]^{f_{i_2}}
  &[A_{i_2}]\ar@{=}[d]\ar[r]^{f_{i_2+1}}
  &\cdots\ar[r]^{f_n}&A_n
 \\
 B_0\ar[r]_{g_1}&\cdots\ar[r]_{g_{j_1}}
  &[B_{j_1}] \ar@/^2ex/[rr]^{\varphi_{g_{j_1+1},\ldots,g_{j_2}}}\ar[r]_{g_{j_1+1}}
  &\cdots\ar[r]_{g_{j_2}}
  &[B_{j_2}]\ar[r]_{g_{j_2+1}}
  &\cdots\ar[r]_{g_m}&B_m
 }
 $$
 is defined to be
 $$
\xymatrix@R=4em{
A_0\ar[r]^{f_1}&\cdots\ar[r]^{f_{i_0}}
  &[A_{i_0}]\ar@/_2ex/[rr]_{\varphi_{f_{i_0+1},\ldots,f_{i_2}}}
    \ar@{}[drr]|{\alpha\otimes\beta}\ar@{=}[d]\ar[r]^{f_{i_0+1}}
  &\cdots\ar[r]^{f_{i_2}}
  &[A_{i_2}]\ar@{=}[d]\ar[r]^{f_{i_2+1}}
  &\cdots\ar[r]^{f_n}&A_n
 \\
  B_0\ar[r]_{g_1}&\cdots\ar[r]_{g_{j_0}}
  &[B_{j_0}] \ar@/^2ex/[rr]^{\varphi_{g_{j_0+1},\ldots,g_{j_2}}}\ar[r]_{g_{j_0+1}}
  &\cdots\ar[r]_{g_{j_2}}
  &[B_{j_2}]\ar[r]_{g_{j_2+1}}
  &\cdots\ar[r]_{g_m}&B_m
 }
 $$
 where $\alpha\otimes\beta$ is the 2-cell in $\calB$ given by the following pasting diagram
 $$
\xymatrix@C=15em@R=8em{
 \ar@/^8ex/[rr]^{\varphi_{f_{i_0+1},\ldots,f_{i_2}}}
 \ar@/_8ex/[rr]_{\varphi_{g_{j_0+1},\ldots,g_{j_2}}}
 \ar@{}@<3ex>[rr]^{\Phi_{f_{i_0+1}\cdots f_{i_1},f_{i_1+1}\cdots f_{i_2}}}
\ar@{}@<-3ex>[rr]_{\Phi_{g_{j_0+1}\cdots g_{j_1},g_{j_1+1}\cdots g_{j_2}}}
\ar@/^4ex/[r]_{\varphi_{f_{i_0+1},\ldots,f_{i_1}}}
\ar@/_4ex/[r]^{\varphi_{g_{j_0+1},\ldots,g_{j_1}}}
\ar@{}[r]|\alpha
&\ar@/^4ex/[r]_{\varphi_{f_{i_1+1},\ldots,f_{i_2}}}
\ar@/_4ex/[r]^{\varphi_{g_{j_1+1},\ldots,g_{j_2}}}
\ar@{}[r]|\beta
&
}
$$

\begin{rmks}\rm
\begin{enumerate}
 \item Note that both the category of horizontal arrows and the category of vertical arrows of
$\Dbl(\calB)$ are posetal.
\item We call $\Dbl(\calB)$ the {\em double category of marked paths in $\calB$}.
\item
For a 2-category $\calC$,
the double category $\Dbl(\calC)$ is not isomorphic to the double category $H\calC$, but it is 2-equivalent to it (in the vertical direction).
And the same is true for a category $\bfC$: $H\bfC\not\cong\Dbl(\bfC)$, but $H\bfC\simeq_2\Dbl(\bfC)$.
\end{enumerate}
\end{rmks}

\section{Companions and precompanions}\label{comppre}

In general double categories, the notions of companion and conjoint have been recognized
as important concepts related to the notion of adjoint.
While adjoint arrows have to be of the same kind, i.e., both horizontal or both vertical, 
the relations of companionship and conjointship are for arrows of different types.
The notions of companion and conjoint were first introduced by Ehresmann in \cite{GP},
but companions in symmetric double categories (where the horizontal and vertical arrow categories are the same)
were studied by Brown and Mosa \cite{BM} under the name `connections'.
Connection pairs were first introduced by Spencer in \cite{spen}.
The existence of companions and conjoints for the vertical arrows in a double category is
related to Shulman's notion of an anchored bicategory \cite{Shul}, 
also called a gregarious double category in \cite{DPP-spans2}.
We will show that for weakly globular double categories horizontal companions are related to
so-called quasi units, i.e., arrows with an invertible 2-cell to a unit arrow, in their associated bicategories.
Then we will introduce a slightly weaker notion, that of a {\em precompanion}. 
We will show that precompanions in weakly globular double categories correspond to equivalences in bicategories.
In Section \ref{fracns} we will show that both companions and precompanions  play an important role in 
the description of the universal properties 
of a weakly globular double
category of fractions. 

We begin this section by repeating the definitions of companion and conjoint 
from \cite{GP} to set our notation, and then we will discuss their relationship to quasi units.
Then we will introduce both a category and a double category of companions. 
And in the last part we will study precompanions and their properties.

\subsection{Companions and conjoints}
Recall the definitions of companion and conjoint.

\begin{dfn}\label{friends}\rm
Let $\bbD$ be a double category and consider horizontal morphisms $f\colon A\rightarrow B$ 
and $u\colon B\rightarrow A$ and a vertical morphism 
$\xymatrix@1@C=1.5em{v\colon A\ar[r]|-{\scriptscriptstyle\bullet}&B}$.
We say that $f$ and $v$ are {\em companions} if there exist {\em binding cells}
$$
\xymatrix{
A\ar@{=}[d]\ar@{=}[r] \ar@{}[dr]|\psi & A\ar[d]|{\scriptscriptstyle\bullet}^v \ar@{}[drr]|{\mbox{and}} 
	&& A\ar[d]|{\scriptscriptstyle\bullet}_v \ar[r]^f \ar@{}[dr]|\chi & B\ar@{=}[d]
\\
A\ar[r]_f & B && B\ar@{=}[r] &B,
}
$$
such that
\begin{equation}\label{compcondns}
\xymatrix@R=.5em@C=2em{
	&&&&&&A\ar@{=}[dd]\ar@{=}[r]\ar@{}[ddr]|\psi & A\ar[dd]|{\scriptscriptstyle\bullet}^v 
\\
A \ar@{=}[dd]\ar@{=}[r] \ar@{}[ddr]|\psi
	& A\ar[dd]|{\scriptscriptstyle\bullet}^v \ar[r]^f \ar@{}[ddr]|\chi
	& B\ar@{=}[dd] \ar@{}[ddr]|{\textstyle =} & A\ar[r]^f\ar@{=}[dd] \ar@{}[ddr]|{1_f} 
	& B\ar@{=}[dd]
	&&&& A \ar[dd]|{\scriptscriptstyle\bullet}_v \ar@{}[ddr]|{\mbox{\scriptsize id}_v} \ar@{=}[r] 
	& A \ar[dd]|{\scriptscriptstyle\bullet}^v
\\
	&&&& \ar@{}[rr]|{\mbox{and}}&& A\ar[r]_f \ar[dd]|{\scriptscriptstyle\bullet}_v \ar@{}[ddr]|\chi 
	& B\ar@{=}[dd] \ar@{}[r]|{\textstyle{=}}&
\\
A\ar[r]_f & B\ar@{=}[r] & B & A\ar[r]_f & B&&	&&B\ar@{=}[r] & B
\\
	&&&&&&B\ar@{=}[r] & B&&.
}
\end{equation}
Dually, $u$ and $v$ are {\em conjoints} if there exist {\em binding cells}
$$
\xymatrix{
A\ar[d]|{\scriptscriptstyle\bullet}_v \ar@{}[dr]|\alpha \ar@{=}[r] & A\ar@{=}[d] \ar@{}[drr]|{\mbox{and}} 
	&& B\ar@{=}[d] \ar[r]^u\ar@{}[dr]|\beta & A \ar[d]|{\scriptscriptstyle\bullet}^v 
\\
B\ar[r]_u & A  && B\ar@{=}[r] & B,	
}
$$
such that
\begin{equation}\label{conjcondns}
\xymatrix@R=.5em{
	&&&&&& A\ar[dd]|{\scriptscriptstyle\bullet}_v\ar@{}[ddr]|\alpha \ar@{=}[r] & A\ar@{=}[dd]
\\
B\ar@{=}[dd]\ar@{}[ddr]|\beta \ar[r]^u 
	& A\ar[dd]|{\scriptscriptstyle\bullet}^v \ar@{}[ddr]|\alpha\ar@{=}[r] 
	& A\ar@{=}[dd] \ar@{}[ddr]|{\textstyle =} 
	& B\ar@{=}[dd]\ar[r]^u\ar@{}[ddr]|{1_u} & A\ar@{=}[dd]
	&&&&A\ar@{=}[r]\ar[dd]|{\scriptscriptstyle\bullet}_v \ar@{}[ddr]|{\mbox{\scriptsize id}_v} & A \ar[dd]|{\scriptscriptstyle\bullet}^v
\\
	&&&&\ar@{}[rr]|{\mbox{and}} 
	&&B\ar@{=}[dd]\ar@{}[ddr]|\beta \ar[r]_u & A\ar[dd]|{\scriptscriptstyle\bullet}^v \ar@{}[r]|{\textstyle =}&
\\
B\ar@{=}[r] & B\ar[r]_u & A &B\ar[r]_u &A, &&&&B\ar@{=}[r] & B
\\
&&&&&&B\ar@{=}[r] & B&&.
}
\end{equation}
\end{dfn}

\begin{rmk}\label{compconj}\rm
Since the vertical arrow category in a weakly globular double category is a posetal groupoid, it follows 
from  Proposition 3.5 in \cite{DPP-spans2} that 
the inverse of a vertical arrow which is a companion is a conjoint, and the binding cells for the conjoint pair can be taken to be the vertical inverses of the
binding cells for the companion pair. So a horizontal arrow in a weakly globular double category has a vertical conjoint if and 
only if it has a vertical companion. So it will often be sufficient to focus on the companion pairs.
\end{rmk}

\subsection{Units}
\label{compquasi}

The units in a bicategory are weak units in the sense that  
there are invertible 2-cells 
\begin{equation}\label{weakunits}
 \lambda_f\colon 1_B\circ f \stackrel{\sim}{\Rightarrow}f
\mbox{ and }\rho_f\colon f\circ 1_A\stackrel{\sim}{\Rightarrow} f
\end{equation}
for any arrow $f\colon A\rightarrow B$
(and these need to satisfy the coherence conditions). Such units are not necessarily unique: any arrow $g\colon A\rightarrow A$ with an invertible 2-cell 
$g\stackrel{\sim}{\Rightarrow} 1_A$ would satisfy (\ref{weakunits}).
Since we don't want to consider the coherence conditions that one normally requires of weak units, 
we introduce the notion of quasi units:

\begin{dfn}\rm
An endomorphism $f\colon A\rightarrow A$ in a bicategory is a {\em quasi unit}
if it satisfies the following equivalent conditions:
\begin{enumerate}
\item
$f\cong 1_A$;
\item
$f\circ g\cong g$ for all $g\colon C\rightarrow A$ and $h\circ f \cong h$ for all $h\colon A\rightarrow B$.
\end{enumerate}
\end{dfn} 

We want to characterize the quasi units in the fundamental bicategory $\Bic\bbX$.

\begin{lma}
Every arrow of the form $\mbox{\rm Id}_{A}\colon \bar{A}\rightarrow\bar{A}$
is a quasi unit in $\Bic\bbX$.
\end{lma}

\begin{proof}
Recall from the end of Section \ref{D:Bic} that the identity arrow on $\bar{A}$ is
$\mbox{\rm Id}_{\mu_0(\bar{A})}$ where $\mu_0\colon\bbX_0^d\rightarrow\bbX_0$ is part of the equivalence of categories 
in the weak globularity condition.  
There is an invertible 2-cell $\mbox{Id}_{A}\Rightarrow \mbox{Id}_{\mu_0(\bar{A})}$
given by the vertically invertible double cell
$$
\xymatrix{
A\ar[d]|\bullet_{x} \ar@{}[dr]|{\mbox{\scriptsize id}_{x}} \ar[r]^{\mbox{\scriptsize Id}_A} & A\ar[d]|\bullet^x
\\
\mu_0(\bar{A}) \ar[r]_{\mbox{\scriptsize Id}_{\mu_0(\bar{A})}} & \mu_0(\bar{A})
}
$$
(with vertical inverse $\mbox{id}_{x^{-1}}$).
\end{proof}

We can now characterize the quasi units in $\Bic\bbX$ as those horizontal arrows in $\bbX$ which have a companion.

\begin{prop}\label{quasi-comp}
Let $w\colon A\rightarrow B$ be a horizontal arrow in a weakly globular double category $\bbX$.
Then $w\colon \bar{A}\rightarrow\bar{B}$ is a quasi unit in $\Bic\bbX$ if and only if 
$w\colon  A\rightarrow B$ has a companion $w_* \colon \xymatrix@1{A\ar[r]|\bullet&B}$ in $\bbX$.
\end{prop}

\begin{proof}
If $w$ is a quasi unit, then $\bar{A}=\bar{B}$ and there is an invertible 2-cell 
$\zeta\colon w\Rightarrow 1_{\bar{A}}$ in $\Bic\bbX$, given by a vertically 
invertible double cell in $\bbX$,
$$
\xymatrix{
A\ar[r]^w\ar[d]|\bullet_x\ar@{}[dr]|\zeta &B\ar[d]|\bullet^y
\\
\theta(\bar{A})\ar[r]_{\mbox{\scriptsize Id}_{\theta(\bar{A})}} & \theta(\bar{A})
}
$$ 
We can compose this with the horizontal identity cell on $y^{-1}$ to obtain a double cell
$$
\xymatrix{
A\ar[r]^w\ar[d]|\bullet_x\ar@{}[dr]|\zeta &B\ar[d]|\bullet^y 
	&& A\ar[dd]|\bullet_{y^{-1}\cdot x} \ar[rr]^w \ar@{}[ddrr]|{\mbox{\scriptsize id}_{y^{-1}}\cdot\zeta} 
	&& B\ar[dd]|\bullet^{1_B}
\\
\theta(\bar{A})\ar[r]_{\mbox{\scriptsize Id}_{\theta(\bar{A})}} \ar[d]|\bullet_{y^{-1}} 
    \ar@{}[dr]|{\mbox{\scriptsize id}_{y^{-1}}} 
	& \theta(\bar{A})\ar[d]|\bullet^{y^{-1}} &=& 
\\
B\ar[r]_{\mbox{\scriptsize Id}_B} & B && B\ar[rr]_{\mbox{\scriptsize Id}_B} && B
}
$$ 
Furthermore, precomposing $\zeta^{-1}$ (the vertical inverse of $\zeta$) with 
the horizontal identity cell on $x$ gives
$$
\xymatrix{
A\ar[r]^{\mbox{\scriptsize Id}_A} \ar[d]|\bullet_{x} \ar@{}[dr]|{\mbox{\scriptsize id}_x} &A \ar[d]|\bullet^x 
	&& A\ar[dd]|\bullet_{1_A} \ar[rr]^{\mbox{\scriptsize Id}_A} \ar@{}[ddrr]|{\zeta^{-1}\cdot \mbox{\scriptsize id}_{x}} 
	&& A\ar[dd]|\bullet^{y^{-1}\cdot x}
\\
\theta(\bar{A})\ar[r]_{\mbox{\scriptsize Id}_{\theta(\bar{A})}} \ar[d]|\bullet_{x^{-1}} \ar@{}[dr]|{\zeta^{-1}} 
	& \theta(\bar{A})\ar[d]|\bullet^{y^{-1}} &=& 
\\
A\ar[r]_{w} & B && A\ar[rr]_{w} && B
}
$$
So we can define $w_*=y^{-1}\cdot x$ with binding cells $\chi_w= \mbox{id}_{y^{-1}}\cdot\zeta$
and $\psi_w=\zeta^{-1}\cdot \mbox{id}_{x}$. (It is straightforward to verify that these double cells satisfy the 
binding cell equations (\ref{compcondns}) for companions.)

Conversely, let $w\colon A\rightarrow B$ be a horizontal arrow with a companion 
$w_*\colon\xymatrix@1{A\ar[r]|\bullet &B}$
and binding cells 
$$
\xymatrix{
A\ar[d]|\bullet_{w_*}\ar@{}[dr]|{\chi_w}\ar[r]^w &B\ar[d]|\bullet^{1_B} &\ar@{}[d]|{\mbox{and}} 
	& A\ar[d]|\bullet_{1_A}\ar@{}[dr]|{\psi_w}\ar[r]^{\mbox{\scriptsize Id}_A} & A\ar[d]|\bullet^{w_*}
\\
B\ar[r]_{\mbox{\scriptsize Id}_B} & B && A\ar[r]_w & B.
}
$$
By the previous lemma it is sufficient to show that there is an invertible 2-cell
$\mbox{Id}_A \Rightarrow w$ in $\Bic\bbX$. Such a 2-cell is provided by the binding cell $\psi_w$.
\end{proof}

Now we start with a bicategory $\calB$ and consider its associated double category $\Dbl(\calB)$.
The following proposition describes the relationship between quasi units in $\calB$ and  companions in
$\Dbl(\calB)$.

\begin{prop}\label{comp-quasi}
A horizontal arrow $$\xymatrix{A_0\ar[r]^{f_1}&\cdots\ar[r]^{f_{i_0}}
  &[A_{i_0}]\ar[r]^{f_{i_0+1}}&\cdots\ar[r]^{f_{i_1}}&[A_{i_1}]\ar[r]^{f_{i_1+1}}&\cdots\ar[r]^{f_{n}}&A_n}$$
in $\Dbl(\calB)$ has a companion if and only if $A_{i_0}=A_{i_1}$ and the chosen composition 
$\varphi_{f_{i_0+1}\cdots f_{i_1}}$ is a quasi unit. 
\end{prop}

\begin{proof}
Suppose that $$\xymatrix{A_0\ar[r]^{f_1}&\cdots\ar[r]^{f_{i_0}}
  &[A_{i_0}]\ar[r]^{f_{i_0+1}}&\cdots\ar[r]^{f_{i_1}}&[A_{i_1}]\ar[r]^{f_{i_1+1}}&\cdots\ar[r]^{f_{n}}&A_n}$$
has a companion. Then there are double cells of the form
$$
\xymatrix@R=2.8em{
A_0\ar[r]^{f_1}&\cdots\ar[r]^{f_{i_0}}
  &[A_{i_0}]\ar@{=}[d] \ar@{}[1,4]|(.25)\psi\ar@{=}[rr] && [A_{i_0}]\ar@{=}[d]\ar[r]\ar[r]^{f_{i_0+1}}&\cdots\ar[r]^{f_n}&A_n
\\
A_0\ar[r]_{f_1}&\cdots\ar[r]_{f_{i_0}}
  &[A_{i_0}]\ar@/^3ex/[rr]^{\varphi_{f_{i_0+1}\cdots f_{i_1}}}\ar[r]_{f_{i_0+1}}
  &\cdots\ar[r]_{f_{i_1}}&[A_{i_1}]\ar[r]_{f_{i_1+1}}&\cdots\ar[r]_{f_{n}}&A_n
}
$$
$$
\xymatrix@R=2.8em{
A_0\ar[r]^{f_1}&\cdots\ar[r]^{f_{i_0}}
  &[A_{i_0}] \ar@/_3ex/[rr]_{\varphi_{f_{i_0+1}\cdots f_{i_1}}}\ar@{=}[d]\ar[r]^{f_{i_0+1}}&\cdots\ar[r]^{f_{i_1}}
  &[A_{i_1}]\ar@{=}[d]\ar[r]^{f_{i_1+1}}&\cdots\ar[r]^{f_{n}}&A_n
\\
A_0\ar[r]_{f_1}&\cdots\ar[r]_{f_{i_1}}
  &[A_{i_1}]\ar@{=}[rr]   \ar@{}[-1,4]|(.25)\chi &&[A_{i_1}]\ar[r]_{f_{i_1+1}}&\cdots\ar[r]_{f_{n}}&A_n\rlap{\,.}
}
$$
So $A_{i_0}=A_{i_1}$ and there are 2-cells $\psi\colon \mbox{Id}_{A_{i_0}}\Rightarrow \varphi_{f_{i_0+1}\cdots f_{i_1}}$
and $\chi\colon \varphi_{f_{i_0+1}\cdots f_{i_1}}\Rightarrow \mbox{Id}_{A_{i_1}}$
in $\calB$. The binding cell equations are equivalent to stating that these two 2-cells are inverse to each other, so 
$\varphi_{f_{i_0+1}\cdots f_{i_1}}$ is a quasi unit.

Conversely, if $A_{i_0}=A_{i_1}$ and  $\varphi_{f_{i_0+1}\cdots f_{i_1}}$ is a quasi unit with an invertible 2-cell
$\theta\colon \mbox{Id}_{A_{i_0}}\Rightarrow \varphi_{f_{i_0+1}\cdots f_{i_1}}$, then
$$\xymatrix{A_0\ar[r]^{f_1}&\cdots\ar[r]^{f_{i_0}}
  &[A_{i_0}]\ar[r]^{f_{i_0+1}}&\cdots\ar[r]^{f_{i_1}}&[A_{i_1}]\ar[r]^{f_{i_1+1}}&\cdots\ar[r]^{f_{n}}&A_n}$$
has a companion
$$
\xymatrix@R=1.5em{
A_0\ar[r]^{f_1}&\cdots\ar[r]^{f_{i_0}}
  &[A_{i_0}]\ar@{=}[d]\ar[r]^{f_{i_0+1}}&\cdots\ar[r]^{f_{n}}&A_n
\\
A_0\ar[r]^{f_1}&\cdots\ar[r]^{f_{i_1}}&[A_{i_1}]\ar[r]^{f_{i_1+1}}&\cdots\ar[r]^{f_{n}}&A_n
}
$$
with binding cells
$$
\xymatrix@R=2.8em{
A_0\ar[r]^{f_1}&\cdots\ar[r]^{f_{i_0}}
	&[A_{i_0}]\ar@{=}[d] \ar@{}[1,4]|(.25)\theta\ar@{=}[rr] 
	&& [A_{i_0}]\ar@{=}[d]\ar[r]\ar[r]^{f_{i_0+1}}&\cdots\ar[r]^{f_n}&A_n
\\
A_0\ar[r]_{f_1}&\cdots\ar[r]_{f_{i_0}}
	&[A_{i_0}]\ar@/^3ex/[rr]^{\varphi_{f_{i_0+1}\cdots f_{i_1}}}\ar[r]_{f_{i_0+1}}
	&\cdots\ar[r]_{f_{i_1}}&[A_{i_1}]\ar[r]_{f_{i_1+1}}&\cdots\ar[r]_{f_{n}}
	&A_n
}
$$
$$
\xymatrix@R=2.8em{
A_0\ar[r]^{f_1}&\cdots\ar[r]^{f_{i_0}}
  &[A_{i_0}] \ar@/_3ex/[rr]_{\varphi_{f_{i_0+1}\cdots f_{i_1}}}\ar@{=}[d]\ar[r]^{f_{i_0+1}}&\cdots\ar[r]^{f_{i_1}}
  &[A_{i_1}]\ar@{=}[d]\ar[r]^{f_{i_1+1}}&\cdots\ar[r]^{f_{n}}&A_n
\\
A_0\ar[r]_{f_1}&\cdots\ar[r]_{f_{i_1}}
  &[A_{i_1}]\ar@{=}[rr]   \ar@{}[-1,4]|(.25){\theta^{-1}} &&[A_{i_1}]\ar[r]_{f_{i_1}+1}&\cdots\ar[r]_{f_{n}}&A_n
}
$$
\end{proof}

\subsection{Preservation of companions}
It is clear that strict functors between double categories preserve companions and conjoints.
It is not immediately obvious that the same is true for pseudo-functors, given the fact that horizontal identities
and domains and codomains are not preserved strictly.

However, since companions in weakly globular 
double categories correspond to quasi units in bicategories, and we know that homomorphisms of bicategories 
preserve quasi units, pseudo-functors between weakly globular double categories should preserve companions.
Here we give a direct proof of this result
in terms of companions and binding cells.

\begin{prop}\label{comp-pres}
Let $F\colon \bbX\rightarrow\bbY$ be a pseudo-functor of weakly globular double categories.
If a horizontal arrow $f $ in $\bbX$ has a companion then so does $F(f)$ in $\bbY$.
\end{prop}

\begin{proof}
 Let $\xymatrix@1{A\ar[r]^f&B}$ be a horizontal arrow with vertical 
companion $\xymatrix@1{A\ar[r]|\bullet^{v}&B}$ and binding cells
$$
\xymatrix{
A\ar@{=}[d]\ar@{=}[r]\ar@{}[dr]|\psi & A\ar[d]|\bullet^v 
	& A\ar[d]|\bullet_v \ar[r]^f  \ar@{}[dr]|\chi & B\ar@{=}[d]
\\
A\ar[r]_f & B&B\ar@{=}[r] & B\rlap{\quad.}
}$$
Consider the following pastings of double cells in $\bbY$:
$$
\xymatrix@C=4em{
& d_0F_1f \ar[r]^{\mathrm{Id}_{d_0F_1f}} \ar[d]|\bullet_{(\varphi_0)_{\mathrm{Id}_A}\cdot(d_0F_1\psi)^{-1}}
	    \ar@{}[dr]|{\mathrm{id}} 
& d_0F_1f  \ar[d]|\bullet^{(\varphi_0)_{\mathrm{Id}_A}\cdot(d_0F_1\psi)^{-1}}
\\
\ar@{}[d]|{\textstyle\psi_{F_1f}:=}&F_0A\ar[r]_{\mathrm{Id}_{F_0A}}\ar[d]|\bullet_{(\varphi_0)_{\mathrm{Id}_A}^{-1}}\ar@{}[dr]|{\sigma_A}
&F_0A\ar[d]|\bullet^{(\varphi_1)_{\mathrm{Id}_A}^{-1}}
\\
&d_0F_1(\mathrm{Id}_A) \ar[d]|\bullet_{d_0F_1\psi} \ar@{}[dr]|{F_1\psi} \ar[r]_{F_1(\mathrm{Id}_A)}
& d_1F_1(\mathrm{Id}_A) \ar[d]|\bullet^{d_1F_1\psi} 
\\
&d_0F_1f\ar[r]_{F_1f} & d_1F_1f
}
$$
and
$$
\xymatrix@C=4em{
& d_0F_1f \ar[r]^{F_1f} \ar[d]|\bullet_{d_0F_1\chi} \ar@{}[dr]|{F_1\chi} 
& d_1F_1f  \ar[d]|\bullet^{d_1F_1\chi}
\\
\ar@{}[d]|{\textstyle\chi_{F_1f}:=}
&d_0F_1(\mathrm{Id}_B)\ar[r]_{F_1(\mathrm{Id}_{B})}\ar[d]|\bullet_{(\varphi_0)_{\mathrm{Id}_B}}\ar@{}[dr]|{\sigma_B^{-1}}
&d_1F_1(\mathrm{Id}_B)\ar[d]|\bullet^{(\varphi_1)_{\mathrm{Id}_B}}
\\
&F_0B\ar[d]|\bullet_{(d_1F_1\chi)^{-1}\cdot(\varphi_1)^{-1}_{\mathrm{Id}_B}} \ar@{}[dr]|{\mathrm{id}} 
	    \ar[r]_{\mathrm{Id}_{F_0B}}
&F_0B\ar[d]|\bullet^{(d_1F_1\chi)^{-1}\cdot(\varphi_1)_{\mathrm{Id}_B}^{-1}} 
\\
&d_1F_1f\ar[r]_{\mathrm{Id}_{d_1F_1f}} & d_1F_1f\rlap{\quad.}
}
$$
Note that $d_1F_1\psi\cdot(\varphi_1)^{-1}_{\mathrm{Id}_A}\cdot(\varphi_0)_{\mathrm{Id}_A}\cdot(d_0F_1\psi)^{-1}=
(d_1F_1\chi)^{-1}\cdot(\varphi_1)^{-1}_{\mathrm{Id}_B}\cdot(\varphi_0)_{\mathrm{Id}_B}\cdot d_0F_1\chi$
since the vertical arrow category is posetal, and we claim that this arrow is the vertical 
companion of $F_1f$ with binding cells $\psi_{F_1f}$ and $\chi_{F_1f}$.
So we need to check that the equations in (\ref{compcondns}) are satisfied.
It is relatively easy to check that the second equation, involving the vertical composition of these cells, 
is satisfied. For the first equation we need to refer to the coherence conditions in 
(\ref{comp-nat}), (\ref{ass-id1}), and (\ref{ass-id2}).
The horizontal composition considered in the binding cell equation is
$$
\xymatrix@C=4em{
d_0F_1f \ar[r]^{\mathrm{Id}_{d_0F_1f}} \ar[d]|\bullet_{(\varphi_0)_{\mathrm{Id}_A}\cdot(d_0F_1\psi)^{-1}}
	    \ar@{}[dr]|{\mathrm{id}} 
& d_0F_1f  \ar[d]|\bullet^(.35){(\varphi_0)_{\mathrm{Id}_A}\cdot(d_0F_1\psi)^{-1}}\ar@{=}[r] \ar@{}[dddr]|{=} 
& d_0F_1f \ar[r]^{F_1f} \ar[d]|\bullet_(.6){d_0F_1\chi}
	    \ar@{}[dr]|{F_1\chi} 
& d_0F_1f  \ar[d]|\bullet^{d_1F_1\chi}
\\
FA\ar[r]|{\mathrm{Id}_{F_0A}}\ar[d]|\bullet_{(\varphi_0)_{\mathrm{Id}_A}^{-1}}\ar@{}[dr]|{\sigma_A}
&FA\ar[d]|\bullet^{(\varphi_1)_{\mathrm{Id}_A}^{-1}} 
& d_0F_1(\mathrm{Id}_B) \ar[r]|{F_1(\mathrm{Id}_{B})} \ar[d]|\bullet_{(\varphi_0)_{\mathrm{Id}_B}} \ar@{}[dr]|{\sigma_B^{-1}}
&d_1F_1(\mathrm{Id}_B)\ar[d]|\bullet^{(\varphi_1)_{\mathrm{Id}_B}}
\\
d_0F_1(\mathrm{Id}_A) \ar[d]|\bullet_{d_0F_1\psi} \ar@{}[dr]|{F_1\psi} \ar[r]|{F_1(\mathrm{Id}_A)}
& d_1F_1(\mathrm{Id}_A) \ar[d]|\bullet^(.35){d_1F_1\psi} 
& FB\ar[d]|\bullet_(.6){(d_1F_1\chi)^{-1}\cdot(\varphi_1)^{-1}_{\mathrm{Id}_B}} \ar@{}[dr]|{\mathrm{id}} 
	    \ar[r]|{\mathrm{Id}_{F_0B}}
&FB\ar[d]|\bullet^{(d_1F_1\chi)^{-1}\cdot(\varphi_1)_{\mathrm{Id}_B}^{-1}} 
\\
d_0F_1f\ar[r]_{F_1f} & d_1F_1f\ar@{=}[r]&d_1F_1f\ar[r]_{\mathrm{Id}_{d_1F_1f}} & d_1F_1f\rlap{\quad.}
}$$
By (\ref{ass-id1}), and (\ref{ass-id2})
this is equal to
$$
\xymatrix@C=6em@R=1.8em{
 \ar[rrr]^{F_1(f\circ \mbox{\scriptsize Id}_A)} \ar@{}[drrr]|{\mu_{f,\mathrm{Id}_A}} \ar[d]|\bullet 
	&&& \ar[d]|\bullet 
\\
\ar[d]|\bullet\ar[r]^{\pi_2(F_2(f,\mathrm{Id}))}\ar@{}[dr]|{(\theta_2)_{f,\mathrm{Id}_A}}
	&\ar[d]|\bullet\ar@{=}[r] \ar@{}[dddr]|{=}
	& \ar[r]^{\pi_1(F_2(f,\mathrm{Id}_A))} \ar@{}[dddr]|{(\theta_1)_{f,\mathrm{Id}_A}}\ar[ddd]|\bullet 
	&\ar[ddd]|\bullet 
\\
\ar[r]|{F_1(\mbox{\scriptsize Id}_A)} \ar[d]|\bullet \ar@{}[dr]|{\sigma_A^{-1}}&\ar[d]|\bullet  &&
\\
\ar[d]|\bullet \ar[r]|{\mbox{\scriptsize Id}_{F_0A}} \ar@{}[dr]|{\mbox{\scriptsize id}} &\ar[d]|\bullet &&
\\
\ar[r]|{\mathrm{Id}_{d_0F_1f}} \ar[d]|\bullet
	    \ar@{}[dr]|{\mathrm{id}} 
&  \ar[d]|\bullet\ar@{=}[r] \ar@{}[dddr]|{=} 
&  \ar[r]|{F_1f} \ar[d]|\bullet
	    \ar@{}[dr]|{F_1\chi} 
&  \ar[d]|\bullet
\\
\ar[r]|{\mathrm{Id}_{F_0A}}\ar[d]|\bullet\ar@{}[dr]|{\sigma_A}
&\ar[d]|\bullet
& \ar[r]|{F_1(\mathrm{Id}_{B})}\ar[d]|\bullet\ar@{}[dr]|{\sigma_B}
&\ar[d]|\bullet
\\
 \ar[d]|\bullet \ar@{}[dr]|{F_1\psi} \ar[r]|{F_1(\mathrm{Id}_A)}
&  \ar[d]|\bullet
& \ar[d]|\bullet \ar@{}[dr]|{\mathrm{id}} 
	    \ar[r]|{\mathrm{Id}_{F_0B}}
&\ar[d]|\bullet
\\
\ar[r]|{F_1f} \ar[ddd]|\bullet\ar@{}[dddr]|{(\theta_2)^{-1}_{\mathrm{Id}_B,f}}
& \ar[ddd]|\bullet \ar@{=}[r]\ar@{}[dddr]|= 
&\ar[r]|{\mathrm{Id}_{d_1F_1f}}\ar[d]|\bullet \ar@{}[dr]|{\mathrm{id}} 
& \ar[d]|\bullet
\\
&&\ar[r]|{\mathrm{Id}_{F_0B}}\ar[d]|\bullet \ar@{}[dr]|{\sigma_B^{-1}} &\ar[d]|\bullet
\\
&&\ar[r]|{F_1(\mathrm{Id}_{B})}\ar[d]|\bullet\ar@{}[dr]|{(\theta_1)^{-1}_{\mathrm{Id}_B,f}}
&\ar[d]|\bullet 
\\
\ar[d]|\bullet \ar[r]_{\pi_2F_2(\mathrm{Id}_B,f)} \ar@{}[drrr]|{\mu^{-1}_{\mathrm{Id}_B,f}}
&\ar@{=}[r]& \ar[r]_{\pi_1F_2(\mathrm{Id}_B,f)} & \ar[d]|\bullet
\\
\ar[rrr]_{F_1(f)} &&&
\rlap{\quad.}
}$$
Cancellation of vertical inverses gives that this is equal to
$$
 \xymatrix@C=4em{
\ar[rrr]^{F_1(f\circ \mathrm{Id}_A)}\ar[d]|\bullet \ar@{}[drrr]|{\mu_{f,\mathrm{Id}_A}} 
	&&& \ar[d]|\bullet
\\
\ar[r]^{\pi_2F_2(f,\mathrm{Id}_A)} \ar[d]|\bullet \ar@{}[dr]|{(\theta_2)_{f,\mathrm{Id}_A}} 
	&\ar[d]|\bullet \ar@{=}[r] \ar@{}[dddr]|=
	&\ar[r]^{\pi_1F_2(f,\mathrm{Id}_A)} \ar[d]|\bullet_\sim \ar@{}[dr]|{(\theta_1)_{f,\mathrm{Id}_A}} 
	&\ar[d]|\bullet
\\
\ar[r]|{F_1\mathrm{Id}_A}\ar[d]|\bullet \ar@{}[dr]|{F_1\psi} & \ar[d]|\bullet 
	& \ar[d]|\bullet\ar@{}[dr]|{F_1\chi}\ar[r]|{F_1f} & \ar[d]|\bullet \ar@{}[drr]|{\textstyle =}
	&& \ar[r]^{F_1(f)}\ar@{}[dr]|{F_1(\chi\circ\psi)} \ar[d]|\bullet &\ar[d]|\bullet
\\
\ar[d]|\bullet \ar[r]|{F_1f} \ar@{}[dr]|{(\theta_2)_{\mathrm{Id}_B,f}^{-1}} &\ar[d]|\bullet
	& \ar[d]|\bullet \ar[r]|{F_1\mathrm{Id}_B} \ar@{}[dr]|{(\theta_1)_{\mathrm{Id}_B,f}^{-1}} 
	&\ar[d]|\bullet
	&& \ar[r]_{F_1(\mathrm{Id}_B\circ f)} &
\\
\ar[r]_{\pi_2F_2(\mathrm{Id}_B,f)} \ar[d]|\bullet \ar@{}[drrr]|{\mu_{\mathrm{Id}_B,f}^{-1}}
	& \ar@{=}[r] & \ar[r]_{\pi_1F_2(\mathrm{Id}_B,f)} &\ar[d]|\bullet
\\
\ar[rrr]_{F_1(\mathrm{Id}_B\circ f)} &&&
}
$$
where the last equality is by (\ref{comp-nat}).
This gives us the required identity double cell.
\end{proof}

\begin{rmk}{\rm
The proof of this proposition uses the fact that we are working with weakly globular double categories rather than arbitary double categories 
in an essential way.
In fact, the result is not true in this generality for arbitrary double categories: it was shown in \cite{DPP-spans2} 
that pseudo-functors between arbitrary double categories preserve companions if and only if they are normal.
}\end{rmk}

\subsection{Companion categories}
The family of horizontal/vertical arrows that have a vertical/horizontal companion is closed under horizontal/vertical composition and
the binding cells of the composite can be expressed as composites  of the binding cells of the individual arrows.
So for an arbitrary double category we can construct a category of companions. This category can furthermore be viewed as the 
category of vertical arrows of a double category which has the same horizontal arrows as the original double category, but its cells are only those that 
interact well with the binding cells of the companion pairs. This is made precise in the following definition.

\begin{dfn}\rm
Let $\bbD$ be an arbitrary double category. Then $\bbComp(\bbD)$, the {\em double category 
of companions in $\bbD$}, is defined as follows. It has the 
same objects and horizontal arrows as $\bbD$, but its vertical arrows are companion pairs 
(with their binding cells)
in $\bbD$. So a vertical arrow $\theta\colon\xymatrix@1{A\ar[r]|\bullet & B}$ in $\bbComp(\bbD)$ is given 
by a quadruple $\theta=(h_\theta,v_\theta,\psi_\theta,\chi_\theta)$, where $v_\theta$ and $h_\theta$ 
are companions with binding cells
$$
\xymatrix{
A\ar[r]^{\mathrm{Id}_A} \ar@{}[dr]|{\psi_\theta} \ar[d]|\bullet_{1_A} & A\ar[d]|\bullet^{v_\theta}\ar@{}[drr]|{\mbox{and}}
    &&  A\ar[d]|\bullet_{v_\theta} \ar[r]^{h_\theta} \ar@{}[dr]|{\chi_\theta}& B\ar[d]|\bullet^{1_A}
\\
A\ar[r]_{h_\theta} & B && B\ar[r]_{\mathrm{Id}_A}  &B.
}$$
The vertical identity arrow $1_A$ is given by $(\mathrm{Id}_A,1_A, \iota_A,\iota_A)$.
Vertical composition is defined by $\theta'\cdot\theta=(h_{\theta'}h_{\theta},v_{\theta'}\cdot v_\theta, 
(\psi_{\theta'}\circ 1_{h_\theta})\cdot(\mbox{id}_{v_{\theta'}}\circ\psi_\theta),
(\chi_{\theta'}\circ\mbox{id}_{v_\theta})\cdot(1_{h_{\theta'}}\circ\chi_\theta))$.
A double cell 
$$
\xymatrix{
A\ar[r]^f\ar[d]|\bullet_\theta\ar@{}[dr]|\Theta & A'\ar[d]|\bullet^{\theta'}
\\
B\ar[r]_g&B'
}
$$
in $\bbComp(\bbD)$ consists of a double cell
$$
\xymatrix@R=2em{
A\ar[r]^f\ar[d]|\bullet_{v_\theta}\ar@{}[dr]|\Theta & A'\ar[d]|\bullet^{v_{\theta'}}
\\
B\ar[r]_g&B'
}
$$ 
in $\bbD$ (involving just the vertical parts of the companion pairs in the cell we are defining) 
which has the following properties:
\begin{enumerate}
 \item the square of horizontal arrows 
$$
\xymatrix@R=2em{
A\ar[d]_{h_\theta} \ar[r]^f& A'\ar[d]^{h_{\theta'}}
\\
B\ar[r]_g & B'
}$$
commutes in $\bbD$;
\item 
$$
\xymatrix@C=3.5em@R=2em{
A\ar[r]^f\ar[d]|\bullet_{v_\theta}\ar@{}[dr]|\Theta 
    & A'\ar[d]|\bullet^{v_{\theta'}}\ar@{}[dr]|{\chi_{\theta'}} \ar[r]^{h_{\theta'}} & B' \ar@{=}[d]
    \ar@{}[dr]|{\textstyle =} & A\ar[d]|\bullet_{v_{\theta}}\ar[r]^{h_\theta}\ar@{}[dr]|{\chi_\theta}
    & B\ar@{=}[d] \ar[r]^g\ar@{}[dr]|{1_g} & B'\ar@{=}[d]
\\
B\ar[r]_g&B' \ar@{=}[r] & B' & B\ar@{=}[r] &B\ar[r]_g & B'
}
$$ 
\item
$$
\xymatrix@C=3.5em@R=2em{
A\ar@{=}[r]\ar@{=}[d] \ar@{}[dr]|{\psi_\theta}
    & A\ar[d]|\bullet^{v_{\theta}}\ar@{}[dr]|{\Theta} \ar[r]^{f} & A' \ar[d]|\bullet^{v_{\theta'}}
    \ar@{}[dr]|{\textstyle =} & A\ar@{=}[d] \ar[r]^{f}\ar@{}[dr]|{1_f}
    & A'\ar@{=}[d] \ar@{=}[r]\ar@{}[dr]|{\psi_\theta} & B'\ar[d]|\bullet^{v_{\theta'}}
\\
A\ar[r]_{h_{\theta}}&B \ar[r]_g & B' & A\ar[r]_f &A'\ar[r]_{h_{\theta'}} & B'
}
$$ 
\end{enumerate}
\end{dfn}

\begin{rmks}\rm
\begin{enumerate}
 \item We write $\Comp(\bbD)$ for the category of vertical arrows, $v\bbComp(\bbD)$, i.e., 
the category of companion pairs (with chosen binding cells).
\item If we have a functorial choice of companions and binding cells  and we only use the 
horizontal arrows that are taken in these choices in the construction of $\bbComp(\bbD)$,
the result is a double category with a thin structure as defined in \cite{BM}.
\end{enumerate}
\end{rmks}

\subsection{Precompanions}\label{pre-comps}
We saw before that companions in weakly globular double categories correspond to quasi units in 
bicategories. 
We now introduce a class of horizontal arrows in weakly globular double categories that will correspond to  internal
equivalences in bicategories.

\begin{dfn}\label{precomp}\rm
Let $\bbX$ be a weakly globular double category.
\begin{enumerate}
 \item A horizontal arrow $\xymatrix@1{A\ar[r]^f&B}$  in $\bbX$ is a {\em left precompanion}
if there are horizontal arrows $\xymatrix@1{A'\ar[r]^{f'} &B'}$ and $\xymatrix@1{B'\ar[r]^{r_f}&C}$
with a vertically invertible double cell
$$
\xymatrix{
A\ar[d]|\bullet\ar@{}[dr]|\varphi \ar[r]^f&B\ar[d]|\bullet\\
A'\ar[r]_{f'}&B'}
$$
such that $r_f\circ f'$ is a companion in $\bbX$.
\item
Dually, a horizontal arrow $\xymatrix@1{A\ar[r]^f&B}$  in $\bbX$ is a {\em right precompanion}
if there are horizontal arrows $\xymatrix@1{A''\ar[r]^{f''} &B''}$ and $\xymatrix@1{D\ar[r]^{l_f}&A''}$
with a vertically invertible double cell
$$
\xymatrix{
A\ar[d]|\bullet\ar@{}[dr]|{\varphi'} \ar[r]^f&B\ar[d]|\bullet\\
A''\ar[r]_{f''}&B''}
$$
such that $f''\circ l_f$ is a companion in $\bbX$.  
\item
A horizontal arrow $\xymatrix@1{A\ar[r]^f&B}$  in $\bbX$ is a {\em precompanion}
if it is both a left and a right precompanion.
\end{enumerate}
\end{dfn}

\begin{eg}{\rm
Every horizontal isomorphism $f\colon A\to B$ is a precompanion with $f'=f=f''$ and $l_f=r_f=f^{-1}$.}
\end{eg}

We first prove some basic properties of precompanions.

\begin{lma}\label{lrprecomps}
If a horizontal arrow $\xymatrix@1{A\ar[r]^f&B}$ is a precompanion in a weakly globular double category, 
then there exist horizontal arrows
$\xymatrix@1{D\ar[r]^{\bar{l}_f}&A'}$, $\xymatrix@1{A'\ar[r]^{\bar{f}} &B'}$ and $\xymatrix@1{B'\ar[r]^{\bar{r}_f}&C}$
with a vertically invertible double cell
$$
\xymatrix{
A\ar[d]|\bullet\ar@{}[dr]|\psi \ar[r]^f&B\ar[d]|\bullet\\
A'\ar[r]_{\bar{f}}&B'}
$$
such that $\bar{r}_f\circ \bar{f}$ is a companion and $\bar{f}\circ \bar{l}_f$ is a companion.
\end{lma}

\begin{proof}
Since $f$ is both a left and a right precompanion there are horizontal arrows $l_f$, $r_f$, $f'$ and $f''$ with double cells 
$\varphi$ and $\varphi'$ as in Definition \ref{precomp}. These can be combined in the following diagram.
$$
\xymatrix{
&A'\ar[d]|\bullet\ar@{}[dr]|\varphi\ar[r]^{f'} & B'\ar[d]|\bullet\ar[r]^{r_f} &C
\\
&A\ar[d]|\bullet\ar@{}[dr]|{\varphi'}\ar[r]^f & B\ar[d]|\bullet
\\
D\ar[r]_{l_f}&A''\ar[r]_{f''}&B''
}
$$
By the weak globularity condition,
we can complete this staircase diagram with arrows and vertically invertible cells as in
the following diagram,
$$
\xymatrix{
\bar{D}\ar[3,0]|\bullet \ar[r]^{\bar{l}_f}\ar@{}[3,1]|{\gamma_1} & \bar{A}\ar[r]^{\bar{f}}\ar[d]|\bullet\ar@{}[dr]|{\gamma_2} 
		& \bar{B} \ar[d]|\bullet\ar@{}[dr]|{\gamma_3}\ar[r]^{\bar{r}_f} & \bar{C}\ar[d]|\bullet
\\
&A'\ar[d]|\bullet\ar@{}[dr]|\varphi\ar[r]^{f'} & B'\ar[d]|\bullet\ar[r]^{r_f} &C
\\
&A\ar[d]|\bullet\ar@{}[dr]|{\varphi'}\ar[r]^f & B\ar[d]|\bullet
\\
D\ar[r]_{l_f}&A''\ar[r]_{f''}&B''		
}
$$
Now the composites $\bar{f}\circ \bar{l}_f$ and $\bar{r}_f\circ \bar{f}$
are both companions, since they are vertically isomorphic to companions.
We will give an explicit proof in terms of binding cells for one of them.
Let the binding cells for $ f'\circ r_f$ with its companion be as in the following diagram.
$$
\xymatrix{
A'\ar@{=}[rr]\ar@{=}[d]\ar@{}[drr]|{\psi_f^r} && A'\ar[d]|\bullet&A'\ar[d]|\bullet \ar[r]^{f'} \ar@{}[drr]|{\chi_f^r} & B'\ar[r]^{r_f} & C\ar@{=}[d]
\\
A'\ar[r]_{f'}&B'\ar[r]_{r_f} & C & C\ar@{=}[rr] && C
}
$$
Then the binding cells for $\bar{f}\circ \bar{r}_f$ with its companion are
$$
\xymatrix{
\bar{A}\ar[d]|\bullet\ar@{=}[rr]\ar@{}[drr]|{\mbox{\scriptsize id}} && \bar{A}\ar[d]|\bullet 
		& \bar{A}\ar[d]|\bullet\ar@{}[dr]|{\gamma_2}\ar[r]^{\bar{f}} 
		& \bar{B}\ar[d]|\bullet\ar@{}[dr]|{\gamma_3} \ar[r]^{\bar{r}_f} & \bar{C}\ar[d]|\bullet
\\
A'\ar@{=}[rr]\ar@{=}[d]\ar@{}[drr]|{\psi_f^r} && A'\ar[d]|\bullet 
		&A'\ar[d]|\bullet \ar[r]^{f'} \ar@{}[drr]|{\chi_f^r} & B'\ar[r]^{r_f} & C\ar@{=}[d]
\\
A'\ar[d]|\bullet \ar@{}[dr]|{\gamma_2^{-1}}\ar[r]_{f'}& B'\ar[r]_{r_f}\ar@{}[dr]|{\gamma_3^{-1}} \ar[d]|\bullet& C \ar[d]|\bullet
		& C\ar[d]|\bullet\ar@{}[drr]|{\mbox{\scriptsize id}}\ar@{=}[rr] && C\ar[d]|\bullet
\\
\bar{A}\ar[r]_{\bar{f}}&\bar{B}\ar[r]_{\bar{r}_f} & \bar{C} & \bar{C}\ar@{=}[rr] && \bar{C}.
}
$$
The rest of the details are left to the reader.
\end{proof}

\begin{lma}\label{lr}
 If $f$ is a precompanion with arrows $l_f$ and $r_f$ as in Definition \ref{precomp}, then 
there is a vertically invertible double cell
of the form
$$
\xymatrix@R=3em@C=3em{
\ar[d]|\bullet\ar@{}[dr]|{\nu_f}\ar[r]^{r_f} & \ar[d]|\bullet
\\
\ar[r]_{l_f} &\rlap{\quad.}
}
$$
\end{lma}

\begin{proof}
 Let $r_f\circ f'$ have vertical companion $v^r_f$ with binding cells
$$
\xymatrix{
\ar@{=}[d]\ar@{=}[rr]\ar@{}[drr]|{\psi_f^r} && \ar[d]^{v^r_f}|\bullet \ar@{}[drr]|{\mbox{and}} 
	&& \ar[r]^{f'}\ar[d]|\bullet_{v^r_f} \ar@{}[drr]|{\chi_f^r} & \ar[r]^{r_f} & \ar@{=}[d]
\\
\ar[r]_{f'}&\ar[r]_{r_f} & && \ar@{=}[rr] &&\rlap{\quad,}
}
$$
and let $f''\circ l_f$ have vertical companion $v^l_f$ with binding cells
$$
\xymatrix{
\ar@{=}[d]\ar@{=}[rr]\ar@{}[drr]|{\psi_f^l} && \ar[d]^{v^l_f}|\bullet \ar@{}[drr]|{\mbox{and}} 
	&& \ar[r]^{l_f}\ar[d]|\bullet_{v^l_f} \ar@{}[drr]|{\chi_f^l} & \ar[r]^{f''} & \ar@{=}[d]
\\
\ar[r]_{l_f}&\ar[r]_{f''} & && \ar@{=}[rr] &&\rlap{\quad.}
}
$$
Further, let $x=(v_f^l)^{-1}\cdot d_1\varphi'\cdot(d_1\varphi)^{-1}$
and $y=(d_0\varphi')\cdot(d_0\varphi)^{-1}\cdot (v^r_f)^{-1}$.
Then $\nu_f$ can be obtained 
as the following pasting of double cells:
$$\xymatrix@C=4em@R=2em{
\ar[d]|\bullet_x\ar@{=}[rr]\ar@{}[drr]|{\mathrm{id}} && \ar[d]|\bullet_x\ar[r]^{r_f}\ar@{}[4,1]|{1_{r_f}} & \ar@{=}[4,0]
\\
\ar@{=}[d]\ar@{=}[rr]\ar@{}[drr]|{\psi_f^l} && \ar[d]|\bullet^{v_f^l}
\\
\ar[r]|{l_f}\ar@{=}[4,0]\ar@{}[4,1]|{1_{l_f}} & \ar[d]|\bullet_{d_0\varphi'^{-1}}\ar[r]|{f''}\ar@{}[dr]|{\varphi'^{-1}} 
	& \ar[d]|\bullet^{d_1\varphi'^{-1}}
\\
& \ar[d]|\bullet_{d_0\varphi}\ar[r]|{f}\ar@{}[dr]|\varphi & \ar[d]|\bullet^{d_1\varphi}
\\
&\ar[d]_{v^r_f}\ar@{}[drr]|{\chi^r_f}\ar[r]|{f'} & \ar[r]|{r_f} &\ar@{=}[d]
\\
&\ar[d]|\bullet_y\ar@{=}[rr]\ar@{}[drr]|{\mathrm{id}} &&\ar[d]|\bullet^y
\\
\ar[r]_{l_f} & \ar@{=}[rr]&&\rlap{\quad .}
}$$
\end{proof}

It is well-known that companions, just like adjoints, are unique up to special invertible double cells.
A similar result applies to precompanions and the proof is a 2-dimensional version of the proof
that the pseudo-inverse of an equivalence in a bicategory is unique up to invertible 2-cell.
 
\begin{lma}\label{pre-comp-uniq}
 For any precompanion $f$ in a weakly globular double category, the
precompanion structure given in Definition \ref{precomp} is unique
up to vertically invertible double cells.
\end{lma}

\begin{proof}
 Suppose that $\varphi_i$, $r_{f,i}$, $v_{f,i}^r$, $f'_i$, $\psi_{f,i}^r$, and $\chi_{f,i}^r$ give two 
 right precompanion structures for $i=1,2$,
and $\varphi'_i$,
$l_{f,i}, f''_i, v_{f,i}^l$, $\psi_{f,i}^l$ and $\chi_{f,i}^l$ give two left precompanion structures for $i=1,2$.
Then there are vertically invertible double cells as follows:
$$
\xymatrix@C=4em{
\ar[r]^{f'_1}\ar[d]|\bullet\ar@{}[dr]|{\varphi_1^{-1}} & \ar[d]|\bullet &\ar[d]|\bullet\ar[r]^{f''_1}\ar@{}[dr]|{(\varphi'_1)^{-1}}& \ar[d]|\bullet
\\
\ar[r]|f \ar[d]|\bullet \ar@{}[dr]|{\varphi_2} & \ar[d]|\bullet & \ar[r]|f  \ar[d]|\bullet \ar@{}[dr]|{\varphi'_2} & \ar[d]|\bullet
\\
\ar[r]_{f'_2} & & \ar[r]_{f''_2} &
}$$
According to Lemma \ref{lr} there are vertically invertible double cells
$$
\xymatrix{
\ar[d]|\bullet\ar[r]^{r_{f,1i}}\ar@{}[dr]|{\nu_{f,ij}} &
\ar[d]|\bullet \ar@{}[drrr]|{\textstyle \mbox{for }i,j\in\{1,2\}.}
\\
\ar[r]_{l_{f,j}}&&&&
}$$
So we obtain vertically invertible cells
$$
\xymatrix@C=5em{
\ar[r]^{r_{f,1}} \ar[d]|\bullet \ar@{}[dr]|{\nu_{f,21}^{-1}\cdot\nu_{f,11}} 
	&\ar[d]|\bullet\ar@{}[dr]|{\mbox{and}} & \ar[r]^{l_{f,1}}\ar[d]|\bullet \ar@{}[dr]|{\nu_{f,22}\cdot\nu_{f,21}^{-1}}& \ar[d]|\bullet
\\
\ar[r]_{r_{f,2}} && \ar[r]_{l_{f,2}}&
}$$
\end{proof}

The next two propositions establish the relationship between 
precompanions in weakly globular double categories and equivalences in bicategories.

\begin{prop}\label{D:precomp-equiv}
 Let $\calB$ be a bicategory.
A horizontal arrow
\begin{equation}\label{arrow}
\xymatrix{A_0\ar[r]^{f_1}&\cdots\ar[r]^{f_{i_0}}
  &[A_{i_0}]\ar[r]^{f_{i_0+1}}&\cdots\ar[r]^{f_{i_1}}&[A_{i_1}]\ar[r]^{f_{i_1}+1}&\cdots\ar[r]^{f_{n}}&A_n}
\end{equation}
in $\Dbl(\calB)$ is a precompanion if and only if the chosen composition $\varphi_{f_{i_0+1}\cdots f_{i_1}}$
is an equivalence in $\calB$.
\end{prop}

\begin{proof}
Suppose that the chosen composition $\varphi_{f_{i_0+1}\cdots f_{i_1}}$ 
is an equivalence in $\calB$. 
Denote the arrow (\ref{arrow}) in $\Dbl(\calB)$ by $f$.
Let $g\colon A_{i_1}\rightarrow A_{i_0}$ be a pseudo-inverse of $\varphi_{f_{i_0+1}\cdots f_{i_1}}$ in $\calB$.
Then we may take $f'$ in Definition \ref{precomp} to be the arrow
$$
\xymatrix@C=5em{
[A_{i_0}]\ar[r]^{\varphi_{f_{i_0+1}\cdots f_{i_1}}}&[A_{i_1}]\ar[r]^{g}&A_{i_0}
}$$
with vertically invertible cell 
$$
\xymatrix{
A_0\ar[r]^{f_1} & \cdots \ar[r]^{f_{i_0}}&[A_{i_0}]\ar@{=}[d] \ar@/_2ex/[rr]_{\varphi_{f_{i_0+1}\cdots f_{i_1}}}\ar[r]^{f_{i_0+1}} 
			    & \cdots \ar[r]^{f_{i_1}}\ar@{}[d]|(.8){=}
			    &[A_{i_1}] \ar[r]^{f_{i_1+1}}\ar@{=}[d] &\cdots\ar[r]^{f_{i_n}}&A_n
\\
&& [A_{i_0}]\ar[rr]_{\varphi_{f_{i_0+1}\cdots f_{i_1}}} && [A_{i_1}]\ar[rr]_{g} && A_{i_0}
}$$
and we may take $r_f$ in Definition \ref{precomp} to be the arrow
$$
\xymatrix@C=4em{
A_{i_0}\ar[r]^{\varphi_{f_{i_0+1}\cdots f_{i_1}}}&[A_{i_1}]\ar[r]^{g}&[A_{i_0}]\rlap{\,.}
}$$
The horizontal composition $r_f\circ f'$ is given by
$$
\xymatrix@C=4em{
[A_{i_0}]\ar[r]^{\varphi_{f_{i_0+1}\cdots f_{i_1}}}&A_{i_1}\ar[r]^{g}&[A_{i_0}]\rlap{\,,}
}$$
and this is a companion by Proposition \ref{comp-quasi}, since $g\circ \varphi_{f_{i_0+1}\cdots f_{i_1}}$ 
is a quasi unit in $\calB$.
So (\ref{arrow}) is a left precompanion.
The proof that it is a right precompanion goes similarly.

Now suppose that (\ref{arrow}) is a precompanion in $\Dbl(\calB)$.
So there are diagrams with vertically invertible double cells of the form
$$
  \xymatrix@R=4em{
A_0\ar[r]^{f_1} & \cdots \ar[r]^{f_{i_0}}
      &[A_{i_0}]\ar@/_2ex/[rr]_{\varphi_{f_{i_0+1}\cdots f_{i_1}}}\ar@{=}[d]\ar[r]^{f_{i_0+1}} 
      & \cdots \ar[r]^{f_{i_1}}\ar@{}[d]|{\varphi}
      &[A_{i_1}]\ar@{=}[d]\ar[r]^{f_{i_1+1}}&\cdots\ar[r]^{f_{n}}&A_n
\\
B_0\ar[r]_{g_1} &\cdots\ar[r]_{g_{j_0}} &[A_{i_0}]\ar@/^2ex/[rr]^{\varphi_{g_{j_0+1}\cdots g_{j_1}}}
	    \ar[r]_{g_{j_0+1}} & \cdots\ar[r]_{g_{j_1}} & [A_{i_1}]\ar[r]_{g_{j_1+1}} 
      & \cdots\ar[r]_{g_{j_2}} & C \ar[r]_{g_{j_2+1}}&\cdots\ar[r]_{g_{m}} & B_{m}
}$$
and
$$
  \xymatrix@R=4em{
&&A_0\ar[r]^{f_1} & \cdots \ar[r]^{f_{i_0}}
      &[A_{i_0}]\ar@/_2ex/[rr]_{\varphi_{f_{i_0+1}\cdots f_{i_1}}}\ar@{=}[d]\ar[r]^{f_{i_0+1}} 
      & \cdots \ar[r]^{f_{i_1}}\ar@{}[d]|{\varphi'}
      &[A_{i_1}]\ar@{=}[d]\ar[r]^{f_{i_1+1}}&\cdots\ar[r]^{f_{n}}&A_n
\\
B'_0\ar[r]_{h_1} &\cdots\ar[r]_{h_{k_0}} & D\ar[r]_{h_{k_0+1}} & \cdots \ar[r]_{h_{k_1}} 
      & [A_{i_0}]\ar@/^2ex/[rr]^{\varphi_{h_{k_1+1}\cdots h_{k_2}}} \ar[r]_{h_{k_1+1}}
      &\cdots \ar[r]_{h_{k_2}} & [A_{i_1}]\ar[r]_{h_{k_2+1}}&\cdots\ar[r]_{h_p} &B'_p
}
$$
such that the arrows 
$$
\xymatrix{
B_0\ar[r]_{g_1} &\cdots\ar[r]_{g_{j_0}} &[A_{i_0}]
	    \ar[r]_{g_{j_0+1}} & \cdots\ar[r]_{g_{j_1}} & A_{i_1}\ar[r]_{g_{j_1+1}} 
      & \cdots\ar[r]_{g_{j_2}} & [C] \ar[r]_{g_{j_2+1}}&\cdots\ar[r]_{g_{m}} & B_{m}
}$$
and
$$
\xymatrix{
B'_0\ar[r]_{h_1} &\cdots\ar[r]_{h_{k_0}} & [D]\ar[r]_{h_{k_0+1}} & \cdots \ar[r]_{h_{k_1}} 
      & A_{i_0} \ar[r]_{h_{k_1+1}}
      &\cdots \ar[r]_{h_{k_2}} & [A_{i_1}]\ar[r]_{h_{k_2+1}}&\cdots\ar[r]_{h_p} &B'_p
}$$
have companions in $\Dbl(\calB)$.
By Proposition \ref{comp-quasi} this implies that $A_{i_0}=C$, $D=A_{i_1}$ and
both chosen composites $\varphi_{g_{j_0+1}\cdots g_{j_2}}$ and $\varphi_{h_{k_0+1}\cdots h_{k_2}}$
are quasi units in $\calB$. So the chosen composite $\varphi_{g_{j_1+1}\cdots g_{j_2}}$
is a pseudo-inverse for $\varphi_{f_{i_0+1}\cdots f_{i_1}}$, making this arrow an equivalence in $\calB$. 
\end{proof}

\begin{prop}\label{D:equiv-precomp}
 Let $\bbX$ be a weakly globular double category with a horizontal arrow $f\colon A\rightarrow B$.
Then the arrow $f\colon \bar{A}\rightarrow\bar{B}$
in $\Bic(\bbX)$ is an equivalence if and only if  $f$ is a precompanion in $\bbX$.
\end{prop}

\begin{proof}
 Suppose that $f\colon \bar{A}\rightarrow\bar{B}$ has pseudo inverse $g\colon \bar{B}\rightarrow\bar{A}$.
The two  compositions of these arrows in $\Bic(\bbX)$ are given by the horizontal compositions
$\tilde{g}\circ \tilde{f}$  and $\hat{f}\circ\hat{g}$ as in the following diagrams of double cells
\begin{equation}\label{firstcomp}
\xymatrix{
\tilde{A} \ar[dd]|\bullet\ar[r]^{\tilde{f}}\ar@{}[ddr]|\cong 
	&\tilde{B}\ar[d]|\bullet\ar[r]^{\tilde{g}}\ar@{}[dr]|\cong 
	& \tilde{A}'\ar[d]|\bullet 
\\
&B'\ar[d]|\bullet\ar[r]_g&A'
\\
A\ar[r]_f & B
}\end{equation}
and
\begin{equation}\label{secondcomp}
\xymatrix{
\hat{B}' \ar[dd]|\bullet\ar[r]^{\hat{g}}\ar@{}[ddr]|\cong 
	&\hat{A}\ar[d]|\bullet\ar[r]^{\hat{f}}\ar@{}[dr]|\cong 
	& \hat{B}\ar[d]|\bullet 
\\
&A\ar[d]|\bullet\ar[r]_f&B
\\
B'\ar[r]_g & A'\rlap{\quad.}
}\end{equation}
By Proposition \ref{quasi-comp}, the compositions $\tilde{g}\circ\tilde{f}$ and $\hat{f}\circ\hat{g}$
are companions in $\bbX$ and the diagrams (\ref{firstcomp}) and (\ref{secondcomp}) show 
that $f$ is a precompanion
with $f'=\tilde{f}$, $r_f=\tilde{g}$, $f''=\hat{f}$ and $l_f=\hat{g}$.

Now suppose that $f$ is a precompanion in $\bbX$ with additional arrows and cells as in Definition \ref{precomp}:
$$
\xymatrix{
\ar[r]^f\ar[d]|\bullet\ar@{}[dr]|\varphi &\ar[d]|\bullet &&&\ar[r]^f\ar[d]|\bullet\ar@{}[dr]|{\varphi'}&\ar[d]|\bullet
\\
\ar[r]_{f'} &\ar[r]_{r_f} & &\ar[r]_{l_f} & \ar[r]_{f''} &\rlap{\quad.}
}$$
By Proposition \ref{quasi-comp}, $r_f\circ f'$ and $f''\circ l_f$ are quasi units in $\Bic(\bbX)$, say with 
invertible 2-cells $\alpha\colon r_f\circ f'\Rightarrow\mbox{Id}_{\bar{A}}$ and
$\beta\colon  f''\circ l_f\Rightarrow\mbox{Id}_{\bar{B}}$.
So we have the composites,
$r_f f\stackrel{r_f\varphi}{\Rightarrow} r_f\circ f'\stackrel{\alpha}{\Rightarrow}\mbox{Id}_{\bar{A}}$
and $f\,l_f\stackrel{\varphi' \,l_f}{\Rightarrow} f''\circ l_f\stackrel{\beta}{\Rightarrow}\mbox{Id}_{\bar{B}}$.
All these 2-cells are invertible and this implies that $f\colon \bar{A}\rightarrow\bar{B}$ in $\Bic(\bbX)$ is 
an equivalence.
\end{proof}

\begin{prop}
 A pseudo-functor between weakly globular double categories preserves precompanions.
\end{prop}

\begin{proof}
 Let $F\colon \bbX\rightarrow\bbY$ be a pseudo-functor between two weakly globular double 
categories $\bbX$ and $\bbY$, and let $\xymatrix@1{A\ar[r]^f&B}$ be a left precompanion in 
$\bbX$, with 
$$
\xymatrix{
A\ar[r]^f\ar[d]|\bullet\ar@{}[dr]|\varphi &B\ar[d]|\bullet
\\
A'\ar[r]_{f'}&B'\ar[r]_g & C
}$$ 
as in Definition \ref{precomp}, companion $v_f$, and binding cells
$$
\xymatrix{
A'\ar@{=}[d]\ar@{=}[rr] \ar@{}[drr]|{\psi_f} && A'\ar[d]|\bullet^{v_f}\ar@{}[drr]|{\mbox{and}}
	&& A'\ar[r]^{f'}\ar[d]|\bullet_{v_f} \ar@{}[drr]|{\chi_f} &B'\ar[r]^g & C\ar@{=}[d]
\\
A'\ar[r]^{f'}&B'\ar[r]^g & C && C\ar@{=}[rr]&& C\rlap{\,.}
}$$

Then we have the following diagram in $\bbY$ 
$$
\xymatrix@C=4em@R=2.5em{
\ar[r]^{Ff} \ar[d]|\bullet\ar@{}[dr]|{F\varphi}&\ar[d]|\bullet
\\
\ar[r]_{Ff'}\ar[dd]|\bullet\ar@{}[ddr]|\cong & \ar[d]|\bullet
\\
&\ar[r]^{Fg}\ar@{}[dr]|\cong\ar[d]|\bullet &\ar[d]|\bullet 
\\
\ar[d]|\bullet  \ar[r]_{\pi_2F_2(g,f')}&\ar@{}[d]|(.65){\mu_{g,f'}^{-1}} \ar[r]_{\pi_1F_2(g,f')}&\ar[d]|\bullet
\\
\ar[rr]_{F_1(gf')} &&\rlap{\quad.}
}$$
By Proposition \ref{comp-pres}, $F_1(gf')$ is a companion in $\bbY$, so  $\pi_1F_2(g,f')\circ\pi_2F_2(g,f')$
is a companion as well.

The fact that pseudo-functors preserve right precompanions goes similarly.
\end{proof}

\section{The weakly globular double category $\Dbl(\bfC(W^{-1}))$}\label{univ1}

Given a category $\bfC$ with a class $W$ of arrows which satisfies the conditions 
to define a category of fractions $\bfC[W^{-1}]$ as given by Gabriel and Zisman (and repeated in Section \ref{GZ-condns} below),
it is possible to not just define the category of fractions $\bfC[W^{-1}]$, but also the bicategory of fractions denoted by $\bfC(W^{-1})$.
Furthermore, the bicategory of fractions is defined precisely when the category of fractions is defined.
The bicategory of fractions has a universal property in terms of hom categories rather than just hom sets,
but the fundamental category of this bicategory is the original category of fractions and the bicategory of fractions 
is even biequivalent to (as bicategories) the category of fractions. 
This shows us that the hom categories in the bicategory of fractions will just consist 
of the equivalence relation on the
arrows of the category of fractions, as described in \cite{GZ}, i.e., there is 2-cell between any two parallel arrows in the
bicategory of fractions precisely when the two arrows are equivalent in the category of fractions and this 2-cell is unique.
In the remainder of this paper we want to study the weakly globular double categories that 
can be used as a counterpart of the bicategory of fractions $\bfC(W^{-1})$.
We will describe their universal properties and finally find a weakly globular double category that can be viewed as the weakly globular double
category of fractions. The absence of non-identity 2-dimensional cells in $\bfC$ simplifies the calculations and helps us to see what 
the basic properties of such weakly globular double categories are. 
In the sequel \cite{sequel} to this paper we will show how the cells of a bicategory or weakly globular double category 
can be added to this construction in a very natural way to give the weakly globular double category of 
fractions of a weakly globular double category with respect to a class of horizontal arrows.

\subsection{Categories of fractions}\label{GZ-condns}
Let ${\bfC}$ be a category with a class $W\subseteq \bfC_1$ of arrows satisfying the conditions 
for a calculus of fractions
given by Gabriel and Zisman \cite{GZ} (which imply the existence of a category of fractions ${\bfC}[W^{-1}]$):
\begin{itemize}
\item{\bf CF1} $W$ contains all isomorphisms in ${\bfC}$ and is closed under composition;
\item{\bf CF2} For any diagram $\xymatrix@1{C\ar[r]^f&B&\ar[l]_w A}$ in $\bfC$ 
with $w\in W$, there exist arrows $\xymatrix@1{D \ar[r]^{\overline{f}} & A}$ and 
$\xymatrix@1{D\ar[r]^{\overline{w}} &C}$ with $\overline{w}\in W$ such that
$$
\xymatrix@R=1.8em@C=1.8em{
D\ar[r]^{\overline{f}}\ar[d]_{\overline{w}} & A\ar[d]^w
\\
C\ar[r]_f & B}
$$
commutes;
\item{\bf CF3} For any diagram $\xymatrix@1{A\ar@<-.5ex>[r]_f\ar@<.5ex>[r]^g &B\ar[r]^w& C}$ 
with $w\in W$, such that $wf=wg$, there exists an arrow $\xymatrix@1{X\ar[r]^{\tilde{w}}&A}$
in $W$ such that $f\tilde{w}=g\tilde{w}$.
\end{itemize}
Note that the conditions on a class of arrows $W$ in a bicategory ${\bfC}$ 
to form a bicategory of fractions coincide with these conditions when $\bfC$ is a category (cf.~\cite{Pr-comp}).
So when a class $W$ of arrows satisfies these conditions, both the category ${\bfC}[W^{-1}]$
and the bicategory ${\bfC}(W^{-1})$ are defined (and the former can be obtained as a quotient of the latter).

\begin{eg}
 {\rm{\em Manifold Atlases} 
Consider manifolds with atlases such that the intersection of two atlas charts is 
either itself a chart or covered by charts.
For any smooth map between two manifolds $f\colon M\rightarrow N$ 
there are such atlases $\calU$ for $M$ and $\calV$ for $N$ such that for each 
chart $U\in{\calU}$, there is a chart $V\in\calV$ with a smooth map $f_U\colon U\rightarrow V$
and such that for any pair of charts with a chart embedding $\lambda\colon U_1\hookrightarrow U_2$,  
there are charts $V_1$ and $V_2$
in $\calV$ with a chart embedding $\mu\colon V_1\hookrightarrow V_2$ and such that $f_{U_i}\colon U_i\rightarrow V_i$
for $i=1,2$ and the diagram
$$\xymatrix@R=1.8em@C=2em{
U_1\ar[d]_\lambda\ar[r]^{f_{U_1}} & V_1\ar[d]^\mu
\\
U_2\ar[r]_{f_{U_2}} & V_2
}
$$
commutes.
We may choose to start with manifolds represented by such atlases (without explicitly keeping track of 
the underlying spaces)
and consider atlas maps to be such families of smooth maps that are locally compatible as described above.
Two such atlas maps $(f_U)_{U\in\mathcal{U}}$  and $(g_U)_{U\in\mathcal{U}}$ are equivalent when 
for each $U\in\mathcal{U}$ with $f_U\colon U\rightarrow V_{f,U}$ and $g_U\colon U\rightarrow V_{g,U}$, 
there is a chart $V\in \mathcal{V}$ with a smooth map $h_U\colon U\rightarrow V$ and chart
embeddings $\lambda_f\colon V\hookrightarrow V_{f,U}$ and $\lambda_{g}\colon V\hookrightarrow V_{g,U}$
such that the following diagram commutes,
$$
\xymatrix{
&V_{f,U}\\
U\ar[ur]^{f_U}\ar[r]^{h_U}\ar[dr]_{g_U} & V\ar[u]_{\lambda_f}\ar[d]^{\lambda_g}
\\
&V_{g,U}\rlap{\,.}
}
$$
The category {\sf Atlases} consists of atlases and equivalence classes of atlas maps.

To obtain the usual category {\sf Mfds} of manifolds from {\sf Atlases} we do the following.
First, note that if $\calU$ and $\calU'$ are two atlases for the same manifold and $\calU'$ is a refinement 
of $\calU$, then there is  a further refinement $\calU''$ of $\calU'$ with an atlas map $\calU''\rightarrow\calU$.
We will call this a refinement atlas map.
Furthermore, the class $W$ of (equivalence classes of) refinement atlas maps satisfies the conditions {\bf CF1}, {\bf CF2} and {\bf CF3} above.
The corresponding category of fractions $\mbox{\sf Atlases}[W^{-1}]$ is categorically equivalent to {\sf Mfds}.
Arrows in this category can be thought of as first taking an atlas refinement and then mapping out.
In the category of fractions we consider two such maps to be the same if there is a common refinement
on which they would become the same (and this is the case precisely when they induce the same maps on the
underlying manifolds.) If we instead consider the bicategory of fractions, we consider two maps defined by 
different refinements as distinct but there is a unique invertible 2-cell between two such maps if
there is a common refinement where they become the same. }
\end{eg}

\subsection{Universal properties}
Both the category and bicategory of fractions have a universal property as a (weak) coinverter \cite{KLW}.
Specifically, a version of the universal property of the bicategory of fractions is given in the following theorem.

\begin{thm} \label{bicatuni} For a bicategory $\calB$ and a class $W$ of arrows satisfying the conditions to admit a bicategory of fractions,
there is a homomorphism of bicategories $J_W\colon\calB\rightarrow\calB(W^{-1})$
such that composition with $J_W$ induces an equivalence of categories
$$
\Bicat(\calB(W^{-1}),\calD)\simeq\Bicat_W(\calB,\calD),
$$
for any bicategory $\calD$.
\end{thm}

\begin{rmks}{\rm
\begin{enumerate}
\item
In this theorem, we may take $\Bicat(\calB,\calD)$ to be the category of homomorphisms and oplax transformations,
and $\Bicat_W(\calB,\calD)$ the 2-category of homomorphism and  oplax transformations
which send elements of $W$ to internal equivalences. We will also call these $W$-homomorphisms and $W$-transformations.
(Note that we can make sense of this statement by representing a transformation in $\Bicat(\calB,\calD)$
by a homomorphism $\calB\rightarrow \mbox{Cyl}\,(\calD)$.) There is also a version for lax transformations.
\item
Note that \cite{Pr-comp} gives this universal property as an equivalence of 2-categories, which involves also the modifications.
However, the proof given there also shows that this restricts to an equivalence of categories when we just ignore the modifications.
\end{enumerate}}
\end{rmks}

So when we consider the notion of weakly globular double category of fractions, we would like to get a similar universal 
property in a 2-category of weakly globular double categories. However, there are two choices for such a 2-category, 
$\WGDbl_{\st,h}$ and $\WGDbl_{\ps,v}$, which are generally very different.
Ideally, we would like to have a weakly globular double category of fractions which has a universal property in both
of these 2-categories. We will begin by applying the 2-functor $\Dbl$ to the bicategory of fractions $\bfC(W^{-1})$ of
a category $\bfC$ and study the universal property of the resulting weakly globular double category.

In order to be able to translate the universal property in Theorem \ref{bicatuni}
into a universal property involving weakly globular bicategories, we note that
the equivalence of categories given there restricts to the categories with  icons as arrows instead of general oplax transformations,
\begin{equation}\label{univequiv}
\NHom(\bfC(W^{-1}),\calD)\simeq\Bicat_{\mbox{\scriptsize\rm icon},W}(\bfC,\calD).
\end{equation}
When we apply the 2-functor $\Dbl$ to this equivalence, we obtain
$$
\WGDbl_{\ps,v}(\Dbl(\bfC(W^{-1})),\Dbl(\calD))\simeq\WGDbl_{\ps,v,W}(\Dbl(\bfC),\Dbl(\calD)),
$$
where the righthand side is for now just defined as the image of $\Bicat_{\mbox{\scriptsize\rm icon},W}(\bfC,\calD)$ 
under the 2-functor $\Dbl$.
In the next sections we will study $\Dbl(\bfC(W^{-1}))$ in more detail and give a description of the pseudo-functors
and vertical transformations in $\WGDbl_{\ps,v,W}(\Dbl(\bfC),\Dbl(\calD))$, so that we can extend this to an
equivalence
$$
\WGDbl_{\ps,v}(\Dbl(\bfC(W^{-1})),\bbX)\simeq\WGDbl_{\ps,v,W}(\Dbl(\bfC),\bbX),
$$
for an arbitrary weakly globular double category $\bbX$.

\subsection{The weakly globular double category $\Dbl(\bfC(W^{-1}))$}
The objects of the weakly globular double category $\Dbl(\bfC(W^{-1}))$ are of the form
$$\xymatrix@C=3em{&\ar[l]_-{q_1}\ar[r]^-{p_1} &\ar@{..}[r]&&\ar[l]_-{q_{i_0}}\ar[r]^-{p_{i_0}}&[A_{i_0}]
	&\ar[l]_-{q_{{i_0}+1}}\ar[r]^-{p_{{i_0}+1}}&\ar@{..}[r]&&\ar[l]_-{q_n}\ar[r]^-{p_n}&}$$
and the arrows are of the form
$$
\xymatrix{&\ar[l]_-{q_1}\ar[r]^-{p_1} &\ar@{..}[r]&&\ar[l]_{q_{i_0}}\ar[r]^-{p_{i_0}}&[A_{i_0}]
\ar@{..}[r]&&\ar[l]_-{q_{i_1}}\ar[r]^-{p_{i_1}}&[A_{i_1}]\ar@{..}[r]&&\ar[l]_-{q_n}\ar[r]^-{p_n}&}
$$
To define 2-cells we choose a composite $\xymatrix{&\ar[l]_{q}\ar[r]^p&}$ 
for each path $$\xymatrix@C=3em{&\ar[l]_-{q_{i_0+1}}\ar[r]^-{p_{i_0+1}}&\ar@{..}[r]&&\ar[l]_-{q_{i_1}}\ar[r]^-{p_{i_1}}&\rlap{\quad.}}$$
Then 2-cells have the form
$$
\xymatrix@C=3em{
A_0 &\ar[l]_-{q_1}\ar[r]^-{p_1} & \cdots [A_{i_0}]\ar@{=}[4,0] &\ar[l]_-{q_{i_0+1}}\ar[r]^-{p_{i_0+1}}&\cdots & \ar[l]\ar[r]& [A_{i_1}]\ar@{=}[4,0]&\ar[l]\ar[r]&\cdots A_n
\\
&&&&\ar[ull]^{q}\ar[urr]_p
\\
&&&&\ar[u]_s\ar[d]^t
\\
&&&&\ar[dll]_{q'}\ar[drr]^{p'}
\\
B_0 &\ar[l]^-{q'_1}\ar[r]_-{p'_1} & \cdots [B_{j_0}] &\ar[l]^-{q'_{j_0+1}}\ar[r]_-{p'_{j_0+1}}&\cdots 
		& \ar[l]\ar[r]& [B_{j_1}]&\ar[l]\ar[r]&\cdots B_m\rlap{\,\,.}
}
$$

We introduce special terminology for the arrows in $\Dbl(\calC)$ which are related to the elements of $W$.

\begin{dfn}{\rm
A horizontal arrow in $\Dbl(\bfC)$ is a called a {\em $W$-doubly marked path} if it is of the form 
$$
 \xymatrix{
A_0\ar[r]^{f_1} & \cdots \ar[r]^{f_{i_0}}&[A_{i_0}]\ar[r]^{f_{i_0+1}} & \cdots \ar[r]^{f_{i_1}}
      &[A_{i_1}]\ar[r]^{f_{i_1+1}}&\cdots\ar[r]^{f_{i_n}}&A_n
}
$$
with $f_{i_1}\circ\cdots\circ f_{i_0+1}\in W$.}
\end{dfn}

Since precompanions in weakly globular double categories correspond to equivalences in bicategories
as shown in Proposition \ref{quasi-comp} and Proposition \ref{comp-quasi}, we can characterize the 
arrows in the image of the $W$-homomorphisms under $\Dbl$ as follows.

\begin{lma}\label{W-hom}
 Let $F\colon \bfC\rightarrow \calD$ be a homomorphism in $\Bicat$.
Then $F$ is a $W$-homomorphism if and only if $\Dbl(F)$ sends a $W$-doubly marked path in $\Dbl(\bfC)$ 
to a precompanion in $\Dbl(\calD)$.
\end{lma}

Although this lemma follows immediately from the results mentioned, it is 
illuminating to consider the precompanion structures required in the proof of the left to right implication, so we include it here.

\begin{proof} (of Lemma \ref{W-hom})
Suppose that $F$ is a $W$-homomorphism.
Let $f$ be the horizontal arrow
$$
 \xymatrix@C=3em{
FA_0\ar[r]^{Ff_1} & \cdots \ar[r]^-{Ff_{i_0}}&[FA_{i_0}]\ar[r]^-{Ff_{i_0+1}} & \cdots \ar[r]^-{Ff_{i_1}}
      &[FA_{i_1}]\ar[r]^-{Ff_{i_1+1}}&\cdots\ar[r]^-{Ff_{i_n}}&FA_n
}
$$
in $\Dbl(\calD)$ with
$w=f_{i_1}\circ\cdots\circ f_{i_0+1}$ an arrow in $W$. 
Let $g\in \calD$ be the chosen composite for $F(f_{i_1})\circ\cdots\circ F(f_{i_0+1})$,
and $\phi\colon g\Rightarrow Fw$, the invertible coherence cell from $F$.
Furthermore, let $(Fw)^*$ be a chosen pseudo-inverse for $Fw$ (which exists since $F$ is a $W$-homomorphism)
with invertible unit and counit cells $\varepsilon_f\colon (Fw)^*\circ Fw\Rightarrow \mbox{Id}_{FA_{i_1}}$ and 
$\eta_f\colon \mbox{Id}_{FA_{i_0}}\Rightarrow Fw\circ(Fw)^*$. 
Then we take  
$f'$ in Definition \ref{precomp} to be the horizontal arrow
$$
\xymatrix{[FA_{i_0}]\ar[r]^-{Fw} & [FA_{i_1}]\ar[r]^-{(Fw)^*} & F_{i_0}}
$$
and $r_f$ to be the horizontal arrow
$$
\xymatrix{FA_{i_0}\ar[r]^-{Fw} & [FA_{i_1}]\ar[r]^-{(Fw)^*} & [F_{i_0}}].
$$
Then the double cell $\varphi$ in Definition \ref{precomp}, Part 1, is
$$
\xymatrix@R=4em{
FA_0\ar[r]^-{Ff_1} & \cdots \ar[r]^-{Ff_{i_0}}&[FA_{i_0}]\ar@{=}[d]\ar@/_2ex/[rr]_g\ar[r]^-{Ff_{i_0+1}} & \cdots \ar[r]^-{Ff_{i_1}}
      &[FA_{i_1}]\ar@{=}[d]\ar[r]^-{Ff_{i_1+1}}&\cdots\ar[r]^-{Ff_{i_n}}&FA_n\\
      &&[FA_{i_0}]\ar[rr]_-{Fw} \ar@{}[urr]|\phi&& [FA_{i_1}]\ar[r]_-{(Fw)^*} & F_{i_0}}
$$
Furthermore, the composition $r_f\circ f'$ is the horizontal arrow
$$
\xymatrix{[FA_{i_0}]\ar[r]^{Fw} & FA_{i_1}\ar[r]^{(Fw)^*} & [F_{i_0}}],
$$
and it is straightforward to check that $\varepsilon_f$ and its inverse give rise to the cells to make this a companion.

Analogously, $f''$ can be taken to be the horizontal arrow
$$
\xymatrix{
FA_{i_1}\ar[r]^{(Fw)^*}&[FA_{i_0}]\ar[r]^{Fw}& [FA_{i_1}]
}
$$
with
$l_f$ the arrow,
$$
\xymatrix{
[FA_{i_1}]\ar[r]^{(Fw)^*}&[FA_{i_0}]\ar[r]^{Fw}& FA_{i_1}.
}
$$
\end{proof}

Since every weakly globular double category is vertically equivalent to one of the form $\Dbl(\calD)$ 
for some bicategory $\calD$ and pseudo-functors of weakly globular double categories preserve precompanions, 
this leads us to the following definition.

\begin{dfn}\label{W-ps-functor}
{\rm A pseudo-functor $F\colon \Dbl(\bfC)\rightarrow \bbX$  is called a {\em $W$-pseudo-functor}
when it sends all $W$-doubly marked paths in $\Dbl(\bfC)$
to precompanions in $\bbX$.}
\end{dfn}

\begin{rmk}{\em
It is clear that we have some choice for the $f'$ and $f''$ in the proof of Lemma \ref{W-hom}. However, once they are chosen, 
$r_f$ and $l_f$ are fixed, since they need to be composable with $f'$ and $f''$ respectively.
We saw in Lemma \ref{lrprecomps} already that it is possible to choose $f'$ and $f''$ as the same horizontal arrow, $\bar{f}$.
One of the questions one may ask about the precompanion structure is whether one could require $l_f$ and $r_f$ to be equal as well.
The case of the proof above shows that one cannot ask this much for arbitrary precompanions.
We see here that we could have made $\bar{f}$ to be equal to
$$\xymatrix{FA_{i_1}\ar[r]^{(Fw)^*} & [FA_{i_0}] \ar[r]^{Fw}& [FA_{i_1}]\ar[r]^{(Fw)^*}&FA_{i_0}}.$$
However, that would have caused us to define $r_f$ to be
$$
\xymatrix{FA_{i_1}\ar[r]^{(Fw)^*} & FA_{i_0} \ar[r]^{Fw}& [FA_{i_1}]\ar[r]^{(Fw)^*}&[FA_{i_0}]}$$ 
and $l_f$  to be
$$
\xymatrix{[FA_{i_1}]\ar[r]^{(Fw)^*} & [FA_{i_0}] \ar[r]^{Fw}& FA_{i_1}\ar[r]^{(Fw)^*}&FA_{i_0}}$$
which are not equal.
}\end{rmk}

We also need a description of the vertical transformations in the image of the $W$-transformations
under $\Dbl$. 
We claim that when $\bbX=\Dbl(\calD)$ they are the following class of vertical transformations.

\begin{dfn}\label{W-trafo}
{\rm Let $W$ in $\bfC$ be a class of arrows to be inverted which satisfies the Gabriel Zisman conditions.
Let $F,G\colon \Dbl(\bfC)\rightrightarrows \bbX$ be  $W$-pseudo-functors and  for any $W$-doubly marked path $w$
in $\Dbl(\bfC)$, choose precompanion structures in $\bbX$,
$$\xymatrix{
& \ar[r]^{Fw}\ar[d]|\bullet \ar@{}[dr]|{\varphi_w^l} & \ar[d]|\bullet
& \ar[r]^{Fw}\ar[d]|\bullet \ar@{}[dr]|{\varphi_w^r} & \ar[d]|\bullet
\\
\ar[r]_{l_{Fw}} & \ar[r]_{(Fw)'} & & \ar[r]_{(Fw)''} & \ar[r]_{r_{Fw}} &}
$$
and
$$\xymatrix{
& \ar[r]^{Gw} \ar[d]|\bullet \ar@{}[dr]|{\psi_w^l} & \ar[d]|\bullet 
& \ar[r]^{Gw} \ar[d]|\bullet \ar@{}[dr]|{\psi_w^r} & \ar[d]|\bullet
\\
\ar[r]_{l_{Gw}} & \ar[r]_{(Gw)'} & & \ar[r]_{(Gw)''} &  \ar[r]_{r_{Gw}}& \rlap{\quad.}}
$$
For a vertical transformation
$\alpha\colon F\Rightarrow G$ we write 
$$
(\alpha_w)'=\raisebox{3.9em}{$
\xymatrix@C=3em{
\ar[r]^{(Fw)'}\ar[d]|\bullet\ar@{}[dr]|{(\varphi_w^l)^{-1}} & \ar[d]|\bullet
\\
\ar[r]^{Fw}\ar[d]|\bullet \ar@{}[dr]|{\alpha_w} &  \ar[d]|\bullet
\\
\ar[r]^{Gw}\ar[d]|\bullet \ar@{}[dr]|{\psi_w^l} & \ar[d]|\bullet
\\
\ar[r]_{(Gw)'}&
}$}
\quad\mbox{ and }\quad (\alpha_w)''= \raisebox{3.9em}{$
\xymatrix@C=3em{
\ar[r]^{(Fw)''}\ar[d]|\bullet\ar@{}[dr]|{(\varphi_w^r)^{-1}} & \ar[d]|\bullet
\\
\ar[r]^{Fw}\ar[d]|\bullet \ar@{}[dr]|{\alpha_w} &  \ar[d]|\bullet
\\
\ar[r]^{Gw}\ar[d]|\bullet \ar@{}[dr]|{\psi_w^r} & \ar[d]|\bullet
\\
\ar[r]_{(Gw)''}&
}$}
$$
Then $\alpha$   is a {\em (vertical) $W$-transformation} if 
there are double cells 
\begin{equation}\label{precompcells}
\xymatrix{
\ar[r]^{l_{Fw}}\ar[d]|\bullet \ar@{}[dr]|{\alpha_w^l}& \ar[d]|\bullet
\ar@{}[drr]|{\mbox{and}} && \ar[r]^{r_{Fw}}\ar[d]|\bullet \ar@{}[dr]|{\alpha_w^r} & \ar[d]|\bullet
\\
\ar[r]_{l_{Gw}}&&& \ar[r]_{r_{Gw}}&
}
\end{equation}
which are compatible with the companion cells in the sense
that the following four equations hold:
$$
\xymatrix{
\ar[d]|\bullet\ar[r]^{l_{Fw}}\ar@{}[dr]|{\alpha_w^l}
	&\ar[d]|\bullet\ar@{}[dr]|{(\alpha_w)'} \ar[r]^{(Fw)'}
	&\ar[d]|\bullet^{d_1\alpha_w'}
	&& \ar[d]|\bullet_{v^l_{Fw}} \ar[r]^{l_{Fw}}\ar@{}[drr]|{\chi^l_{Fw}}
	&\ar[r]^{(Fw)'} 
	&\ar@{=}[d]
\\
\ar[r]_{l_{Gw}}\ar[d]|\bullet_{v_{Gw}^l}\ar@{}[drr]|{\chi^l_{Gw}} & \ar[r]_{(Gw)'} & \ar@{=}[d]
&=& \ar@{=}[rr]\ar[d]|\bullet_{d_1(\alpha_w)'} \ar@{}[drr]|{\mathrm{id}} && \ar[d]|\bullet^{d_1(\alpha_w)'}
\\
\ar@{=}[rr] && && \ar@{=}[rr] &&
\\
\ar@{=}[rr]\ar[d]|\bullet_{d_0\alpha_w^l}\ar@{}[drr]|{\mathrm{id}} &&\ar[d]|\bullet^{d_0\alpha_w^l} 
	&& \ar@{=}[d]\ar@{=}[rr]\ar@{}[drr]|{\psi^l_{Fw}} && \ar[d]|\bullet^{v^l_{Fw}}
\\
\ar@{=}[rr]\ar@{=}[d]\ar@{}[drr]|{\psi^l_{Gw}} && \ar[d]|\bullet^{v^l_{Gw}} 
	&=&\ar[d]|\bullet_{d_0\alpha^l_w} \ar@{}[dr]|{\alpha_w^l}\ar[r]_{l_{Fw}} 
	& \ar[d]|\bullet \ar[r]_{(Fw)'}\ar@{}[dr]|{(\alpha_w)'} & \ar[d]|\bullet 
\\
\ar[r]_{l_{Gw}}&\ar[r]_{Gw} & && \ar[r]_{l_{Gw}}&\ar[r]_{(Gw)'} &
\\
\ar[r]^{(Fw)'}\ar[d]|\bullet \ar@{}[dr]|{(\alpha_w)'} & \ar[d]|\bullet \ar[r]^{r_{Fw}}\ar@{}[dr]|{\alpha^r_w}
	& \ar[d]|\bullet^{d_1\alpha^r_w} && \ar[r]^{(Fw)'}\ar[d]|\bullet_{v^r_{Fw}}\ar@{}[drr]|{\chi^r_{Fw}} 
	& \ar[r]^{r_{Fw}} & \ar@{=}[d]
\\
\ar[d]|\bullet_{v^r_{Gw}}\ar@{}[drr]|{\chi^r_{Gw}} \ar[r]_{(Gw)'} & \ar[r]_{r_{Gw}} & \ar@{=}[d] 
	&=& \ar[d]|\bullet_{d_1\alpha^r_w}\ar@{}[drr]|{\mathrm{id}} \ar@{=}[rr] && \ar[d]|\bullet^{d_1\alpha^r_w}
\\
\ar@{=}[rr] && && \ar@{=}[rr] &&
\\
\ar@{=}[d]\ar@{=}[rr]\ar@{}[drr]|{\psi^r_{Fw}} && \ar[d]|\bullet^{v_{Fw}^r} 
	&& \ar[d]|\bullet_{d_0(\alpha_w)'}\ar@{}[drr]|{\mathrm{id}} \ar@{=}[rr]&& \ar[d]|\bullet^{d_0(\alpha_w)'}
\\
\ar[d]|\bullet_{d_0(\alpha_w)'}\ar[r]_{(Fw)'}\ar@{}[dr]|{(\alpha_w)'} 
	& \ar[d]|\bullet \ar[r]_{r_{Fw}}\ar@{}[dr]|{\alpha^r_w} & \ar[d]|\bullet 
	&=& \ar@{=}[d]\ar@{=}[rr]\ar@{}[drr]|{\psi^r_{Gw}} && \ar[d]|\bullet^{v^r_{Gw}}
\\
\ar[r]_{(Gw)'}&\ar[r]_{r_{Gw}} & &&\ar[r]_{(Gw)'}&\ar[r]_{r_{Gw}} &
}$$}
\end{dfn}

\begin{rmk}{\rm
Note that since the precompanion structures for a given horizontal arrow are unique up to vertical isomorphism,
the existence of the cells (\ref{precompcells}) does not depend on the choice of the precompanions structure.
Furthermore, when we have a collection of these cells that satisfy the conditions of the definition above,
conjugation by the isomorphisms of Lemma \ref{pre-comp-uniq} will give another set of cells that satisfy the conditions.}
\end{rmk}

Recall that an icon $\psi\colon F\Rightarrow G$ between $W$-homomorphisms is a $W$-icon
if and only if for each $w\in W$ and choice of pseudo-inverses $(Fw)^*$ and $(Gw)^*$ of $Fw$ and $Gw$ respectively
there is a 2-cell $\alpha_w^*\colon (Fw)^*\Rightarrow(Gw)^*$ such that the pastings of the diagrams
$$
\xymatrix{
FA\ar@{=}[d]\ar@{}[dr]|{\alpha_w}\ar[r]^{Fw}\ar@{=}@/^5ex/[rr] \ar@<2.5ex>@{}[rr]|{\eta_{Fw}}
		& FB\ar@{=}[d]\ar@{}[dr]|{\alpha_w^*} \ar[r]^{(Fw)^*}
		& FA\ar@{=}[d] \ar@{}[dr]|{\mbox{and}}
		& FB \ar@{=}[d] \ar@{}[dr]|{\alpha_w^*} \ar[r]^{(Fw)^*} \ar@{=}@/^5ex/[rr] \ar@<2.5ex>@{}[rr]|{\varepsilon_{Fw}^{-1}}
		& FA\ar@{=}[d]\ar@{}[dr]|{\alpha_w} \ar[r]^{Fw}
		& FB\ar@{=}[d]
\\
GA\ar[r]_{Gw}\ar@{=}@/_5ex/[rr] \ar@<-2.5ex>@{}[rr]|{\eta_{Gw}^{-1}}  & GB \ar[r]_{(Gw)^*}  & GA
		& GB\ar[r]_{(Gw)^*}\ar@{=}@/_5ex/[rr] \ar@<-2.5ex>@{}[rr]|{\varepsilon_{Gw}} & GA \ar[r]_{Gw} & GB
}
$$ 
are identity 2-cells. It is now straightforward to prove the following result.

\begin{lma}
An icon $\psi\colon F\Rightarrow G$ between $W$-homomorphisms is a $W$-icon if and only if  the corresponding vertical transformation
$\Dbl(\psi)\colon \Dbl(F)\Rightarrow\Dbl(G)$ is a $W$-transformation.
\end{lma}

Note that the horizontal composition of a $W$-transformation with a pseudo-functor is again a $W$-transformation.
We conclude that, since every weakly globular double category is vertically equivalent to one of the form $\Dbl(\calD)$
for some bicategory $\calD$, the vertical universal property of 
$\Dbl(\bfC(W^{-1}))$ is as stated in the following theorem.

\begin{thm}\label{univ1-thm}
For any category $\bfC$ with a collection $W$ of arrows satisfying the conditions to form a bicategory of
fractions, composition with the functor $\Dbl(J_W)\colon \Dbl\bfC\rightarrow\Dbl(\bfC(W^{-1}))$ 
induces an equivalence of categories
$$
\WGDbl_{\ps,v}(\Dbl(\bfC(W^{-1})),\bbX)\simeq\WGDbl_{\ps,v,W}(\Dbl(\bfC),\bbX)
$$
where $\WGDbl_{\ps,v,W}(\Dbl(\bfC),\bbX)$ is the category of $W$-pseudo-functors and vertical $W$-transformations.
\end{thm}

We would like to have a horizontal universal property for $\Dbl(\bfC(W^{-1}))$ as well, but there is no obvious one
at this point.
Furthermore, the description of the arrows to be inverted in $\Dbl(\bfC)$ is rather cumbersome.
So in the next section we will introduce a new candidate for the weakly globular double category of fractions.
In order to get a better handle on the horizontal arrows that will be inverted and to keep 
the complexity of the construction down, we will just consider the case of a class of horizontal arrows in a weakly globular 
double category of the form $H\bfC$ where $\bfC$ is a category.
The resulting double category $\bfC\{W\}$ will be vertically equivalent to the image $\Dbl(\bfC(W^{-1}))$ of the bicategory of fractions and will hence 
inherit its vertical universal property from the one given above.
However, the horizontal universal property will be a completely new one.

\section{The construction of $\bfC\{W\}$}\label{CW-con}

In this section we introduce a weakly globular double category $\CW$ 
which is (vertically) 2-equivalent to $\Dbl(\bfC(W^{-1}))$, and hence shares its vertical universal property.
Furthermore, we will show that it has an additional horizontal universal property as a double category
with companions and conjoints. 
The two properties together determine $\CW$ up to vertical 2-equivalence and horizontal equivalence of 
weakly globular double categories.
 
Furthermore, one draw-back of both 
the category and bicategory of fractions constructions is that 
the resulting category or bicategory does not necessarily have small hom-sets. 
The same thing is true for $\Dbl(\bfC(W^{-1}))$, but 
not for $\CW$: it has small homsets in both directions, both for arrows and double cells.

The reason we can achieve this is as follows.
For categories and bicategories of fractions we kept the same objects, but added new arrows and 2-cells, in order to obtain 
(pseudo) inverses for the arrows in the class $W$.
However, for the weakly globular double category of fractions, we are not looking for the structure of a pseudo inverse, but rather 
for the structure of a precompanion, since we saw in Propositions \ref{D:precomp-equiv} and 
\ref{D:equiv-precomp} that precompanions in weakly globular double categories
correspond to equivalences in bicategories and we established the universal property of $\Dbl(\bfC(W^{-1}))$ 
(Theorem \ref{univ1-thm}, with Definitions \ref{W-ps-functor} and \ref{W-trafo}) in terms of 
precompanions. Intuitively we can think of precompanions as having a horizontal pseudo inverse up to a vertical isomorphism.
So we can add new objects with our new arrows, as long as the new objects are vertically isomorphic to the old ones. 
The construction of $\Dbl(\bfC(W^{-1}))$ also added new objects, but there 
the idea was that the new objects were dictated by the new paths. In this section we want to give a free construction 
that starts by adding new objects and lets the arrows and double cells follow from what is needed to obtain the precompanion structure.
We will add an additional object for each arrow $w$ in $W$, with a horizontal arrow to the domain of the original arrow.
The idea is that this additional arrow will be $l_w$, so that we need to give $w\circ l_w$ the structure of a companion.
We will not take everything as free as possible, because the vertical universal property is in terms of pseudo-functors, and this gives us 
a bit of flexibility to make additional assumptions. For instance, we will make $l_w$ a horizontal isomorphism. This tightening in the horizontal direction
is what will give us eventually a horizontal universal property for the resulting weakly globular double category.

The idea of the construction comes from a kind of symmetrization of the original construction of a (bi)category of (right) fractions.
An arrow in a (bi)category of fractions is given as a span $A\stackrel{w}{\leftarrow}A'\stackrel{f}{\rightarrow}B$ where $w$ is in the class of
arrows to be inverted. Sometimes these arrows are also called generalized arrows.
We could think of all objects that are connected to each other by arrows that are to be inverted as different representations
of the same object in a quotient of the category. The Gabriel-Zisman conditions assure us that an arrow in the quotient category can be given by
taking a different representation followed by 
an arrow out of the new representation in the original category. The (bi)category of fractions construction does not go as far as to take the quotient,
but rather keeps track of how the new representation is related to the domain of the generalized arrow.
For the weakly globular double category of fractions, we will allow for this type of change of representation in both the domain and the codomain
of an arrow. Furthermore, rather than using generalized arrows, it turns out that we can do this by adding new objects.
The new objects will correspond to the arrows in $W$, so they will be objects with a second representation (which may be thought of as some type of refinement). Horizontal arrows will then be given by arrows between the refinements.
In the traditional constructions of (bi)categories of fractions one always needs to consider an equivalence relation 
for the cells at the highest level and make two of those cells equal if they agree on a further refinement.
In the case of weakly globular double categories we can use the vertical structure to deal with this as we will show below.

\subsection{The construction}\label{construction}
We define the weakly globular double category $\bfC\W$ as follows:
\begin{itemize}
\item {\em Objects} are arrows in $W$, $(w)=\left(\xymatrix@1{A\ar[r]^w&B}\right)$.
\item A {\em vertical arrow} $(u_1,C,u_2)\colon \xymatrix@1{(A_1\ar[r]^{w_1}&B)
\ar[r]|-\bullet &(A_2\ar[r]^{w_2}&B)}$ 
is an equivalence class of 
commutative diagrams
$$
\xymatrix@R=1.3em{A_1\ar[r]^{w_1} &B\ar@{=}[dd]
\\
C\ar[u]^{u_1}\ar[d]_{u_2}
\\
A_2\ar[r]_{w_2} &B}
$$
with $w_1u_1=w_2u_2$ in $W$.
Two such diagrams, $(u_1,C,u_2)$ and $(v_1,D,v_2)$, are equivalent (i.e., represent the same vertical arrow) 
when there are arrows $\xymatrix@1{C &E\ar[l]_{r_1}\ar[r]^{r_2}& D}$ 
such that 
$$
\xymatrix@C=1.5em{
&A_1
\\
C\ar[ur]^{u_1} \ar[dr]_{u_2} & E\ar[l]_{r_1}\ar[r]^{r_2} & D\ar[ul]_{v_1}\ar[dl]^{v_2}
\\
&A_2
}
$$
commutes and $w_1u_1r_1=w_1v_1r_2=w_2v_2r_2=w_2u_2r_1$ is in $W$.
\item A {\em horizontal arrow} $\xymatrix@1{(A\ar[r]^w&B)\ar[r]^f&(A'\ar[r]^{w'}&B')}$
is given by an arrow $\xymatrix@1{A\ar[r]^f&A'}$ in $\bfC$.
We will usually draw this as $\xymatrix@1{(B&\ar[l]_w A)\ar[r]^f&(A'\ar[r]^{w'}&B')}$.
\item
A {\em double cell} 
$$
\xymatrix{
(w_1)\ar[d]|\bullet_{[u_1,C,u_2]}\ar[r]^{f_1}\ar@{}[dr]|{(\varphi)} & (w_1')\ar[d]|\bullet^{[u_1',C',u_2']}\\
(w_2)\ar[r]_{f_2} & (w_2').
}$$
is an equivalence class of commutative diagrams of the form
\begin{equation}\label{square}
\xymatrix@C=5em@R=2em{
B\ar@{=}[dd] &A_1\ar[l]_{w_1}\ar[r]^{f_1} & A_1'\ar[r]^{w_1'} & B'\ar@{=}[dd]
\\
& C\ar[u]_{u_1}\ar[d]^{u_2}\ar[r]_{\varphi} &C' \ar[u]_{u_1'}\ar[d]^{u_2'} &
\\
B & A_2\ar[l]^{w_2}\ar[r]_{f_2} & A_2'\ar[r]_{w_2'} & B'
}
\end{equation}
The diagram (\ref{square}) is equivalent to the diagram
$$\xymatrix@C=5em@R=2em{
B\ar@{=}[dd] &A_1\ar[l]_{w_1}\ar[r]^{f_1} & A_1'\ar[r]^{w_1'} & B'\ar@{=}[dd]
\\
& D\ar[u]_{v_1}\ar[d]^{v_2}\ar[r]_{\psi} &D' \ar[u]_{v_1'}\ar[d]^{v_2'} &
\\
B & A_2\ar[l]^{w_2}\ar[r]_{f_2} & A_2'\ar[r]_{w_2'} & B'
}
$$
if and only if there are arrows $r,s,r',s'$ and $\chi$ 
as in 
$$
\xymatrix{
C\ar[rr]^\varphi&&C'\\
E\ar[u]^r\ar[d]_s\ar[rr]^\chi && E'\ar[u]_{r'}\ar[d]^{s'}\\
D\ar[rr]_\psi&&D'
}$$
such that $w_1u_1r\in W$, $w'_1u_1'r'\in W$,
and making the following diagram commute:
\begin{equation}\label{dblcellequiv}
\xymatrix@C=7ex{
\ar@{=}[4,0] && \ar[ll]_{w_1} \ar[rr]^{f_1} && \ar[rr]^{w_1'} && \ar@{=}[4,0]
\\
&&& \ar@/_1ex/[ul]^{v_1} \ar[rr]|(.22)\hole^(.58){\psi} \ar@/^1.5ex/[3,-1]_{v_2}|(.3)\hole|(.63)\hole 
		&& \ar@/_1ex/[ul]^{v_1'} \ar@/^1.5ex/[3,-1]^{v_2'} &&
\\
&& \ar[ur]_{s}\ar[dl]^r \ar[rr]_\chi|(.59)\hole &&\ar[ur]_{s'}\ar[dl]^{r'}
\\
& \ar@/^1.5ex/[-3,1]^{u_1}\ar@/_.5ex/[dr]_{u_2} \ar[rr]_{\varphi} 
		&&\ar@/^1.5ex/[-3,1]_{u_1'} \ar@/_.5ex/[dr]_{u_2'}
\\
&&\ar[ll]^{w_2}\ar[rr]_{f_2} && \ar[rr]_{w_2'}&&
}
\end{equation}
\end{itemize}

Note that a double cell may not have a representative for each combination of representatives of each vertical arrows,
however, there is always a representative with the given representative of the codomain vertical arrow  
as in the following lemma.

\begin{lma}\label{reps}
 Given a representative of a double cell,
\begin{equation}\label{square1}
\xymatrix@C=5em@R=2em{
B\ar@{=}[dd] &A_1\ar[l]_{w_1}\ar[r]^{f_1} & A_1'\ar[r]^{w_1'} & B'\ar@{=}[dd]
\\
& C\ar[u]_{u_1}\ar[d]^{u_2} \ar[r]^\varphi &C' \ar[u]_{u_1'}\ar[d]^{u_2'} &
\\
B & A_2\ar[l]^{w_2}\ar[r]_{f_2} & A_2'\ar[r]_{w_2'} & B',
}
\end{equation}
and another representative $(v_1,D,v_2)$ of the codomain vertical arrow, then there are arrows $r\colon E\to C$ and $\psi\colon E\to D$
such that
\begin{equation}\label{square2}
\xymatrix@C=5em@R=2em{
B\ar@{=}[dd] &A_1\ar[l]_{w_1}\ar[r]^{f_1} & A_1'\ar[r]^{w_1'} & B'\ar@{=}[dd]
\\
& E\ar[u]_{u_1r}\ar[d]^{u_2r} \ar[r]^\psi &D \ar[u]_{v_1}\ar[d]^{v_2} &
\\
B & A_2\ar[l]_{w_2}\ar[r]_{f_2} & A_2'\ar[r]_{w_2'} & B'
}
\end{equation}
and (\ref{square1}) represent the same double cell.
\end{lma}

\begin{proof}
 Since 
 $$
 \xymatrix{
 A_1'\ar[r]^{w_1'}& B'\ar@{=}[dd] \ar@{}[2,2]|{\mbox{and}} && A_1'\ar[r]^{w_1'} & B'\ar@{=}[dd]
 \\
 C'\ar[u]_{u_1'}\ar[d]^{u_2'} &&&D\ar[u]_{v_1}\ar[d]^{v_2}
 \\
 A_2'\ar[r]_{w_2'}&B' && A_2'\ar[r]_{w_2'}&B'
 }
 $$
 represent the same vertical arrow, there are arrows $r_1$ and $r_2$ such that
 $w_1'u_1'r_1,w_1'v_1r_2\in W$ and
 $$
 \xymatrix{
 &A_1'
 \\
 C'\ar[ur]^{u_1'}\ar[dr]_{u_2'} & F\ar[l]_{r_1}\ar[r]^{r_2} & D\ar[ul]_{v_1}\ar[dl]^{v_2}
 \\
 &A_2'
 }
 $$
 commutes.
 By condition {\bf CF2} there are arrows $\bar{r}_1\in W$ and $\bar{\varphi}$  that make the following
 square commute,
 $$
 \xymatrix{
 \bar{F}\ar[d]_{\bar{r}_1}\ar[rr]^{\bar\varphi} && F\ar[d]^{w_2'v_2r_1}
 \\
 C\ar[r]_{\varphi} & C'\ar[r]_{w_2'v_2} & B'.
 }
 $$
 By condition {\bf CF3} there exists an arrow $(\tilde{v}\colon E\rightarrow\bar{F})\in W$
 such that 
 $$
 \xymatrix{
 E\ar[d]_{\bar{r}_1\tilde{v}}\ar[r]^{\bar{\varphi}\tilde{v}}&F\ar[d]^{r_1}
 \\
 C\ar[r]_\varphi & C'
 }
 $$
 commutes.
 We can use all this to construct the following commutative diagram
 $$
 \xymatrix{
 B\ar@{=}[4,0] && A_1\ar[ll]_{w_1} \ar[r]^{f_1} & A_1'\ar[rr]^{w_1'} && B'\ar@{=}[4,0]
 \\
 &&C\ar[u]^{u_1}\ar[r]^{\varphi} &C'\ar[u]^{u_1'}
 \\
 &&E\ar[u]^{\bar{r}_1\tilde{v}}\ar[r]^{\bar{\varphi}\tilde{v}}\ar[d]_{\bar{r}_1\tilde{v}} 
 			& F\ar[u]^{r_1}\ar[d]_{r_1}\ar[r]^{r_2}&D\ar[uul]_{v_1}\ar[ddl]^{v_2}
 \\
 &&C\ar[r]_\varphi\ar[d]^{u_2} & C'\ar[d]_{u_2'}
 \\
 B&&A_2\ar[ll]^{w_2}\ar[r]_{f_2} & A_2'\ar[rr]_{w_2'}&& B'
 }
 $$
 So we obtain a double cell 
 $$
 \xymatrix{
 B\ar@{=}[dd]&A_1\ar[l]_{w_1}\ar[r]^{f_1} & A_1'\ar[r]^{w_1'} &B'\ar@{=}[dd]
 \\
 &E\ar[r]^{r_2\bar{\varphi}\tilde{v}} \ar[u]^{u_1\bar{r}_1\tilde{v}}\ar[d]_{u_2\bar{r}_1\tilde{v}}&D\ar[u]_{v_1}\ar[d]^{v_2}
 \\
 B&A_2\ar[l]^{w_2}\ar[r]_{f_2}& A_2'\ar[r]_{w_2'}& B'. 
 }
 $$
 This cell is equivalent to (\ref{square1}) since we have a commutative diagram
 $$
 \xymatrix{
 E\ar@{=}[d]\ar[r]^{r_2\bar{\varphi}\tilde{v}}&D
 \\
 E\ar[r]^{\bar{\varphi}\tilde{v}}\ar[d]_{\bar{r}_1\tilde{v}} & F\ar[u]_{r_2}\ar[d]^{r_1}
 \\
 C\ar[r]_{\varphi}& C'
 }
 $$
 which fits in a diagram like (\ref{dblcellequiv}).
 By taking $r=\bar{r}_1\tilde{v}$ and $\psi=r_2\bar{\varphi}\tilde{v}$ we see that we have indeed
  a cell as in (\ref{square2}).
\end{proof}

\begin{eg}
{\rm In the example where $\bfC=\mbox{\sf Atlases}$ and $W$ the class of atlas refinements,
the situation is  as follows.
The objects of $\CW$ are atlases with a refinement. There is a vertical arrow between any two such objects 
if there is a common refinement, and a horizontal arrow corresponds to an atlas map between the refinements.
Double cells correspond to common refinements where the two atlas maps corresponding to the horizontal arrows
agree.}
\end{eg}
 
 \subsection{Some basic properties}
 We will now study some of the basic properties of $\CW$ and in particular, we will show that $\CW$ is a weakly globular double category.
 
The following lemma is equivalent to stating that the codomain map $d_1\colon\CW_1\rightarrow \CW_0$,
sending squares and horizontal arrows to their horizontal codomains is an isofibration.

\begin{lma}\label{isofib}
For every pair of a horizontal arrow $\xymatrix@1{(w_2)\ar[r]^g&(w_2')}$ and a vertical arrow 
$\xymatrix@1@C=3.5em{(w_1')\ar[r]|\bullet^{[v_1,D,v_2]}&(w_2')}$ in $\CW$, there is a double cell 
$$
\xymatrix{
(w_1)\ar[r]^f\ar[d]|\bullet_{[u_1,C,u_2]}\ar@{}[dr]|{(\varphi)} & (w_1')\ar[d]|\bullet^{[v_1,D,v_2]}
\\
(w_2)\ar[r]_g & (w_2')\rlap{\,.}
}
$$
\end{lma}

\begin{proof}
Let $\xymatrix@1{A_2\ar[r]^{w_2}&B}$ and $\xymatrix@1{A_i'\ar[r]^{w_i'}&B'}$, 
for $i=1,2$ in $W$.
By condition {\bf CF2} there exists a commutative square 
$$
\xymatrix{
E\ar[r]^{\varphi_E}\ar[d]_{\overline{v}_2}&D\ar[d]^{w_2'v_2}
\\
A_2\ar[r]_{w_2'g} & B'}
$$
in $\bfC$.
Since $w_2'\in W$, there is an arrow $(\xymatrix@1{C\ar[r]^{\tilde{w}'_2}&E})\in W$ 
such that
$$
\xymatrix{
C\ar[d]_{\overline{v}_2\tilde{w}_2'}\ar[r]^{\varphi_E\tilde{w}_2'} & D\ar[d]^{v_2}
\\
A_2\ar[r]_g&A_2'
}
$$
also commutes (by condition {\bf CF3}).
So let $\varphi_C=\varphi_E\tilde{w}_2'$ $u_2=\overline{v}_2\tilde{w}_2'$. Note that $u_2\in W$,
and the diagram
$$
\xymatrix@R=2em{
(B\ar@{=}[dd]&\ar[l]_{w_2u_2}C)\ar[r]^{v_1\varphi_C} & (A_1'\ar[r]^{w_1'}&B'\ar@{=}[dd])
\\
	& C\ar@{=}[u]\ar[d]^{u_2}\ar[r]^{\varphi_C}&D\ar[d]_{v_2}\ar[u]^{v_1}
\\
(B&A_2\ar[l]^{w_2})\ar[r]_g&(A_2'\ar[r]_{w_2'}&B')
}
$$
commutes in $\bfC$ and satisfies all the conditions to represent a double cell
in $\CW$.
\end{proof}

The following lemma implies that the vertical arrow category is a posetal groupoid.
 
\begin{lma}\label{connectedcomp}
 There is a vertical arrow $\xymatrix@1{(A\ar[r]^w&B)\ar[r]|\bullet &(A'\ar[r]^{w'}&B')}$
if and only if $B=B'$ and furthermore, this arrow is unique.
\end{lma}

\begin{proof}
It is obvious that the existence of such a vertical arrow implies that $B=B'$.

Now suppose that we have objects $\xymatrix@1{(A\ar[r]^w&B)}$ and $\xymatrix@1{(A'\ar[r]^{w'}&B)}$
in $\CW$.
Since $W$ satisfies condition {\bf CF2} above, there is a commutative square
$$
\xymatrix@R=1.5em{C\ar[r]^{u}\ar[d]_{u'}&A\ar[d]^{w}
\\
A'\ar[r]_{w'} &B
}$$
in $\bfC$ with $u'\in W$ and consequently, $wu=w'u'\in W$.
So $$\xymatrix@R=1.5em{A\ar[r]^w & B\ar@{=}[dd]
\\
C\ar[u]^{u}\ar[d]_{u'}
\\
A'\ar[r]_{w'}&B
}$$
represents a vertical arrow as required.

To show that this arrow is unique, suppose that we have two representatives
$$
\xymatrix@R=1.5em{
A\ar[r]^w & B\ar@{=}[dd] &&A\ar[r]^w & B\ar@{=}[dd]
\\
C\ar[u]^{u}\ar[d]_{u'}&&\mbox{and}&D\ar[u]^{v}\ar[d]_{v'}&
\\
A'\ar[r]_{w'}&B&&A'\ar[r]_{w'}&B.
}$$
By conditions {\bf CF1} and {\bf CF2}, there are arrows  $\xymatrix@1{C&E\ar[l]_r\ar[r]^s& D}$ 
such that $wur=wvs\in W$. Now it follows that $w'v's=wvs=wur=w'u'r$ and $w'\in W$. 
By condition {\bf CF3}, there is an arrow $\xymatrix@1{E'\ar[r]^{\tilde w}& E}$ such that
$v's\tilde{w}=u'r\tilde{w}$. 
So the following diagram commutes:
$$
\xymatrix{
&A&
\\
C \ar[ur]^{u}\ar[dr]_{u'} & E'\ar[l]_{r\tilde{w}}\ar[r]^{s\tilde{w}}&D\ar[ul]_{v}\ar[dl]^{v'}
\\
&A'
}
$$
and hence $[u,C,u']=[v,D,v']$.
\end{proof}

Note that this implies that $\pi_0$ of the vertical arrow category of $\bfC\W$
is isomorphic to the set of objects of $\bfC$. 

The category of double cells in $\CW$ with vertical composition (and horizontal arrows as objects) is also
posetal, as stated in the following lemma.

\begin{lma}\label{uniquedblcell}
For any pairs of horizontal arrows and vertical arrows fitting together as in
$$
\xymatrix{
(w_1)\ar[d]|\bullet_{[u_1,C,u_2]} \ar[r]^{f_1} & (w'_1) \ar[d]|\bullet^{[v_1,D,v_2]}
\\
(w_2)\ar[r]_{f_2} & (w_2')
}
$$
there is at most one double cell in $\bfC\W$ that fills this.
\end{lma}

\begin{proof}
Suppose that both $(\varphi)$ and $(\psi)$ fit in this frame.
By Lemma \ref{reps} we may assume that  $\varphi$ and $\psi$ 
are of the following form:
$$
\xymatrix{
\ar@{=}[dd] & \ar[l]_{w_1}\ar[r]^{f_1}&\ar[r]^{w_1'} & \ar@{=}[dd]&&\ar@{=}[dd] & \ar[l]_{w_1}\ar[r]^{f_1}&\ar[r]^{w_1'} & \ar@{=}[dd]
\\
&\ar[u]^{u_1'}\ar[d]_{u_2'}\ar[r]^\varphi & \ar[u]_{v_1}\ar[d]^{v_2}&&\mbox{and}&&\ar[u]^{u_1''}\ar[d]_{u_2''}\ar[r]^\psi & \ar[u]_{v_1}\ar[d]^{v_2}
\\
&\ar[l]^{w_2}\ar[r]_{f_2}&\ar[r]_{w_2'}& &&&\ar[l]^{w_2}\ar[r]_{f_2}&\ar[r]_{w_2'}&
}
$$
Since $[u_1',u_2']=[u_1,u_2]=[u_1'',u_2'']$, there are arrows
$r'$ and $r''$ such that $w_1u_1'r',w_1u_1''r''\in W$
and 
the following diagram commutes
$$
\xymatrix{
&
\\
\ar[ur]^{u_1'}\ar[dr]_{u_2'}&\ar[l]_{r'}\ar[r]^{r''}&\ar[ul]_{u_1''}\ar[dl]^{u_2''}
\\
&
}
$$
Note that $w_1'v_1\varphi r'=w_1'f_1u_1'r'=w_1'f_1u_1''r''=w_1'v_1\psi r''$.
Since $w_1'v_1\in W$, we can apply condition {\bf CF3} to obtain an arrow $\tilde{w}$
such that $\varphi r'\tilde{w}=\psi r''\tilde{w}$. Finally,
we see that 
$$
\xymatrix@C=3em{
\ar@{=}[dd] & \ar[l]_{w_1}\ar[r]^{f_1}&\ar[r]^{w_1'} & \ar@{=}[dd]
		&&\ar@{=}[dd] & \ar[l]_{w_1}\ar[r]^{f_1}&\ar[r]^{w_1'} & \ar@{=}[dd]
\\
&\ar[u]^{u_1'}\ar[d]_{u_2'}\ar[r]^\varphi & \ar[u]_{v_1}\ar[d]^{v_2}
		&&\sim&&\ar[u]^{u_1'r'\tilde{w}}\ar[d]_{u_2'r'\tilde{w}}\ar[r]^{\varphi r'\tilde{w}} & \ar[u]_{v_1}\ar[d]^{v_2}&&=
\\
&\ar[l]^{w_2}\ar[r]_{f_2}&\ar[r]_{w_2'}& &&&\ar[l]^{w_2}\ar[r]_{f_2}&\ar[r]_{w_2'}&
\\
\ar@{=}[dd] & \ar[l]_{w_1}\ar[r]^{f_1}&\ar[r]^{w_1'} & \ar@{=}[dd]
		&&\ar@{=}[dd] & \ar[l]_{w_1}\ar[r]^{f_1}&\ar[r]^{w_1'} & \ar@{=}[dd]
\\
&\ar[u]^{u_1''r''\tilde{w}}\ar[d]_{u_2''r''\tilde{w}}\ar[r]^{\psi r'' \tilde{w}} & \ar[u]_{v_1}\ar[d]^{v_2}&
		&\sim&&\ar[u]^{u_1''}\ar[d]_{u_2''}\ar[r]^\psi & \ar[u]_{v_1}\ar[d]^{v_2}
\\
&\ar[l]^{w_2}\ar[r]_{f_2}&\ar[r]_{w_2'}& 
		&&&\ar[l]^{w_2}\ar[r]_{f_2}&\ar[r]_{w_2'}&
}
$$
So $(\varphi)$ and $(\psi)$ represent the same double cell.
\end{proof}

\begin{rmk}\rm Since the vertical arrow category of $\CW$ is posetal and groupoidal, this lemma implies that the
vertical category $(\CW)_1$  of horizontal arrows and double cells is also posetal and groupoidal.
\end{rmk}

\begin{thm}
The double category $\CW$ is weakly globular.
\end{thm}

\begin{proof}
By Lemma \ref{connectedcomp} the vertical arrow category of $\CW$ is groupoidal
and posetal as required. Furthermore,  the
codomain $d_1\colon (\CW)_1\rightarrow (\CW)_0$ is an isofibration.
For any pair of a horizontal arrow  $\xymatrix@1{(w_2)\ar[r]^g&(w_2')}$ and a vertical arrow 
$\xymatrix@1@C=3.5em{(w_1')\ar[r]|\bullet^{[v_1,D,v_2]}&(w_2')}$, there is a double cell 
$$
\xymatrix{
(w_1)\ar[r]^f\ar[d]|\bullet_{[u_1,C,u_2]}\ar@{}[dr]|{(\varphi)} & (w_1')\ar[d]|\bullet^{[v_1,D,v_2]}
\\
(w_2)\ar[r]_g & (w_2')
}
$$
by Lemma \ref{isofib}. And this double cell is vertically invertible since the vertical category $(\CW)_1$ is posetal groupoidal
by Lemma \ref{uniquedblcell}.
\end{proof}

\subsection{Composition}
{\em Composition of vertical arrows} is defined using condition {\bf CF2}. 
For a pair of composable vertical arrows,
$$
\xymatrix@R=1.2em{A_1\ar[r]^{w_1} & B\ar@{=}[dd]
\\
C\ar[u]^{u_1}\ar[d]_{u_2}&
\\
A_2\ar[r]_{w_2} & B\ar@{=}[dd]
\\
D\ar[u]^{v_1}
\ar[d]_{v_2} &
\\
A_3\ar[r]_{w_3} & B,
}
$$
there is a commutative square in $\bfC$ of the form
$$
\xymatrix{
E\ar[d]_{\bar{u}_2}\ar[r]^{\bar{v}_1}&C\ar[d]^{w_2u_2}
\\
D\ar[r]_{w_2v_1} & A_2
}
$$ 
with $u_2\bar{v}_1=v_1\bar{u}_2\in W$ (by Condition {\bf CF2}).
So $w_1u_1\bar{v}_1=w_2u_2\bar{v}_1=w_2v_1\bar{u}_2=w_3v_2\bar{u}_2$, and
a representative for the vertical composition is given by
$$
\xymatrix@R=1.6em{
A_1\ar[r]^{w_1} & B\ar@{=}[dd]
\\
D\ar[u]^{u_1\bar{v}_1}\ar[d]_{v_2\bar{u}_2}
\\
A_3\ar[r]_{w_3} &B
}
$$

Note that this is well-defined on equivalence classes, doesn't depend on the choice of the
square in condition {\bf CF2}, and is associative and unital, because the vertical arrow category is posetal.

Analogously, to define {\em vertical composition of double cells}, 
$$
\xymatrix@R=1.6em{
(B\ar@{=}[dd] & A_1)\ar[l]_{w_1}\ar[r]^{f_1} & (A_1'\ar[r]^{w_1'} & B')\ar@{=}[dd]
\\
&C\ar[u]^{u_1}\ar[d]_{u_2}\ar@{}[r]|{(\varphi)}&C'\ar[u]^{u_1'}\ar[d]_{u_2'}
\\
(B\ar@{=}[dd]&A_2)\ar[l]^{w_2}\ar[r]_{f_2}&(A_2'\ar[r]_{w_2'} & B')\ar@{=}[dd]
\\
&D\ar[u]^{v_1}
\ar[d]_{v_2} \ar@{}[r]|{(\psi)}&D'\ar[u]^{v_1'}
\ar[d]_{v_2'}
\\
(B&A_3)\ar[l]^{w_3}\ar[r]_{f_3}&(A_3'\ar[r]_{w_3} & B)
}
$$
we only need to give a representative to fit the square with
the composed boundary arrows. Without loss of generality we may assume that
the double  cells we are composing are such that
$\varphi\colon C\rightarrow C'$ and $\psi\colon D\rightarrow D'$ fit in the diagrams given.
So suppose that
we have commutative squares
$$
\xymatrix{
E\ar[r]^{\bar{v}_1}\ar[d]_{\bar{u}_2} & C\ar[d]^{u_2} & E'\ar[r]^{\bar{v}_1'}\ar[d]_{\bar{u}_2'} & C\ar[d]^{u_2'}
\\
D\ar[r]_{v_1} & A_2 & D'\ar[r]_{v_1'}& A_2',
}
$$
giving the following composition of the vertical arrows involved:
$$
\xymatrix@C=4em@R=1.6em{
(B\ar@{=}[dd] & \ar[l]_{w_1} A_1)\ar[r]^{f_1} & (A_1'\ar[r]^{w_1'}& B')\ar@{=}[dd]
\\
& E\ar[u]^{u_1\bar{v}_1}\ar[d]_{v_2\bar{u}_2}\ar@{}[r]|{(\psi)\cdot(\varphi)} 
	& E'\ar[u]^{u_1'\bar{v}_1'}\ar[d]_{v_2'\bar{u}_2'}
\\
(B & \ar[l]^{w_3} A_3) \ar[r]_{f_3} & (A_3'\ar[r]_{w_3'} & B')
}
$$
We will now construct $(\xi)=(\psi)\cdot(\varphi)$.
Note that it is sufficient to find an arrow $\xymatrix@1{\bar{E}\ar[r]^r&E}$ in $W$
with an arrow $\xi\colon \bar{E}\rightarrow E'$
such that $\bar{v}_1'\xi =\varphi \bar{v}_1r$ and $\bar{u}_2'\xi=\psi\bar{u}_2r$.

Apply condition {\bf CF2} to obtain a commutative square
$$\xymatrix{
\tilde{E}\ar[d]_{\bar{\bar{v}}_1'}\ar[rr]^{\bar{\varphi}} &&E'\ar[d]^{\bar{v}_1'}
\\
E\ar[r]_{\bar{v}_1} & C\ar[r]_\varphi & C'.
}
$$ in $\bfC$ with $\bar{\bar{v}}_1'\in W$.
Then
\begin{eqnarray*}
w_2'v_1'\bar{u}_2'\bar{\varphi}&=&w_2'u_2'\bar{v}_1'\bar{\varphi}\\
&=&w_2'u_2'\varphi\bar{v}_1\bar{\bar{v}}_1'\\
&=&w_2'f_2u_2\bar{v}_1 \bar{\bar{v}}_1'\\
&=&w_2'f_2v_1\bar{u}_2\bar{\bar{v}}_1'\\
&=&w_2'v_1'\psi\bar{u}_2\bar{\bar{v}}_1'
\end{eqnarray*}
Since $w_2'v_1'\in W$, there is an arrow $s\colon \bar{E}\rightarrow \tilde{E}$
such that $\bar{u}_2'\bar{\varphi}s=\psi\bar{u}_2\bar{\bar{v}}_1's$ (by condition {\bf CF3}).
So let $r=\bar{\bar{v}}_1's$ and $\xi=\bar{\varphi}s$.

{\em Composition of horizontal arrows} in $\bfC\W$ is defined directly in terms of composition of arrows in $\bfC$:
the horizontal composition of $\xymatrix{(B&\ar[l]_{w}A)\ar[r]^f &(A'\ar[r]^{w'}&B')}$
and $\xymatrix{(B'& \ar[l]_{w'}A)\ar[r]^{f'} &(A''\ar[r]^{w''}&B'')}$
is $\xymatrix{(B&\ar[l]_{w}A)\ar[r]^{f'f} &(A'\ar[r]^{w''}&B'')}$.

{\em Horizontal composition of double cells} is defined as follows.
Let 
$$
\xymatrix{
\ar@{=}[dd] & \ar[l]_{w_1}\ar[r]^{f_1}&\ar[r]^{w_1'} & \ar@{=}[dd]
		&&\ar@{=}[dd] & \ar[l]_{w_1'}\ar[r]^{g_1}&\ar[r]^{w_1''} & \ar@{=}[dd]
\\
&\ar[u]^{u_1}\ar[d]_{u_2}\ar[r]^\varphi & \ar[u]_{u_1'}\ar[d]^{u_2'}&&\mbox{and}&&\ar[u]^{v_1}\ar[d]_{v_2}\ar[r]^\psi & \ar[u]_{v_1'}\ar[d]^{v_2'}
\\
&\ar[l]^{w_2}\ar[r]_{f_2}&\ar[r]_{w_2'}& &&&\ar[l]^{w_2'}\ar[r]_{g_2}&\ar[r]_{w_2''}&
}
$$
be two horizontally composable double cells.
This means that $[u_1',u_2']=[v_1,v_2]$, so there exist arrows $r$ and $s$ such that
$w_1'u_1'r\in W$ and the following diagram commutes,
$$
\xymatrix{
&
\\
\ar[ur]^{u_1'} \ar[dr]_{u_2'}& \ar[l]_r\ar[r]^s & \ar[ul]_{v_1}\ar[dl]^{v_2}
\\
&
}
$$
By conditions {\bf CF2} and {\bf CF3} there is a commutative square of the form
$$
\xymatrix{
\ar[d]_{\bar{r}}\ar[r]^{\bar{\varphi}} & \ar[d]^r\\
\ar[r]_\varphi&
}
$$
with $\bar{r}\in W$.
Now it is not difficult to check that
$$
\xymatrix{
\ar@{=}[dd] & \ar[l]_{w_1}\ar[r]^{f_1}&\ar[r]^{w_1'} & \ar@{=}[dd]
		&&\ar@{=}[dd] & \ar[l]_{w_1}\ar[r]^{f_1}&\ar[r]^{w_1'} & \ar@{=}[dd]
\\
&\ar[u]^{u_1}\ar[d]_{u_2}\ar[r]^\varphi & \ar[u]_{u_1'}\ar[d]^{u_2'}&
		&\sim&&\ar[u]^{u_1\bar{r}}\ar[d]_{u_2\bar{r}}\ar[r]^{\bar{\varphi}} & \ar[u]_{u_1'r}\ar[d]^{u_2'r}
\\
&\ar[l]^{w_2}\ar[r]_{f_2}&\ar[r]_{w_2'}& &&&\ar[l]^{w_2}\ar[r]_{f_2}&\ar[r]_{w_2'}&
}
$$
and
$$
\xymatrix{
\ar@{=}[dd] & \ar[l]_{w_1'}\ar[r]^{g_1}&\ar[r]^{w_1''} & \ar@{=}[dd]
		&&\ar@{=}[dd] & \ar[l]_{w_1'}\ar[r]^{g_1}&\ar[r]^{w_1''} & \ar@{=}[dd]
\\
&\ar[u]^{v_1}\ar[d]_{v_2}\ar[r]^\psi & \ar[u]_{v_1'}\ar[d]^{v_2'}&
		&\sim&&\ar[u]^{v_1s}\ar[d]_{v_2s}\ar[r]^{\psi s} & \ar[u]_{v_1'}\ar[d]^{v_2'}
\\
&\ar[l]^{w_2'}\ar[r]_{g_2}&\ar[r]_{w_2''}& &&&\ar[l]^{w_2'}\ar[r]_{g_2}&\ar[r]_{w_2''}&
}
$$
Since $u_1'r=v_1s$ and $u_2'r=v_2s$, the horizontal composition is given by
$$
\xymatrix{
\ar@{=}[dd] &\ar[l]_{w_1}\ar[r]^{g_1f_1} & \ar[r]^{w_1''}&\ar@{=}[dd]
\\
&\ar[u]^{u_1\bar{r}}\ar[d]_{u_2\bar{r}}\ar[r]^{\psi s\bar{\varphi}}& \ar[u]_{v_1'}\ar[d]^{v_2'}
\\
&\ar[l]^{w_2}\ar[r]_{g_2f_2}&\ar[r]_{w_2''} &\rlap{\,.}
}
$$

Middle four interchange holds because the double cells are posetal in the vertical direction.

\begin{rmk}
{\rm For any objects $w_1,w_2\in \CW_{00}$, the hom-set $\CW_h(w_1,w_2)$ is small 
since it is isomorphic to ${\bfC}(d_0(w_1),d_0(w_2))$  and $\CW_v(A,B)$ is small since
it contains at most one element. Furthermore, for any given pair of horizontal or vertical arrows,
there is a set of double cells  with those arrows as domain and codomain, since given any frame 
of horizontal and vertical arrows, there is at most one double cell that fills it.
This shows that although $\bfC(W^{-1})$ is generally not a bicategory with small hom-sets (and consequently, 
its weakly globular double category of marked paths does not have small horizontal hom-sets either),
the double category $\CW$ does. This is one of the advantages of $\CW$ over $\bfC(W^{-1})$.}
\end{rmk}

\subsection{The fundamental bicategory of $\CW$}
In this section we establish the relationship between the weakly globular double category $\bfC\W$
and the bicategory of fractions $\bfC(W^{-1})$. We will construct a biequivalence of bicategories
$$\omega\colon\Bic(\CW)\rightarrow \bfC(W^{-1}),$$
with pseudo-inverse $\alpha\colon \bfC(W^{-1})\rightarrow \Bic(\CW)$.

The objects of the associated bicategory $\Bic(\CW)$ are the connected components 
of the vertical arrow category of 
$\bfC\W$. By Lemma \ref{connectedcomp} the map 
$$\omega_0\colon\Bic(\CW)_0\rightarrow \bfC(W^{-1})_0=\bfC_0,$$
sending the connected component of $\xymatrix@1{(A\ar[r]^w&B)}$ to $B$ is an isomorphism.

An arrow in $\Bic(\CW)$ corresponds to a horizontal arrow in $\CW$, so it is given by a diagram
$\xymatrix@1{B&A\ar[l]_{w}\ar[r]^f & A'\ar[r]^{w'}&B'}$ in $\bfC$ with $w,w'\in W$.
We define 
$$\omega_1(\xymatrix{B&A\ar[l]_{w}\ar[r]^f & A'\ar[r]^{w'}&B'})=(\xymatrix{B&A\ar[l]_-w\ar[r]^-{w'f}&B'}).$$

A 2-cell $\varphi$ from $\xymatrix@1{B&A_1\ar[l]_{w_1}\ar[r]^{f_1} & A_1'\ar[r]^{w_1'}&B'}$
to $\xymatrix@1{B&A_2\ar[l]_{w_2}\ar[r]^{f_2} & A_2'\ar[r]^{w_2'}&B'}$
in $\Bic(\CW)$
corresponds to an equivalence class of commutative diagrams of the form
$$
\xymatrix@R=1.5em{
B\ar@{=}[dd]&A_1\ar[l]_{w_1}\ar[r]^{f_1} & A_1'\ar[r]^{w_1'}&B'\ar@{=}[dd]
\\
&C\ar[u]^{u_1}\ar[d]_{u_2}\ar[r]^\varphi & C'\ar[u]_{v_1}\ar[d]^{v_2}
\\
B&A_2\ar[l]_{w_2}\ar[r]^{f_2} & A_2'\ar[r]^{w_2'}&B'
}$$
This means that $w_1'f_1u_1=w_1'v_1\varphi=w_2'v_2\varphi=w_2'f_2u_2$ and $w_1u_1=w_2u_2$, so
we define 
$$\omega_2(\varphi)=\left(\raisebox{3.2em}{$
\xymatrix@R=1.8em{ & A_1\ar[dl]_{w_1}\ar[dr]^{w_1'f_1}
\\
B&C\ar[u]^{u_1}\ar[d]_{u_2} & B'
\\
&A_2\ar[ul]^{w_2}\ar[ur]_{w_2'f_2}
}$}\right)$$
The fact that this is well defined on equivalence classes of double cells
follows from the following lemma about the 2-cells in a bicategory of fractions of a category.

\begin{lma}
For  any category $\bfC$ with a class of arrows $W$ satisfying the conditions {\bf CF1}, {\bf CF2}, and {\bf CF3} 
above, the bicategory of fractions $\bfC(W^{-1})$ has at most one 2-cell between any two arrows.
\end{lma}

\begin{proof}
The arrows in this bicategory are spans $\xymatrix@1{A&S\ar[l]_w\ar[r]^f & B}$ with $w\in W$.
A 2-cell from  $\xymatrix@1{A&S_1\ar[l]_{w_1}\ar[r]^{f_1} & B}$
to $\xymatrix@1{A&S_2\ar[l]_{w_2}\ar[r]^{f_2} & B}$ is represented by a commutative diagram of the form
$$
\xymatrix@R=1.3em{
&S_1\ar[dl]_{w_1}\ar[dr]^{f_1}\\
A &T\ar[u]_{u_1}\ar[d]^{u_2}&B\\
&S_2\ar[ul]^{w_2}\ar[ur]_{f_2}
}
$$
Two such diagrams,
\begin{equation}\label{two cells}
\xymatrix@R=1.3em{
&S_1\ar[dl]_{w_1}\ar[dr]^{f_1}&&& &S_1\ar[dl]_{w_1}\ar[dr]^{f_1}\\
A &T\ar[u]_{u_1}\ar[d]^{u_2}&B&\mbox{and} & A &T'\ar[u]_{u'_1}\ar[d]^{u'_2}&B\\
&S_2\ar[ul]^{w_2}\ar[ur]_{f_2} &&& &S_2\ar[ul]^{w_2}\ar[ur]_{f_2}
}
\end{equation}
 represent the same 2-cell when there exist arrows $t$ and $t'$ such that the
 diagram 
\begin{equation}\label{equivalence}
\xymatrix@R=1.3em{
&S_1
\\
T\ar[ur]^{u_1} \ar[dr]_{u_2}& \bar{T}\ar[l]_t\ar[r]^{t'} & T'\ar[ul]_{u_1'}
\ar[dl]^{u_2'} 
\\
&S_2
}
\end{equation}
commutes and $w_1u_1t\in W$.

Note that for any two 2-cells as in (\ref{two cells}), we can find $t$ and $t'$ as in (\ref{equivalence})
by first using condition {\bf CF2} to find a commutative square
$$\xymatrix{
\ar[r]^r\ar[d]_{r'} &\ar[d]^{w_1u_1}
\\
\ar[r]_{w_1u_1'}&\rlap{\,.}
}
$$
We have then also that $w_2u_2r=w_1u_1r=w_1u_1'r'=w_2u_2'r'$.
By condition {\bf CF3} applied to both $w_1$ and $w_2$ we find that there
is an arrow $\tilde{w}\in W$ such that  $u_1r\tilde{w}=u_1'r'\tilde{w}$ and $u_2r\tilde{w}=u_2'r'\tilde{w}$.
So $t=r\tilde{w}$ and $t'=r'\tilde{w}$ fit in the diagram (\ref{equivalence}).
\end{proof}

\begin{thm}\label{equiv}
The bicategories, $\Bic(\bfC\W)$ and $\bfC(W^{-1})$ are equivalent in the 2-category
$\NHom$.
\end{thm} 

\begin{proof}
We first show that the functions $\omega_0$, $\omega_1$ and $\omega_2$ defined above give 
a homomorphism of bicategories:
$\omega\colon\Bic(\CW)\rightarrow\bfC(W^{-1})$.
To prove this we only need to give the comparison 2-cells for units and composition. 
(They will automatically be coherent since the hom-categories in $\bfC(W^{-1})$ are posetal.)

We will represent each connected component of $(\CW)_0$ by the identity arrow in it.
So an identity arrow in $\Bic(\CW)$ has the form $\xymatrix@1{B&B\ar[l]_{1_B}\ar[r]^{1_B} & B\ar[r]^{1_B}&B}$,
and $$\omega_1\left(\xymatrix@1{B&B\ar[l]_{1_B}\ar[r]^{1_B} & B\ar[r]^{1_B}&B}\right)=
\left(\xymatrix@1{B&B\ar[l]_{1_B}\ar[r]^{1_B}&B}\right).$$ So $\omega_1$ preserves identities strictly.

For any two composable arrows 
\begin{equation}\label{composable}
 \xymatrix@1{B_1&\ar[l]_{w_1}A_1\ar[r]^{f}&A_2\ar[r]^{w_2}&B_2}
\mbox{ and }\xymatrix@1{B_2&A_3\ar[l]_{w_3}\ar[r]^{g}&A_4\ar[r]^{w_4}&B_3}
\end{equation}
in $\Bic(\CW)$,
the composition is found by first constructing the following double cell in $\CW$,
$$
\xymatrix@R=1.5em{
(B_1\ar@{=}[dd] 
	&C)\ar[l]_{w_1\bar{\bar{w_3}}}\ar@{=}[d]\ar[r]^{\bar{w}_2\bar{f}} 
	& (A_3\ar[r]^{w_3}&B_2)\ar@{=}[dd]
\\
&C\ar[d]_{\bar{\bar{w}}_3}\ar[r]^{\bar{f}} &D\ar[u]_{\bar{w}_2}\ar[d]^{\bar{w}_3} 
\\
(B_1&\ar[l]_{w_1}A_1)\ar[r]^{f}&(A_2\ar[r]^{w_2}&B_2)
}$$
(using chosen squares in $\bfC$ to obtain $\bar{w}_2$, $\bar{w}_3$, $\bar{\bar{w}}_3$ and $\bar{f}$)
and then composing the domain horizontal arrow with $\xymatrix@1{(B_2&A_3)\ar[l]_{w_3}\ar[r]^{g}&(A_4\ar[r]^{w_4}&B_3)}$ to get
$$
 \xymatrix@1{(B_1&\ar[l]_{w_1\bar{\bar{w}}_3}C)\ar[rr]^{g\bar{w}_2\bar{f}}&&(A_4\ar[r]^{w_4}&B_3)\rlap{\,.}}
$$
So $\omega_1$ applied to the composite gives
\begin{equation}\label{imageofcomp}
 \xymatrix@C=4em{B_1&C\ar[l]_{w_1\bar{\bar{w}}_3}\ar[r]^{w_4g\bar{w}_2\bar{f}}&B_3\rlap{\,.}}
\end{equation}
On the other hand, the composite of the images of the arrows in (\ref{composable}) under $\omega$
is
\begin{equation}\label{compofimages}
 \xymatrix@C=4em{B_1&E\ar[l]_{w_1\hat{w}_3}\ar[r]^{w_4g\widehat{w_2f}}&B_3\rlap{\,,}}
\end{equation}
where $\hat{w}_3$ and $\widehat{w_2f}$ are obtained using chosen squares.

To construct a 2-cell from (\ref{imageofcomp}) to (\ref{compofimages}) we consider the following diagram
containing the three chosen squares used above:
$$
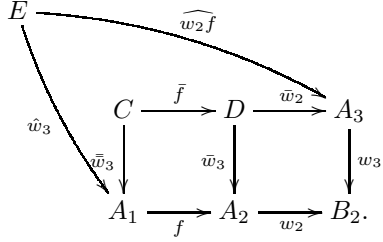
\xymatrix{
E\ar@/_1ex/[ddr]_{\hat{w}_3}\ar@/^2ex/[drrr]^{\widehat{w_2f}}
\\
&C\ar[d]_{\bar{\bar{w}}_3}\ar[r]^{\bar{f}} & D\ar[d]_{\bar{w}_3}\ar[r]^{\bar{w}_2} & A_3\ar[d]^{w_3}
\\
& A_1\ar[r]_f&A_2\ar[r]_{w_2} & B_2.
}
$$
We want to construct a span $E\leftarrow G\rightarrow C$ which will make this diagram commutative.

Using Condition {\bf CF2} we find a commutative square 
$$
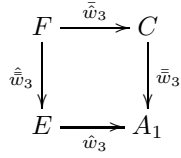
\xymatrix{
F\ar[r]^{\bar{\hat{w}}_3}\ar[d]_{\hat{\bar{\bar{w}}}_3} & C\ar[d]^{\bar{\bar{w}}_3}
\\
E\ar[r]_{\hat{w}_3} & A_1}
$$
and using Condition {\bf CF3} we find an arrow $\xymatrix@1{G\ar[r]^{\tilde{w}_3}&F}$ in $W$
such that all parts of the following diagram commute:
$$
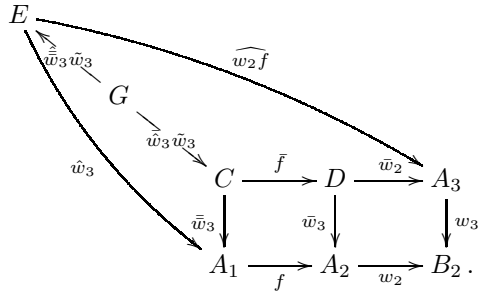
\xymatrix@R=1.8em{
E\ar@/_2ex/[3,2]_{\hat{w}_3}\ar@/^2ex/[2,4]^{\widehat{w_2f}}
\\
& G\ar[ul]|{\hat{\bar{\bar{w}}}_3\tilde{w}_3}\ar[dr]|{\bar{\hat{w}}_3\tilde{w}_3}
\\
&&C\ar[d]_{\bar{\bar{w}}_3}\ar[r]^{\bar{f}} & D\ar[d]_{\bar{w}_3}\ar[r]^{\bar{w}_2} & A_3\ar[d]^{w_3}
\\
&& A_1\ar[r]_f&A_2\ar[r]_{w_2} & B_2\rlap{\,.}
}
$$
So the required comparison 2-cell is given by
$$
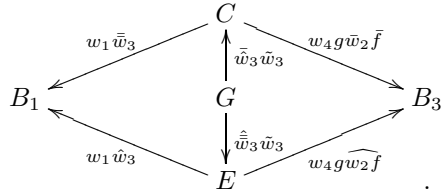
\xymatrix@C=6em@R=1.8em{
&C\ar[dl]_{w_1\bar{\bar{w}}_3}\ar[dr]^{w_4g\bar{w}_2\bar{f}}
\\
B_1 & G\ar[u]_{\bar{\hat{w}}_3\tilde{w}_3}\ar[d]^{\hat{\bar{\bar{w}}}_3\tilde{w}_3} & B_3
\\
&E\ar[ul]^{w_1\hat{w}_3}\ar[ur]_{w_4g\widehat{w_2f}}&.
}$$

We will now construct a pseudo-inverse $\alpha\colon\bfC(W^{-1})\rightarrow\Bic(\bfC\{W\})$
for $\omega$. 

On objects, $\alpha(A)=(1_A)$. On arrows, $$\alpha(\xymatrix@1{A&C\ar[l]_w\ar[r]^f & B})=
( \xymatrix@1{A&C\ar[l]_w\ar[r]^f &B\ar[r]^{1_B} & B}),$$ and on 2-cells,
$$
\alpha\left(\raisebox{3.1em}{$\xymatrix{&\ar[dl]_w\ar[dr]^f\\&\ar[u]^v\ar[d]_{v'}&\\&\ar[ul]^{w'}\ar[ur]_{f'}}$}\right)=
\raisebox{3.5em}{$\xymatrix{\ar@{=}[dd]&\ar[l]_w\ar[r]^f & \ar@{=}[r]&\ar@{=}[dd]
\\ &\ar[r]^{fv}\ar[u]_{v}\ar[d]^{v'} &\ar@{=}[d]\ar@{=}[u] \\&\ar[l]^{w'}\ar[r]_{f'} & \ar@{=}[r] &\rlap{\,.}}$}
$$
We will show that $\alpha$ is a homomorphism of bicategories.
First, note that $\alpha$ preserves units strictly:
$\alpha(\xymatrix@1{&\ar[l]_{1_A}\ar[r]^{1_A}&})=(\xymatrix@1{&\ar[l]_{1_A}\ar[r]^{1_A} & \ar[r]^{1_A}&})$.
Composition is also preserved strictly as long as we choose the same chosen squares.
We will now show this.
So let $\xymatrix@1{A&\ar[l]_{w_1}S\ar[r]^{f_1} & B}$ and $\xymatrix@1{B&T\ar[l]_{w_2}\ar[r]^{f_2}& C}$
be a composable pair of arrows in $\bfC(W^{-1})$. Let 
\begin{equation}\label{chosen}
\xymatrix{
U\ar[r]^{\bar{f}_1}\ar[d]_{\bar{w}_2} & T\ar[d]^{w_2}
\\
S\ar[r]_{f_1} & B
}\end{equation}
be a chosen square. Then the composition in $\bfC(W^{-1})$ is 
$\xymatrix@1{A & U\ar[l]_{w_1\bar{w}_2}\ar[r]^{f_2\bar{f}_1} & C}$,
and $$\alpha(\xymatrix@1{A & U\ar[l]_{w_1\bar{w}_2}\ar[r]^{f_2\bar{f}_1} & C})=
(\xymatrix@1{A&U\ar[l]_{w_1\bar{w}_2}\ar[r]^{f_2\bar{f}_1}&C\ar@{=}[r]&C}).$$
To calculate the composition in $\Bic(\CW)$ of 
$$\alpha(\xymatrix@1{A&\ar[l]_{w_1}S\ar[r]^{f_1} & B})=
(\xymatrix@1{A&S\ar[l]_{w_1}\ar[r]^{f_1}&B\ar@{=}[r]&B})$$ and 
$$\alpha(\xymatrix@1{B&\ar[l]_{w_2}T\ar[r]^{f_1} & C})=
(\xymatrix@1{B&T\ar[l]_{w_2}\ar[r]^{f_2}&C\ar@{=}[r]&C}),$$ we note that we can use 
the same chosen square (\ref{chosen})
to construct a pair of horizontal arrows in $\CW$ that is composable there:
there is a double cell
$$
\xymatrix{
(\ar@{=}[dd]&)\ar[l]_{w_1\bar{w}_2}\ar[r]^{\bar{f}_1} & (\ar[r]^{w_2}&)\ar@{=}[dd]
\\
&\ar[d]_{\bar{w}_2}\ar@{=}[u]\ar[r]^{\bar{f}_1} &\ar[d]^{w_2}\ar@{=}[u]
\\
(&\ar[l]^{w_1})\ar[r]_{f_1} & (\ar@{=}[r]&)
}
$$
Composing the top row of this double cell with $\xymatrix@1{(&)\ar[l]_{w_2}\ar[r]^{f_2}&(\ar@{=}[r]&)}$
gives us
$$\xymatrix{(&)\ar[l]_{w_1\bar{w}_2}\ar[r]^{f_2\bar{f}_1} & (\ar@{=}[r]&)}.$$

Now it remains for us to calculate the composites $\alpha\circ\omega$ and $\omega\circ\alpha$.
First, a straightforward calculation shows that $\omega\alpha=\mbox{Id}_{\bfC(W^{-1})}$.
The other composition requires a bit more attention.

On objects, $\alpha\omega(\overline{(\stackrel{w}{\rightarrow})})=\overline{(\stackrel{1_B}{\rightarrow})}$
for $A\stackrel{w}{\rightarrow}B$, and $\overline{(\stackrel{w}{\rightarrow})}=\overline{(\stackrel{1_B}{\rightarrow})}$
in $\Bic(\CW)$. So  $\alpha\omega$ is the identity on objects.

On arrows, $\alpha\omega(\xymatrix@1{&\ar[l]_{w_1}\ar[r]^f&\ar[r]^{w_2}&})=
(\xymatrix@1{&\ar[l]_{w_1}\ar[r]^{w_2f}&\ar@{=}[r]&})$. And on 2-cells, 
\begin{eqnarray*}
\alpha\omega \left(\raisebox{3.5em}{$\xymatrix{\ar@{=}[dd]&\ar[l]_{w_1}\ar[r]^{f} &\ar[r]^{w_2}&\ar@{=}[dd]
\\
&\ar[u]_{u_1}\ar[r]^h\ar[d]^{v_1} & \ar[u]_{u_2}\ar[d]^{v_2}
\\
&\ar[l]_{w_1'}\ar[r]_{f'} & \ar[r]_{w_2'}& }$}\right)&=&
\alpha\left(\raisebox{3.1em}{$\xymatrix{&\ar[dl]_{w_1}\ar[dr]^{w_2f}
\\
&\ar[u]_{u_1}\ar[d]^{v_1}&
\\
&\ar[ul]^{w_1'}\ar[ur]_{w_2'f'}}$}\right)
\\
&=& \raisebox{3.5em}{$\xymatrix{\ar@{=}[dd]&\ar[l]_{w_1}\ar[r]^{w_2f} &\ar@{=}[r]&\ar@{=}[dd]
\\
&\ar[u]_{u_1}\ar[r]^{w_2fu_1}\ar[d]^{v_1} & \ar@{=}[u]\ar@{=}[d]
\\
&\ar[l]_{w_1'}\ar[r]_{w_2'f'} & \ar@{=}[r]&\rlap{\,.}
}$}
\end{eqnarray*}

So we see that $\alpha\circ\omega$ is not the identity, but there
is an invertible icon $\zeta\colon\alpha\omega\Rightarrow\mbox{Id}_{\smBic(\bfC(W)^{-1})}$
such that the component of $\zeta$ at $\xymatrix@1{&\ar[l]_{w_1}\ar[r]^f&\ar[r]^{w_2}&}$ is
$$
\xymatrix{
\ar@{=}[dd]&\ar[l]_{w_1}\ar[r]^{w_2f} &\ar@{=}[r]&\ar@{=}[dd]
\\
&\ar@{=}[u]\ar[r]^f\ar@{=}[d] & \ar[u]_{w_2}\ar@{=}[d]
\\
&\ar[l]^{w_1}\ar[r]_{f} & \ar[r]_{w_2}& 
}
$$

The fact that the naturality squares of 2-cells commute follows from the fact that this bicategory is
locally posetal. 

We conclude that $\omega$ and $\alpha$ form a biequivalence of bicategories $\Bic(\CW)\simeq\bfC(W^{-1})$.
\end{proof}

\begin{cor}\label{2-equiv}
 $\CW$ and $\Dbl(\bfC(W^{-1}))$ are equivalent in the 2-category $\WGDbl_{\mbox{\scriptsize\rm ps},v}$.
\end{cor}

\begin{proof}
Apply $\Dbl$ to the equivalence of Theorem \ref{equiv} and then compose the resulting vertical equivalence of 
weakly globular double categories with the equivalence of $\CW$ with $\Dbl\Bic(\CW)$.
\end{proof}

\section{The Universal Property of $\CW$} \label{fracns}

The goal of this section is to describe the vertical and horizontal universal properties 
of the weakly globular double category $\CW$ we constructed in the previous section. 
First we derive the universal property of $\CW$ inherited from
$\bfC(W^{-1})$ by Corollary \ref{2-equiv}. This is a property in terms of 
pseudo-functors and vertical transformations.
The second universal property we give for $\CW$ is in terms of strict double functors 
and horizontal transformations. The two properties together determine $\CW$ 
up to horizontal and vertical equivalence.

\subsection{The vertical universal property of $\CW$} 
Recall that Theorem \ref{univ1-thm} states that composition with 
$\Dbl(J_W)\colon \Dbl(\bfC)\rightarrow\Dbl(\bfC(W^{-1}))$
induces an equivalence of categories,
$$
\WGDbl_{\ps,v}(\Dbl(\bfC(W^{-1})),\bbX)\simeq\WGDbl_{\ps,v,W}(\Dbl(\bfC),\bbX)
$$
for any weakly globular double category $\bbX$.

To obtain a vertical universal property for $\CW$, note that $\CW$ is vertically equivalent to $\Dbl(\bfC(W^{-1}))$.
We compose $\Dbl(J_W)$ with the equivalence $\xymatrix@1{\Dbl(\bfC(W^{-1}))\ar[r]^-{\sim}&\CW}$
to get a pseudo-functor $\tilde{J}_W\colon \Dbl(\bfC)\rightarrow\CW$ with the following universal property:

\begin{thm}\label{D:firstthm}
Composition with the pseudo-functor 
$\tilde{J}_W\colon \Dbl(\bfC)\rightarrow \CW,$
gives rise to an equivalence of categories,
$$
\WGDbl_{\ps,v}(\CW,\bbD)\simeq\WGDbl_{\ps,v,W}(\Dbl(\bfC),\bbD).
$$
\end{thm}

Recall that $\WGDbl_{v,{\ps},W}(\Dbl(\bfC),\bbD)$ is the category of $W$-pseudo-functors (which 
send $W$-doubly marked paths to precompanions), and vertical $W$-transformations
as described in Definition \ref{W-trafo}. 

Since furthermore, $H(\bfC)$ is equivalent to $\Dbl(\bfC)$, we may compose $\tilde{J}_W$ with this equivalence and obtain 
$\calJ_{W}\colon H(\bfC)\rightarrow \CW$.
This functor has the universal property expressed in the following corollary.

\begin{cor}
 Composition with $\calJ_{W}\colon H\bfC\rightarrow \CW$ gives rise to an equivalence of categories,
$$
\WGDbl_{\ps,v}(\CW,\bbD)\simeq\WGDbl_{\ps,v,W}(H\bfC,\bbD).
$$
\end{cor}

\begin{rmks}\rm
 \begin{enumerate}
 \item In this corollary, $\WGDbl_{\ps,v,W}(H\bfC,\bbD)$ has as objects pseudo-functors which send the elements of $W$ to
precompanions. Note that this vertical universal property is easier to express than the one in terms of $\Dbl(\bfC)$, as given in Theorem \ref{D:firstthm}.
 \item This characterization of $\CW$ determines this weakly globular double category up to a vertical equivalence.
\item
It is straightforward to check that the composition $\calJ_{W}\colon H\bfC\rightarrow\CW$ used in the last corollary
is the obvious inclusion functor, sending and object of $C$ to the object in $\CW$ represented by its identity arrow. 
It is clear that this is a strict functor of weakly globular double categories. 
 \end{enumerate}
\end{rmks}

In the rest of this section we will discuss 
the universal property of the functor $\calJ_{W}\colon H\bfC\rightarrow\CW$  
in terms of 2-categories of strict functors and horizontal transformations.
This will then characterize $\CW$ up to a horizontal equivalence.
This characterization will be in terms of companions rather than precompanions.

\subsection{Companions and conjoints in $\CW$}
As we saw in an Section \ref{compquasi}, companions in weakly globular double categories are closely related to 
quasi units in bicategories. Requiring in the universal property  that the elements of $W$ themselves become 
companions is too strong, 
but requiring that they become companions after composition with a horizontal isomorphism 
turns out to be just right for a universal property with respect to strict functors and horizontal transformations.

For this part we will make the added assumption on the class $W$ that it satisfies the 3 for 2 property:
if $f,g,h$ are arrows in $\bfC$ such that $gf=h$ and two of $f,g$ and $h$ are in $W$, so is the third.

The following lemma characterizes the horizontal and vertical arrows in $\bfC\W$ that have 
companions or conjoints.

\begin{lma}\label{compchar}
\begin{enumerate}
 \item 
A vertical arrow in $\CW$ has a horizontal companion if and only if it has a representative of the form
\begin{equation}\label{vertcomp}
\xymatrix@R=1.5em{
(A_1\ar[r]^{wu} & B)\ar@{=}[dd]
\\
C\ar@{=}[u]\ar[d]_u 
\\
(A_2\ar[r]_{w}&B)
}
\end{equation}
\item
A vertical arrow in $\CW$ has a horizontal conjoint if and only if it has a representative of the form
\begin{equation}\label{vertconj}
\xymatrix@R=1.5em{
(A_1\ar[r]^{w} & B)\ar@{=}[dd]
\\
A_2\ar@{=}[d]\ar[u]^u 
\\
(A_2\ar[r]_{wu}&B)
}
\end{equation}
\item
A horizontal arrow in $\CW$ has a vertical companion and conjoint if and only if
it is of the form
$$
 \xymatrix{
(B&\ar[l]_{wu} A_2)\ar[r]^u&(A_1\ar[r]^{w} & B)\rlap{\,.}
}
$$ 
\end{enumerate}
\end{lma}

\begin{proof}
To prove the first part, the companion for (\ref{vertcomp}) is
the horizontal arrow
$$
\xymatrix@1{(B&\ar[l]_{wu} A_1)\ar[r]^{u}&(A_2\ar[r]^w&B)}
$$
with binding cells
\begin{equation}\label{bindingcells}
\xymatrix@R=1.5em{
&(B\ar@{=}[dd] & A_1)\ar[l]_{wu}\ar@{=}[r] & (A_1\ar[r]^{wu}&B)\ar@{=}[dd]
\\
\ar@{}[r]|(0.2){\textstyle \psi_{u,w}=}&&A_1\ar@{=}[d]\ar@{=}[u] \ar@{=}[r] & A_1\ar@{=}[u]\ar[d]_{u}&
\\
&(B&\ar[l]^{wu}A_1)\ar[r]_{u} & (A_2\ar[r]_w & B)
\\
&(B\ar@{=}[dd] & A_1)\ar[l]_{wu}\ar[r]^u & (A_2\ar[r]^w & B)\ar@{=}[dd]
\\
\ar@{}[r]|(0.2){\textstyle \chi_{u,w}=}& & A_1\ar[d]_u\ar@{=}[u] \ar[r]^u & A_2\ar@{=}[u]\ar@{=}[d]&
\\
&(B&\ar[l]^{w}A_2)\ar@{=}[r] & (A_2\ar[r]_w & B)
}
\end{equation}
We leave it up to the reader to check that these cells satisfy the equations (\ref{compcondns}).

Conversely, suppose that a vertical arrow
\begin{equation}\label{vertarrow}
\xymatrix@R=1.5em{
A_1\ar[r]^{w_1} & B\ar@{=}[dd]
\\
C\ar[u]^{u_1}\ar[d]_{u_2} 
\\
A_2\ar[r]_{w_2}&B
}
\end{equation}
has a companion $\xymatrix@1{(B&\ar[l]_{w_1}A_1)\ar[r]^{f}&(A_2\ar[r]^{w_2}&B)}$. Then there is 
a binding cell
$$
\xymatrix@R=1.5em{
(B\ar@{=}[dd] & A_1)\ar[l]_{w_1}\ar[r]^f & (A_2\ar[r]^{w_2} & B)\ar@{=}[dd]
\\
 & C\ar[d]_{u_2}\ar[u]^{u_1} \ar@{}[r]|{(\chi)} & A_2\ar@{=}[u]\ar@{=}[d]&
\\
(B&\ar[l]^{w_2}A_2)\ar@{=}[r] & (A_2\ar[r]_{w_2} & B)
}
$$
So there is an arrow $(r\colon D\rightarrow C)\in W$ and an arrow $\chi\colon D\rightarrow A_2$
such that the following diagram commutes,
$$
\xymatrix@R=1.5em{
(B\ar@{=}[4,0] & A_1)\ar[l]_{w_1}\ar[r]^f & (A_2\ar[r]^{w_2} & B)\ar@{=}[4,0]
\\
 & C\ar[u]^{u_1}  & &
\\
&D\ar[u]^r \ar[d]_{r} \ar[r]^{\chi} & A_2 \ar@{=}[uu] \ar@{=}[dd]
\\
& C\ar[d]_{u_2}  & &
\\
(B&\ar[l]^{w_2}A_2)\ar@{=}[r] & (A_2\ar[r]_{w_2} & B)\rlap{\quad.}
}
$$

So $fu_1r= \chi=u_2r$. It follows that 
(\ref{vertarrow}) is equivalent to
$$
\xymatrix@R=1.5em{
A_1\ar[r]^{w_1} & B\ar@{=}[dd]
\\
A_1\ar@{=}[u]\ar[d]_{f} 
\\
A_2\ar[r]_{w_2}&B
}
$$
This completes the proof of the first part.

The second part follows from the first part, together with Remark \ref{compconj}:
a vertical arrow in $\CW$ has a horizontal conjoint if and only if its inverse has this
horizontal arrow as companion.

The last part follows from the proofs of the previous two parts.
\end{proof}

It is easy to see that the vertical companions and conjoints generate the category of vertical arrows:
$$
\xymatrix@R=1.5em{
\ar@{=}[dd]B&\ar[l]_{w_1} A_1
\\
&C\ar[u]_{u_1}\ar[d]^{u_2}
\\
B&\ar[l]^{w_2} A_2
}
$$
is the vertical composition of 
$$
\xymatrix@R=1.5em{
\ar@{=}[dd]B&\ar[l]_{w_1} A_1 && \ar@{=}[dd]B&\ar[l]_{w_2u_2} C 
\\
&C\ar[u]_{u_1}\ar@{=}[d]&\mbox{and} & &C\ar@{=}[u]\ar[d]^{u_2}
\\
B&\ar[l]^{w_1u_1} C&& B&\ar[l]^{w_2} A_2
}$$
which are composable since $w_1u_1=w_2u_2$.

The following lemma shows that the companion binding cells and 
their vertical inverses (the conjoint binding cells)
generate all the cells of $\CW$.

\begin{lma}\label{dblcellfac}
Each double cell in $\CW$ can be written as a pasting diagram of companion binding 
cells as in (\ref{bindingcells}) and their
vertical inverses.
\end{lma}

\begin{proof}
Let 
\begin{equation}\label{generic}
\xymatrix@R=1.5em{
(B\ar@{=}[dd] & \ar[l]_{w_1}A_1)\ar[r]^{f_1}&(A_1'\ar[r]^{w_1'} & B')\ar@{=}[dd]
\\
&C\ar[u]_{u_1}\ar[r]^{\xi}\ar[d]^{u_2}&C'\ar[u]_{u_1'}\ar[d]^{u_2'} &
\\
(B & A_2)\ar[l]^{w_2}\ar[r]_{f_2}&( A_2'\ar[r]_{w_2'}& B')
}
\end{equation}
be a representative of a double cell in $\CW$.
This cell can be written as a pasting of the following 
array of double cells by first composing the double cells in each row
horizontally and then composing the resulting double cells vertically.
$$
\xymatrix@R=1.5em{
(B\ar@{=}[dd] & \ar[l]_{w_1}A_1)\ar@{=}[r]&(A_1\ar[r]^{w_1} & B)\ar@{=}[dd]
	&(B\ar@{=}[dd] & \ar[l]_{w_1}A_1)\ar[r]^{f_1}&(A_1'\ar[r]^{w_1'} & B')\ar@{=}[dd]
\\
&C\ar[u]_{u_1}\ar[r]^{u_1}\ar@{=}[d]&A_1\ar@{=}[u]\ar@{=}[d]&
	&&A_1\ar@{=}[u]\ar@{=}[d]\ar[r]^{f_1}&A_1'\ar@{=}[u]\ar@{=}[d] &
\\
(B & C)\ar[l]^{w_1u_1}\ar[r]_{u_1}& (A_1\ar[r]_{w_1}& B)
	& (B & A_1)\ar[l]^{w_1}\ar[r]_{f_1}& (A_1'\ar[r]_{w_1'}& B')
}
$$
$$\xymatrix@R=1.5em{
(B\ar@{=}[dd] & \ar[l]_{w_1u_1} C)\ar[r]^{\xi}&(C'\ar[r]^{w_1'u_1'} & B')\ar@{=}[dd]
	& (B'\ar@{=}[dd] & \ar[l]_{w_1'u_1'}C')\ar[r]^{u_1'}&(A_1'\ar[r]^{w_1'} & B')\ar@{=}[dd]
\\
&C\ar@{=}[u]\ar[r]^{\xi}\ar@{=}[d]&C'\ar@{=}[u]\ar@{=}[d]&
	&&C'\ar@{=}[u]\ar@{=}[r]\ar@{=}[d]&C'\ar[u]_{u_1'}\ar@{=}[d] &
\\
(B & C)\ar[l]^{w_1u_1}\ar[r]_{\xi}& (C'\ar[r]_{w_1'u_1'}& B')
	& (B' & C')\ar[l]^{w_1'u_1'}\ar@{=}[r]& (C'\ar[r]_{w_1'u_1'}& B')
}
$$
$$\xymatrix@R=1.5em{
(B\ar@{=}[dd] & \ar[l]_{w_1u_1}C)\ar[r]^{\xi}&(C'\ar[r]^{w_2'u_2'} & B')\ar@{=}[dd]
	& (B'\ar@{=}[dd] & \ar[l]_{w_2'u_2'}C')\ar@{=}[r]&(C'\ar[r]^{w_2'u_2'} & B')\ar@{=}[dd]
\\
&C\ar@{=}[u]\ar[r]^{\xi}\ar@{=}[d]&C'\ar@{=}[u]\ar@{=}[d] &
	&&C'\ar@{=}[u]\ar@{=}[r]\ar@{=}[d]&C'\ar@{=}[u]\ar[d]^{u_2'} &
\\
(B & C)\ar[l]^{w_2u_2}\ar[r]_{\xi}& (C'\ar[r]_{w_2'u_2'}& B')
	& (B' & C')\ar[l]^{w_2'u_2'}\ar[r]_{u_2'}& (A_2'\ar[r]_{w_2'}& B')
}
$$
$$\xymatrix@R=1.5em{
(B\ar@{=}[dd] & \ar[l]_{w_2u_2}C)\ar[r]^{u_2}&(A_2\ar[r]^{w_2} & B)\ar@{=}[dd]
	& (B\ar@{=}[dd] & \ar[l]_{w_2}A_2)\ar[r]^{f_2}&(A_2'\ar[r]^{w_2'} & B')\ar@{=}[dd]
\\
&C\ar@{=}[u]\ar[r]^{u_2}\ar[d]^{u_2}&A_2\ar@{=}[u]\ar@{=}[d] &
	&& A_2\ar@{=}[u]\ar[r]^{f_2}\ar@{=}[d]&A_2'\ar@{=}[u]\ar@{=}[d] &
\\
(B & A_2)\ar[l]^{w_2} \ar@{=}[r]& (A_2\ar[r]_{w_2}& B)
	& (B & A_2)\ar[l]^{w_2} \ar[r]_{f_2}& (A_2'\ar[r]_{w_2'}& B')\rlap{\,.}
}
$$ 
Each one of these cells is either a horizontal or vertical identity cell or a companion binding cell or the 
inverse of a companion binding cell.
\end{proof}

\subsection{$W$-friendly functors and transformations}
As we have seen in 
Lemma \ref{compchar} and Lemma \ref{dblcellfac},  the strict inclusion functor $\calJ_{W}\colon H\bfC\rightarrow \CW$
adds companions and conjoints to $H\bfC$.
We want to make precise in what sense it does so freely.
The goal of this section is to introduce a notion of $W$-friendly double functor out of $H\bfC$ and a notion of
$W$-friendly horizontal transformation between these $W$-friendly double functors,
such that $\calJ_W$ the universal $W$-friendly double functor out of $H\bfC$
in the sense that we have an equivalence of categories,
$$\WGDbl_{\st,h}(\CW,\bbD)\simeq\WGDbl_{\st,h,W}(H\bfC,\bbD),$$ 
where $\WGDbl_{\st,h,W}(H\bfC,\bbD)$
is the category of $W$-friendly functors and $W$-friendly horizontal transformations.

The first question is for which arrows it adds the companions and conjoints.
It is tempting to think that it does this for the horizontal arrows in the image of $W$.
However, those would be the arrows of the form
$$\xymatrix@1{(1_A)\ar[r]^w&(1_B)}$$
and the arrows that obtain companions and conjoints are of the form
$$
\xymatrix@1{(uw)\ar[r]^w&(u)\rlap{\,.}}
$$
In particular, there are companions and conjoints for arrows
of the form 
\begin{equation}\label{w}
\xymatrix@1{(w)\ar[r]^w&(1_B)\rlap{\,,}}
\end{equation}
but not for
$$\xymatrix@1{(1_A)\ar[r]^w&(1_B)}$$
unless $A=B$ and $w=1_B$.
The arrow in (\ref{w}) can be factored as
$$
\xymatrix@1{(B&A\ar[l]_w)\ar[r]^{1_A}&(A\ar[r]^{1_A}&A)\ar[r]^w &(B\ar[r]^{1_B}&B)}
$$
We will denote the first arrow in this factorization by $\varphi_w\colon (w)\rightarrow(1_A)$.
Note that this is a horizontal isomorphism in $\CW$.
Our first step will be to show that the elements of the family $\{\varphi_w\}$ for $w\in W$ form the components of an invertible 
natural transformation.
To express this property we need the following comma category derived from $W$.

\begin{dfn}\label{nabla}
{\rm Let $\bfC$ be a category with a class $W$ of arrows.
We define $\nabla W$ to be the category with arrows in $W$ as objects and an 
arrow from $\xymatrix@1{A\ar[r]^w&B}$
to $\xymatrix@1{A'\ar[r]^{w'}&B}$ is given by a commutative triangle of arrows in $\bfC$:
$$
\xymatrix{
A\ar[dr]_w\ar[rr]^v&&A'\ar[dl]^{w'}
\\
&B
}
$$
We denote this arrow by $(v,w')\colon w\rightarrow w'$.
There is a functor $D_0\colon \nabla W\rightarrow \bfC$ defined by $D_0(w)=d_0(w)$ (the domain)
on objects and $D_0(v,w')=v$.}
\end{dfn}

\begin{rmk}
 {\rm When $W$ satisfies the 3 for 2 property, the commutativity of the triangle in Definition
\ref{nabla} implies that $v\in W$. Furthermore, by condition {\bf CF1}, $W$ forms a subcategory of $\bfC$ 
which contains all objects of $\bfC$. So we can view $\nabla W$ as the comma category $(1_W\downarrow L_W)$,
where $L_W\colon d(C_0)\rightarrow W$ is the inclusion functor of the discrete category on the objects of $\bfC$
into the category $W$. In this notation, $D_0$ is just the first projection functor 
$(1_W\downarrow L_W)\rightarrow W$ followed by the inclusion into $\bfC$.}
\end{rmk}

Let $\Phi\colon\nabla W\rightarrow \Comp(\CW)$
be the functor that sends an arrow
$$
\xymatrix@R=1.5em{
A\ar[dr]_{wu}\ar[rr]^u&&A'\ar[dl]^{w}
\\
&B
}
$$
to the horizontal companion arrow $\xymatrix@1{(B&A\ar[l]_{wu})\ar[r]^u&(A'\ar[r]^{w}&B)}$
with the vertical companion and binding cells in $\CW$ as indicated in the proof of Lemma \ref{compchar}.
Now we can view $\varphi$ as a natural transformation in the following diagram
$$
\xymatrix{
\nabla W\ar[d]_{D_0}\ar[r]^-\Phi\ar@{}[dr]|{\stackrel{\varphi}{\Leftarrow}} & \Comp(\CW)\ar[d]^{h_-}
\\
\bfC\ar[r]_{h\JC}&h\CW\rlap{\,.}
}
$$
This is almost enough to describe $\JC$ as a $W$-friendly functor.
The problem is that the information given thus far is not enough to describe 
$W$-friendly transformations between  $W$-friendly functors.
Such transformations need to have components that interact well with the companion binding cells.
In order to ensure this, $\Phi$ needs to be viewed as a double functor $V\nabla W\rightarrow \bbComp(\CW)$.
This leads us to the following definition.

\begin{dfn}
{\rm A {\em $W$-friendly structure} $(\Gamma,\gamma)$ for a functor $G\colon H\bfC\rightarrow \bbD$ consists of 
a functor $\Gamma\colon V\nabla W\rightarrow \bbComp(\bbD)$ together with 
an invertible natural transformation $\gamma$,
$$
\xymatrix{
\nabla W\ar[d]_{D_0}\ar[r]^-{v\Gamma}\ar@{}[dr]|{\stackrel{\,\,\sim}{\Leftarrow}}_(.57)\gamma 
  & v\bbComp(\bbD)\ar[d]^{h_-}
\\
\bfC\ar[r]_{hG} & h\bbD,}
$$
where $h_-$ is the identity on objects and takes the horizontal component of each arrow (companion pair).
We will also refer to $G$ (and to the pair $(G,\Gamma)$) as a {\em $W$-friendly functor}.}
\end{dfn}

\begin{prop}\label{canonicalfriendly}
There is a canonical $W$-friendly structure $$(\Phi\colon V\nabla W\rightarrow\bbComp(\CW), \varphi)$$ 
for the functor $\calJ_{\bfC}\colon H\bfC\rightarrow\CW$.
\end{prop}

\begin{proof}
On objects, $\Phi(w)=(w)$. On horizontal arrows, $\Phi$ is the identity.
On vertical arrows,
$\Phi((u,w))=(h_{(u,w)},v_{(u,w)},\psi_{(u,w)},\chi_{(u,w)})$,
where 
\begin{eqnarray*}
h_{(u,w)}&=&\xymatrix@1{(&)\ar[l]_{wu}\ar[r]^u&(\ar[r]^w&)}\\
v_{(u,w)}&=&\xymatrix@R=1.4em{(\ar[r]^{wu}\ar@{=}[d] &)\ar@{=}[dd]\\\ar[d]_u\\(\ar[r]_w&)}
\end{eqnarray*}
and the binding cells are
$$
\xymatrix@R=1.4em{
(\ar@{=}[dd]&\ar[l]_{wu})\ar@{=}[r] & (\ar[r]^{wu}& \ar@{=}[dd])
      && (\ar@{=}[dd]&\ar[l]_{wu})\ar[r]^u&(\ar[r]^w&\ar@{=}[dd])
\\
&\ar@{=}[u]\ar@{=}[d]\ar@{=}[r] & \ar@{=}[u]\ar[d]^u &&\mbox{and}
      & & \ar@{=}[u]\ar[r]^u\ar[d]^u & \ar@{=}[u]\ar@{=}[d]
\\
(&\ar[l]^{wu})\ar[r]_u&(\ar[r]_w&) && (&\ar[l]^w)\ar@{=}[r] & (\ar[r]_w&)\rlap{\,.}
}
$$
Furthermore, the invertible natural transformation $\varphi$
in 
$$
\xymatrix@R=1.9em{
\nabla W\ar[d]_{D_0}\ar[r]^-{v\Phi} \ar@{}[dr]|{\stackrel{\,\,\sim}{\Leftarrow}}_(.57)\varphi 
  & v\bbComp(\CW)\ar[d]^{h_-}
\\
\bfC\ar[r]_{hF} & h\CW}
$$
has components $\varphi_w$,
$$
\xymatrix{(&\ar[l]_w)\ar@{=}[r] &(\ar@{=}[r]&).}
$$
\end{proof}

Let $L\colon\CW\rightarrow \bbD$ be a functor of double categories.
Composition with $\calJ_{\bfC}\colon H\bfC\rightarrow \CW$ gives rise to a
functor $L\calJ_{\bfC}\colon H\bfC\rightarrow \bbD$ with the $W$-friendly structure
$(\bbComp(L)\circ\Phi,\lambda)$. Here, 
$\bbComp(L)\circ\Phi\colon V\nabla W\rightarrow \bbComp(\bbD)$
sends $w$ to $L(w)$ and a vertical arrow $(u,v)\colon\xymatrix@1{vu\ar[r]|\bullet^u&v}$
to 
$$\left(L\left(\xymatrix{&\ar[l]_{vu}\ar[r]^u & \ar[r]^v&}\right),
L\left(\raisebox{1.95em}{$\xymatrix@R=1.4em{\ar@{=}[dd]&\ar[l]_{vu}\\&\ar@{=}[u]\ar[d]^u \\ 
&\ar[l]^{v}}$}\right),
L\left(\raisebox{1.95em}{$\xymatrix@R=1.4em{\ar@{=}[dd]& \ar[l]_{vu}\ar@{=}[r]& \ar[r]^{vu}&\ar@{=}[dd]\\
& \ar@{=}[u]\ar@{=}[d]
		&\ar@{=}[u]\ar[d]^u \\ &\ar[l]^{vu}\ar[r]_u&\ar[r]_v&}$}\right),\right.
		$$
		$$\left.
L\left(\raisebox{1.95em}{$\xymatrix@R=1.4em{\ar@{=}[dd]& \ar[l]_{vu}\ar[r]^u
		& \ar[r]^{v}&\ar@{=}[dd]\\& \ar@{=}[u]\ar[d]^u
		&\ar@{=}[u]\ar@{=}[d] \\ &\ar[l]^{v}\ar@{=}[r]&\ar[r]_v&}$}\right)\right).$$
Furthermore, the natural transformation $\lambda$ has components 
$\lambda_w=L\left(\xymatrix@1@C=1.5em{&\ar[l]_w\ar@{=}[r]&\ar@{=}[r] &}\right)$.

The appropiate horizontal transformations between $W$-friendly functors 
are described in the following definition.

\begin{dfn}
{\rm For $W$-friendly functors
$$(G,\Gamma,\gamma), (L,\Lambda,\lambda)\colon 
H\bfC\rightrightarrows\bbD, V\nabla W\rightrightarrows \bbComp(\bbD),$$
a {\em $W$-friendly horizontal transformation} is a pair 
$(a\colon G\Rightarrow L,\alpha\colon\Gamma\Rightarrow\Lambda)$
of horizontal transformations
such that the following square commutes: 
\begin{equation}\label{W-friendly}
\xymatrix@R=1.7em{
h_-\circ v\Gamma\ar[d]_\gamma\ar[r]^{h_-\circ v\alpha}& h_-\circ v\Lambda\ar[d]^{v\lambda}
\\
hG\circ D_0\ar[r]_{ha\circ D_0} & hL\circ D_0.
}
\end{equation}}
\end{dfn}

\subsection{The horizontal universal property of $\CW$} 
We can now state and prove the universal property of $\calJ_{W}\colon H\bfC\rightarrow \CW$ in terms of $W$-friendly double functors 
and $W$-friendly horizontal transformations.
 
 \begin{thm}\label{main}
 Composition with $\calJ_{W}\colon H\bfC\rightarrow \CW$ induces an equivalence of 
 categories $$\WGDbl_{\st,h}(\CW,\bbD)\simeq\WGDbl_{\st,h,W}(H\bfC,\bbD),$$ where $\WGDbl_{\st,h,W}(H\bfC,\bbD)$
 is the category of $W$-friendly functors and $W$-friendly horizontal transformations.
 \end{thm}
 
 \begin{proof}
Composition with $\calJ_W$ 
sends a functor $L\colon\CW\rightarrow \bbD$ of weakly globular double categories
to the triple $(L\calJ_W, L\circ\Phi,L\varphi)$, i.e., the functor $L\calJ_W$ with $W$-friendly 
structure $(L\Phi,L\varphi)$, as described above.

 To show that composition with $\calJ_W$ is essentially surjective on objects,
 we show how to lift a functor $G\colon H\bfC\rightarrow \bbD$ with 
 a $W$-friendly structure $(\Gamma,\gamma)$ to a functor $\tilde{G}\colon\CW\rightarrow \bbD$.
 
 Define $\tilde G$ on objects by $\tilde{G}(\xymatrix@1{A\ar[r]^w&B})=\Gamma(w)$.
 On vertical arrows, $\tilde G$ is defined by
 $$\tilde{G}\left(\raisebox{2.65em}{$\xymatrix@R=1.5em{(\ar[r]^{w_1}&)\ar@{=}[dd]
\\
\ar[u]^{u_1}\ar[d]_{u_2} &
\\
(\ar[r]_{w_2}&)}$}\right)=v_{\Gamma(u_2,w_2)}\cdot(v_{\Gamma(u_1,w_1)})^{-1}
 $$
 (where the notation is as in the proof of Proposition \ref{canonicalfriendly})
 and on horizontal arrows, $\tilde G$ is defined by
 $\tilde{G}\left(\xymatrix@1{(&\ar[l]_w)\ar[r]^f&(\ar[r]^v&)}\right)=\gamma_v^{-1}G(f)\gamma_w.$
 To define $\tilde{G}$ on double cells, we use the factorization of a generic double
 cell in $\CW$ given in the proof of  Lemma \ref{dblcellfac}. 
The result of applying $\tilde{G}$ to (\ref{generic}) is given in Figure \ref{imagedblcell}.
\begin{figure}[ht]
 $$\xymatrix@R=1.4em{
 \Gamma(w_1)\ar[d]|\bullet^{v^{-1}_{\Gamma(u_1,w_1)}}\ar@{=}[rrr]\ar@{}[drrr]|{\chi^{-1}_{\Gamma(u_1,w_1)}} 
		&&&\Gamma(w_1)\ar[r]^{\gamma_{w_1}} \ar@{=}[d]& GA_1 \ar[r]^{Gf_1} 
		& GA_1'\ar[r]^{\gamma_{w_1'}^{-1}} 
		& \Gamma(w_1') \ar@{=}[d]
\\
\Gamma(w_1u_1) \ar[rrr]_{h_{\Gamma(u_1,w_1)}}  \ar@{=}[d] 
		&&& \Gamma(w_1) \ar[r]^{\gamma_{w_1}} \ar@{=}[d]
		& GA_1 \ar[r]^{Gf_1} & GA_1'\ar[r]^{\gamma_{w_1'}^{-1}} 
		& \Gamma(w_1') \ar@{=}[d]
\\
\Gamma(w_1u_1) \ar[r]^-{\gamma_{w_1u_1}} \ar@{=}[d] & GC \ar@{=}[d] \ar[r]^{Gu_1} 
		&GA_1\ar[r]^{\gamma_{w_1}^{-1}}
		& \Gamma(w_1) \ar[r]^{\gamma_{w_1}} &  GA_1 \ar[r]^{Gf_1} 
		& GA_1' \ar@{=}[d] \ar[r]^{\gamma_{w_1'}^{-1}} 
		& \Gamma(w_1') \ar@{=}[d]
\\
\Gamma(w_1u_1) \ar[r]^-{\gamma_{w_1u_1}} \ar@{=}[5,0] & GC \ar@{=}[5,0] \ar[r]^{G\xi} 
		& GC' \ar@{=}[dd] \ar[r]^-{\gamma^{-1}_{w_1'u_1'}}
		& \Gamma(w_1'u_1') \ar@{=}[d] \ar[r]^-{\gamma_{w_1'u_1'}} 
		& GC' \ar[r]^{Gu_1'} & GA_1' \ar[r]^{\gamma_{w_1'}^{-1}} 
		& \Gamma(w_1') \ar@{=}[d]
\\
&&& \Gamma(w_1'u_1')\ar@{=}[d] \ar[rrr]^{h_{\Gamma(u_1',w_1')}}\ar@{}[drrr]|{\psi^{-1}_{\Gamma(w_1')}}
		&&& \Gamma(w_1') \ar[d]|\bullet^{v^{-1}_{\Gamma(w_1')}}
\\
&& GC'\ar@{=}[d] \ar[r]^-{\gamma_{w_1'u_1'}^{-1}} &  \Gamma(w_1'u_1') \ar@{=}[d] \ar@{=}[rrr] 
		&&& \Gamma(w_1'u_1') \ar@{=}[d]
\\
&& GC'\ar@{=}[dd] \ar[r]^-{\gamma_{w_2'u_2'}^{-1}} 
		&  \Gamma(w_2'u_2') \ar@{=}[d] \ar@{=}[rrr] \ar@{}[drrr]|{\psi_{\Gamma(u_2',w_2')}} 
		&&& \Gamma(w_2'u_2') \ar[d]|\bullet^{v_{\Gamma(u_2',w_2')}}
\\
&&& \Gamma(w_2'u_2')\ar@{=}[d] \ar[rrr]_{h_{\Gamma(u_2',w_2')}} &&& \Gamma(w_2')\ar@{=}[d]
\\
\Gamma(w_2u_2) \ar[r]^-{\gamma_{w_2u_2}} \ar@{=}[d]& GC \ar[r]^{G\xi} \ar@{=}[d] 
		& GC'\ar@{=}[d] \ar[r]^-{\gamma_{w_2'u_2'}^{-1}} 
		& \Gamma(w_2'u_2') \ar[r]^-{\gamma_{w_2'u_2'}} & GC' \ar[r]^{Gu_2'} 
		&GA_2'\ar[r]^{\gamma_{w_2'}^{-1}} \ar@{=}[d] &\Gamma(w_2') \ar@{=}[d]
\\
\Gamma(w_2u_2) \ar[r]^-{\gamma_{w_2u_2}} \ar[d]|\bullet^{v_{\Gamma(u_2,w_2)}} 
			  \ar@{}[drrr]|{\chi_{\Gamma(u_2,w_2)}} 
		& GC \ar[r]^{Gu_2} & GA_2 \ar[r]^{\gamma_{w_2}^{-1}} 
		& \Gamma(w_2) \ar@{=}[d] \ar[r]^{\gamma_{w_2}} 
		& GA_2\ar[r]^{Gf_2} & GA_2' \ar[r]^{\gamma_{w_2'}^{-1}} 
		& \Gamma(w_2')\ar@{=}[d]
\\
\Gamma(w_2)\ar@{=}[rrr] &&& \Gamma(w_2) \ar[r]^{\gamma_{w_2}} & GA_2\ar[r]^{Gf_2} 
		& GA_2' \ar[r]^{\gamma_{w_2'}^{-1}} & \Gamma(w_2')
 }$$
 \caption{The image of a generic double cell in $\CW$ under $\tilde{G}$.}\label{imagedblcell}
\end{figure}
 It is not hard to see that there is an invertible horizontal transformation
 $\bar\gamma\colon \tilde{G}\circ \JC\Rightarrow G$ with components $\bar{\gamma}_{A}=\gamma_{1_A}$.
 Furthermore, it is straightforward to check that $\tilde{G}\Phi=\Gamma$, and $(\bar\gamma,\mbox{id})$ 
is an invertible 
$W$-friendly transformation from $(\tilde{G}\JC,\tilde{G}\Phi,\tilde{G}\varphi)$ to $(G,\Gamma,\gamma)$.
 
We now want to show that composition with $\JC$ is fully faithful on arrows, i.e., horizontal transformations.
So we will show that horizontal transformations $$b\colon G\Rightarrow L\colon \CW\rightrightarrows\bbD$$
are in one-to-one correspondence with
$W$-friendly transformations $$(b\JC,\beta)\colon (G\JC,G\Phi,G\varphi)\Rightarrow(L\JC,L\Phi,L\varphi)$$ 
 where $\beta\colon G\Phi\Rightarrow L\Phi$ is defined by $\beta_w=b_w$.
 The condition that the square in (\ref{W-friendly}) commutes for $(b\JC,\beta)$
 is equivalent to the following square of horizontal arrows commuting for each $w\in W$:
 $$
 \xymatrix@C=3.5em{
 G(w)\ar[d]_{G(\stackrel{v}{\longleftarrow}\Longrightarrow\Longrightarrow)}\ar[r]^{b_v}
	&L(v)\ar[d]^{G(\stackrel{v}{\longleftarrow}\Longrightarrow\Longrightarrow)}
 \\
 G(\mbox{id}_A)\ar[r]_{b_{\mbox{\scriptsize id}_A}} &L(\mbox{id}_A)
 }
 $$
 and the commutativity of this diagram follows immediately from the vertical naturality of $b$. 
 
Conversely, given a $W$-friendly horizontal transformation 
$(a\colon G\JC\Rightarrow L\calJ_W,\alpha\colon G\Phi\Rightarrow L\Phi)$,
we get $\tilde{a}\colon G\Rightarrow L$ with components $\tilde{a}_w=\alpha_w$ and
$$
\tilde{a}\left(\raisebox{3.2em}{$\xymatrix{(\ar[r]^{wu}&)\ar@{=}[dd]
\\
\ar@{=}[u]\ar[d]_{u} &
\\
(\ar[r]_{w}&)}$}\right) = \raisebox{4.2em}{$\xymatrix@C=3.9em@R=1.25em{G(wu)\ar[rr]|{h_{\alpha(wu)}} \ar@{=}[d] 
		&& L(wu) \ar@{=}[d]\ar@{=}[r] \ar@{}[dr]|{L(\psi_{\Phi(u,v)})} & L(wu)\ar[d]|\bullet^{L(v_{\Phi(u,w)})}
\\
G(wu)  \ar@{=}[d]\ar@{=}[r] & G(wu)\ar@{=}[d]  \ar[r]|{h_{\alpha(wu)}} & L(wu)\ar[r]|{L(h_{\Phi(u,w)})}
	& L(w)\ar@{=}[d]
\\
G(wu) \ar@{=}[r] \ar[d]|\bullet_{G(v_{\Phi(u,v)})} 
	& G(wu) \ar[d]|\bullet_{G(v_{\Phi(u,w)})} \ar@{}[dr]|{G(\chi_{\Phi(u,w)})} \ar[r]|{G(h_{\Phi(u,w)})} 
	& G(w)\ar@{=}[d]\ar[r]|{h_{\alpha(w)}}&L(w)\ar@{=}[d]\\
G(w)\ar@{=}[r] & G(w)\ar@{=}[r] & G(w)\ar[r]|{h_{\alpha(w)}} &L(w)}$}
$$
Note that it is sufficient to define $\tilde{a}$ on these particular vertical arrows, since all others are 
generated by these and their inverses.
 \end{proof}

\begin{rmks}{\rm
\begin{enumerate}
 \item  One might be tempted to think that if we take $W=\bfC_1$,  the class of all arrows in $\bfC$,
then every horizontal arrow in $\bfC\{W\}$
will have  a companion and a conjoint. Unfortunately that is not the case: only arrows
of the form $\xymatrix{(&\ar[l]_{uv})\ar[r]^u &(\ar[r]^v&)}$
get a companion and a conjoint. However, if we restrict ourselves to such horizontal arrows (for any chosen class $W$),
take all vertical arrows and the full sub double category of double cells, 
we do get a double category in which every horizontal arrow has a companion and a conjoint.
That double category is the transpose of what Shulman \cite{Shul} has called a fibrant double category and
what was called a gregarious double category in \cite{DPP-spans2}.
\item
The vertical and horizontal property together determine $\CW$ up to both horizontal 
and vertical equivalence of weakly globular double categories,
but note that for the vertical equivalence the arrows are pseudo-functors, where for
the horizontal equivalence the arrows are strict functors.
\item The construction of $\CW$  given in Section \ref{construction}
can be extended to 2-categories, where the universal properties of $\calC\{W\}$
can be formulated using the horizontal double category $H\calC$. It can even be extended 
to arbitrary weakly globular double categories. 
We will present these results in a sequel to this paper,
as they involve a lot of technical details.
\end{enumerate}}

\end{rmks}

\section{Conclusions}\label{conc}
In this paper we studied what the horizontal and vertical universal properties of a weakly globular double category of fractions
should be and we have seen how this determines this weakly globular double category both up to horizontal and vertical equivalence.
Both companions and precompanions play an important role in the description of its universal properties.
We have seen that the notion of precompanion generalizes both the notions of companion and of horizontal isomorphism in a double category.

For double categories with strict functors, the hom-object $\mbox{\bf DblCat}(\bbC,\bbD)$ can be given the structure of a double category
where the horizontal arrows are horizontal transformations and the vertical arrows are vertical transformations, and modifications 
are given as families of double cells that are functorial and natural in appropriate ways (the precise definition can be found in \cite{EE}).
This leads to the question whether the universal property of $\CW$ as a coinverter can be expressed in terms of double categories rather than 2-categories. This is indeed possible, but requires a further study of the modifications for weakly globular double categories.
We will discuss this in a separate paper.


\begin{thebibliography}{99}
\bibitem{Adem} A.~Adem,  J.~Leida, Y.~Ruan, \emph{Orbifolds and Stringy Topology}, 
Cambridge Tracts in Mathematics 171, Cambridge University Press (2007).
\bibitem{ben} 
J.~B\'enabou, Introduction to bicategories, in {\it Reports of the Midwest Category Seminar}, 
Springer L.N.M. 47 (1967), pp.~1--77.
\bibitem{bpa} 
David Blanc, Simona Paoli, $n$-Fold groupoids, $n$-Types and $n$-Track
categories (2012) arXiv: 1204.5101.
\bibitem{bp} 
David Blanc, Simona Paoli, Two-track categories, \emph{Journal of
K-theory}, {\bf 8} no.1 (2011), pp.~59--106.
\bibitem{BM}
Ronald Brown, Ghafar Mosa, Double categories, 2-categories, thin structures and connections, 
{\it Theory and Applications
of Categories} {\bf 5} (1999), pp.~163--175.
\bibitem{BS} Ronald Brown,  Christopher B.~Spencer,  Double groupoids and crossed modules, 
{\em Cahiers Topologie G\'eom. Diff\'erentielle} {\bf 17} (1976), no. 4, pp.~343--362.
\bibitem{DPP-spans2} 
Robert Dawson, Robert Par\'e, Dorette Pronk, The span construction, 
{\em Theory and Applications of Categories} {\bf 24} (2010), no. 13, pp.~302--377.
\bibitem{hammock} W.G.~Dwyer, D.M.~Kan, Calculating simplicial localizations, {\em Journal of Pure and Applied Algebra} {\bf 18} (1980), pp.~17--35.
\bibitem{Ehr}
Charles Ehresmann,  Cat\'egories doubles et cat\'egories structur\'ees, {\em C. R. Acad. Sci. Paris} {\bf 256} (1963) pp.~1198--1201. 
\bibitem{EE} A.~Ehresmann, C.~Ehresmann, Multiple functors, III. The Cartesian closed category $\mbox{Cat}_n$, {\em Cahiers Topologie G\'eom. Diff\'erentielle} {\bf 19} (1978), pp.~387--443.
\bibitem{FPP} Thomas M.~Fiore, Simona Paoli, Dorette Pronk,  Model structures on the category of small double categories,
{\em Algebr. Geom. Topol.} {\bf 8} (2008), no. 4, pp.~1855--1959.
\bibitem{GZ} 
P.~Gabriel, M.~Zisman, {\it Calculus of Fractions and Homotopy Theory}, Springer-Verlag, New York, 1967.
\bibitem{Gr} Marco Grandis, Categorically algebraic foundations for homotopical algebra, 
{\em Appl. Categ. Structures} {\bf 5} (1997), no. 4, pp.~363--413. 
\bibitem{GP} 
Marco Grandis, Robert Par\'e, Limits in double categories, {\em Cahiers Topologie G\'eom. 
Diff\'erentielle Cat\'eg.} {\bf 40} (1999), no. 3, pp.~162--220.
\bibitem{GP-companions}Marco Grandis, Robert Par\'e, Adjoint for double categories, {\em Cahiers Topologie G\'eom. 
Diff\'erentielle Cat\'eg.} {\bf 45} (2004), pp.~193--240.
\bibitem{lk}
G.M.~Kelly, Stephen Lack, Monoidal functors generated by adjunctions, with
applications to transport of structure, {\em Galois theory, Hopf algebras, and semiabelian categories}, 
pp.~319--340, {\em Fields Inst. Commun.} {\bf 43}, Amer. Math. Soc., Providence, RI, 2004.
\bibitem{KLW} G.M.~Kelly, Stephen Lack, R.F.C.~Walters, Coinverters and categories of fractions for categories with structure, {\it Applied Categorical Structures} {\bf 1} (1993), pp.~95-102.
\bibitem{lp} 
Stephen Lack, Simona Paoli, 2-nerves for bicategories, {\it $K$-Theory} {\bf 38} 
(2008), no. 2, pp.~153--175. 
\bibitem{Lei1} Tom Leinster, 
A survey of definitions of $n$-category.  
{\em Theory Appl. Categ.} {\bf 10} (2002), pp.~1--70.
\bibitem{Lei2} Tom Leinster, {\em Higher operads, higher categories.} 
London Mathematical Society Lecture Note Series, 298, Cambridge University Press, Cambridge, 2004.
\bibitem{MM} I.~Moerdijk, J.~Mrcun, {\em Introduction to Foliations and Lie Groupoids}, Cambridge Studies in Advanced Mathematics, 91, Cambridge University Press, Cambridge, 2003.
\bibitem{MP} I.~Moerdijk, D.A.~Pronk, Orbifolds, groupoids and sheaves, {\it $K$-theory} 12 (1997), p.~3--21.
\bibitem{Noohi} Ettore Aldrovandi, Behrang Noohi,
Butterflies. I. Morphisms of 2-group stacks. {\em Adv. Math.} {\bf 221} (2009), no. 3, pp.~687--773.
\bibitem{PP1} Simona Paoli, Dorette Pronk, A double categorical model of weak 2-categories, {\em Theory and Applications of Categories} {\bf 28} (2013), pp.~933--980.
\bibitem{sequel} Simona Paoli, Dorette Pronk, Weakly globular double categories of fractions, in preparation.
\bibitem{Pr-comp} D.A.~Pronk, Etendues and stacks as bicategories of fractions, \emph{Compositio Math.}, 
{\bf 102} (1996), pp.~243--303.
\bibitem{P-thesis} D.A. Pronk, {\em Groupoid Representations for Sheaves on Orbifolds}, Ph.D. thesis, Utrecht 1995.
\bibitem{PW} Dorette A. Pronk, Michael A. Warren, Bicategorical fibration structures and stacks, preprint.
\bibitem{Quillen} D.~Quillen, {\em Homotopical algebra}, {\em Lecture Notes in
  Mathematics} {\bf 43}, Springer Verlag, New York, 1967.
\bibitem{Simpson} Carlos Simpson, {\em Homotopy theory of higher categories.} {\em New Mathematical 
Monographs}, {\bf 19}. Cambridge University Press, Cambridge, 2012.
\bibitem{Shul}  Michael Shulman, Framed bicategories and monoidal fibrations, 
{\em Theory Appl. Categ.} {\bf 20} (2008), 
no. 18, pp.~650--738.
\bibitem{spen} C.B.~Spencer, An abstract setting for homotopy pushouts and pullbacks,
{\it Cahiers Tp. G\'eo. Diff.} {\bf 18} (1977), pp.~409--430.
\bibitem{tam} Zouhair Tamsamani, Sur des notions de $n$-cat\'{e}gorie et $n$-groupo\"{i}de
non strictes via des ensembles multi-simpliciaux, \emph{$K$-theory}, \textbf{16}
(1999), pp.~51--99.
\end{thebibliography}
\end{document}